\documentclass[reqno,twoside]{amsart}
\usepackage{amscd}
\usepackage{amsfonts}
\usepackage{amsmath}
\usepackage{amssymb}
\usepackage{amsthm}
\usepackage{amsaddr}
\usepackage{fancyhdr}
\usepackage{latexsym}
\usepackage[colorlinks=true, pdfstartview=FitV, linkcolor=blue, citecolor=blue, urlcolor=blue]{hyperref}
\usepackage{enumitem}      
\usepackage{mathtools}            
\usepackage{indentfirst} % for indentation after chapter, section...      
\usepackage{color}
\usepackage{caption}          
\usepackage[normalem]{ulem}%for using \uline
\usepackage{upgreek}
\usepackage{dutchcal}
\usepackage{mathrsfs}
\newcommand{\nn}{\nonumber}
\newcommand{\p}{\partial}

\newcommand{\no}[1]{\left\| #1 \right\|}
\newcommand{\what}{\widehat}
\usepackage{accents}
\renewcommand*{\dot}[1]{%
\accentset{\mbox{\small\bfseries .}}{#1}}
\usepackage{thmtools}
\declaretheoremstyle[headfont=\kpfonts]{normalhead}

\newtheorem{theorem}{Theorem}[section]
\newtheorem{proposition}{Proposition}[section]
\newtheorem{lemma}{Lemma}[section]

\theoremstyle{definition}

\newtheorem{remark}{Remark}[section]

\allowdisplaybreaks
\numberwithin{equation}{section}
\numberwithin{figure}{section}
\usepackage{geometry}
\usepackage[breakable, theorems, skins]{tcolorbox}
\tcbset{enhanced}

\usepackage{tikz}
\usetikzlibrary{arrows.meta}
\usetikzlibrary{decorations.markings}
\tikzset{->-/.style={decoration={
  markings,
  mark=at position #1 with {\arrow{>}}},postaction={decorate}}}
  \tikzset{middlearrow/.style={
        decoration={markings,
            mark= at position 0.55 with {\arrow{#1}} ,
        },
        postaction={decorate}
    }
}
\let\OLDthebibliography\thebibliography
\renewcommand\thebibliography[1]{
  \OLDthebibliography{#1}
  \setlength{\parskip}{0pt}
  \setlength{\itemsep}{0pt plus 0.3ex}
}
\AtBeginDocument{%to remove MR number from references
   \def\MR#1{}
}

\makeatletter
\@namedef{subjclassname@2020}{%
  \textup{2020} Mathematics Subject Classification}
\makeatother

\begin{document}

\title{Low-Regularity Solutions of the Nonlinear Schr\"odinger Equation \\ on the Spatial Quarter-Plane}

\author{\vspace*{-0.25cm}Dionyssios Mantzavinos$^\dagger$ \&  Türker Özsarı$^*$}

\address{
\normalfont $^\dagger$Department of Mathematics, University of Kansas 
\\
\normalfont $^*$Department of Mathematics, Bilkent University
}
 \email{\!mantzavinos@ku.edu \textnormal{(Corresponding author)}, turker.ozsari@bilkent.edu.tr}

\thanks{\textit{Acknowledgements.} DM is thankful to A. Alexandrou Himonas for extensive discussions related to the problem under consideration. DM's research is partially supported by  the U.S. National Science Foundation (NSF-DMS 2206270 and NSF-DMS 2509146) and the Simons Foundation (SFI-MPS-TSM-00013970). TÖ's research is supported by BAGEP 2020 Young Scientist Award. The authors are grateful to the reviewers for their careful assessment of the paper.
\\
\indent
\textit{Disclosure statement.} The authors report that there are no competing interests to declare.
\\
\indent
\textit{Data availability statement.} No data set is associated with this work.
}
\date{April 26, 2024. \textit{Revised}: July 11, 2025}

\begin{abstract}
The Hadamard well-posedness of the nonlinear Schr\"odinger equation with power nonlinearity formulated on the spatial quarter-plane is established in a low-regularity setting with Sobolev initial data and Dirichlet boundary data in appropriate Bourgain-type spaces. As both of the spatial variables are restricted to the half-line, a different approach is needed than the one previously used for the well-posedness of other initial-boundary value problems. In particular, now the solution of the forced linear initial-boundary problem is estimated \textit{directly}, both in Sobolev spaces and in Strichartz-type spaces, i.e. without a linear decomposition that would require estimates for the associated homogeneous and nonhomogeneous initial value problems. In the process of deriving the linear estimates, the function spaces for the boundary data are identified as the intersections of certain modified Bourgain-type spaces that involve spatial half-line Fourier transforms instead of the usual whole-line Fourier transform found in the definition of the standard Bourgain space associated with the one-dimensional initial value problem. The fact that the quarter-plane has a corner at the origin poses an additional challenge, as it requires one to expand the validity of certain Sobolev extension results to the case of a domain with a non-smooth (Lipschitz) and non-compact boundary.
\end{abstract}

\vspace*{-0.5cm}
\maketitle
\markboth
{D. Mantzavinos \& T. Özsarı}
{Low-regularity solutions of the nonlinear Schr\"odinger equation on the spatial quarter-plane}

%{\hypersetup{linkcolor=black}
%\tableofcontents}

\section{Introduction}

The nonlinear Schr\"odinger (NLS) equation is a universal model in mathematical physics, notably in nonlinear optics, water waves, plasmas, and Bose-Einstein condensates. The present work establishes the Hadamard well-posedness of the initial-boundary value problem for the NLS equation on the spatial quarter-plane:
\begin{equation}\label{qnls-ibvp}
\begin{aligned}
&iu_t + u_{x_1x_1} + u_{x_2x_2} = \pm |u|^{\alpha-1} u, \quad (x_1, x_2) \in \mathbb R_+^2, \ t \in (0, T),
\\
&u(x_1, x_2, 0) = u_0(x_1, x_2) \in H^s(\mathbb R_+^2),
\\
&u(x_1, 0, t) = g_0(x_1, t) \in \mathcal X_{1, T}^{0, \frac{2s+1}{4}} \cap \mathcal X_{1, T}^{s, \frac 14},
\\
&u(0, x_2, t) = h_0(x_2, t) \in \mathcal X_{2, T}^{0, \frac{2s+1}{4}} \cap \mathcal X_{2, T}^{s, \frac 14}.
\end{aligned}
\end{equation}
The space $H^s(\mathbb R_+^2)$ for the initial data is the $L^2$-based Sobolev space on the quarter-plane $\mathbb R_+^2 := \mathbb R_+ \times \mathbb R_+$. It is defined as the restriction on $\mathbb R_+^2$ of the Sobolev space
\begin{equation}\label{hsr2-def}
\begin{aligned}
H^s(\mathbb R^2)
= \left\{\varphi \in L^2(\mathbb R^2): \no{\varphi}_{H^s(\mathbb R^2)} := \no{(1+k_1^2+k_2^2)^{\frac s2} \mathcal F\{\varphi\}(k_1, k_2)}_{L^2(\mathbb R^2)} < \infty \right\},
\end{aligned}
\end{equation}
where $\mathcal F\{\varphi\}$ is the two-dimensional Fourier transform  of $\varphi$, i.e.
\begin{equation}
\mathcal F\{\varphi\}(k_1, k_2) := \int_{\mathbb R} \int_{\mathbb R} e^{-ik_1x_1-ik_2x_2} \varphi(x_1, x_2) dx_1 dx_2, \quad k_1, k_2 \in \mathbb R,
\end{equation}
and the corresponding restriction norm is given by 
\begin{equation}
\no{\phi}_{H^s(\mathbb R_+^2)}
:=
\inf\left\{\no{\varphi}_{H^s(\mathbb R^2)}: \varphi|_{\mathbb R_+^2} = \phi \right\}.
\end{equation}
We remark that, throughout this work, in the interest of clarity  the variable(s) with respect to which a norm is taken may not always be suppressed, e.g. as in the case of the $L^2$ norm with respect to $k_1, k_2$ in \eqref{hsr2-def}.
Moreover, for the purpose of local well-posedness studied in this work, the focusing (negative sign) and defocusing (positive sign) NLS equations in \eqref{qnls-ibvp} are treated simultaneously.

For $j=1, 2$, the space $\mathcal X_{j, T}^{\sigma, b}$ for the boundary data emerges naturally through the estimation of the relevant  components of the solution to the forced linear counterpart of problem \eqref{qnls-ibvp}. It is defined as the space of all functions $g(x_j, t) \in L^2(\mathbb R_+\times (0, T))$ such that the following norm is finite:
\begin{equation}\label{xt-def}
\no{g}_{\mathcal X_{j, T}^{\sigma, b}} := \left(\int_{\mathbb R} \left(1+k_j^2\right)^\sigma \no{e^{ik_j^2 t} \, \what g^{x_j}(k_j, t)}_{H_t^b(0, T)}^2 dk_j \right)^{\frac 12},
\end{equation}
where $\what g^{x_j}(k_j, t)$ denotes the \textit{half-line Fourier transform} of $g(x_j ,t)$ with respect to the spatial variable $x_j$, i.e.
\begin{equation}\label{hl-ft-def}
\what g^{x_j}(k_j, t) := \int_{\mathbb R_+} e^{-ik_j x_j} g(x_j, t) dx_j.
\end{equation}
It is interesting to note that the \textit{boundary data space} $\mathcal X_{j, T}^{\sigma, b}$ can be regarded as an analogue, on the semi-infinite strip $\mathbb R_+ \times (0, T)$, of the \textit{Bourgain space} $X_j^{\sigma, b}$. This latter space plays the role of the \textit{solution space} in the analysis of the one-dimensional NLS initial value problem on the line \cite{b1999} (see also \cite{b1993-nls} for the torus) and is defined on the entire spatiotemporal domain $\mathbb R\times\mathbb R$ as the space of all functions $G(x_j, t) \in L^2(\mathbb R^2)$ with finite norm 
\begin{equation}\label{xsb-def}
\no{G}_{X_j^{\sigma, b}} 
:= 
\left(\int_{\mathbb R} \left(1+k_j^2\right)^\sigma \int_{\mathbb R}  \left(1+\left|\tau+k_j^2\right|^2\right)^b \left|\mathcal F\{G\}(k_j, \tau)\right|^2 d\tau dk_j\right)^{\frac 12},
\end{equation}
where $\mathcal F\{G\}$ denotes the spatiotemporal Fourier transform of $G$. Observe that by letting $\tau \mapsto \tau-k_j^2$ we can express this norm in the alternative form
\begin{equation*}
\no{G}_{X_j^{\sigma, b}} 
=
\left(\int_{\mathbb R} \left(1+k_j^2\right)^\sigma \no{e^{ik_j^2 t}  \mathcal F_{x_j}\{G\}(k_j, t)}_{H_t^b(\mathbb R)}^2 dk_j \right)^{\frac 12},
\end{equation*}
which has the same structure with the semi-infinite strip norm \eqref{xt-def}.

The NLS equation in \eqref{qnls-ibvp} is a typical example of a dispersive nonlinear partial differential equation. Such equations describe phenomena in which waves of different frequencies propagate at different speeds; they arise in an abundance of applications and have been studied extensively via a wide variety of theoretical and computational techniques. As for any differential equation, a fundamental question in the study of dispersive nonlinear equations is that of Hadamard well-posedness, namely existence and uniqueness of solution, as well as continuous dependence of this solution on the associated data. Since the 1970s, a remarkable volume of works in the literature has been devoted to the well-posedness (and ill-posedness) of dispersive equations  in the context of their initial value (Cauchy) problem. In that context, celebrated dispersive equations  such as the NLS and Korteweg-de Vries (KdV) equations have been shown to be well-posed via harmonic analysis techniques and the contraction mapping principle. Importantly, new ideas and tools have been developed, such as Strichartz estimates \cite{s1977} and Bourgain spaces~\cite{b1999,b1993-nls,b1993-kdv} (see \eqref{xsb-def} above), leading to optimal well-posedness results in terms of the assumptions on data regularity. A key feature of the initial value problem is that the solution to the forced linear counterpart of the dispersive equation under study can be expressed in terms of the Fourier transform (or Fourier series, in the case of the periodic problem) of the relevant data. This allows one to move the problem from the physical spatial domain ($\mathbb R^n$, or $\mathbb T^n$ in the case of the periodic problem) to the Fourier frequency domain, where a plethora of powerful harmonic analysis techniques can be applied.

In contrast, initial-\textit{boundary} value problems are formulated on subsets of $\mathbb R^n$ that involve a boundary (e.g. the half-line $(0, \infty)$ in one dimension). Hence, the traditional Fourier transform and associated tools are no longer available. This is one of the reasons why the rigorous analysis of dispersive nonlinear equations becomes significantly more challenging in the initial-boundary value problem setting. This fact is clearly reflected in the literature, where the number of works devoted to the well-posedness of such problems is dramatically lower than that of works on the initial value problem. At the same time, initial-boundary value problems are extremely relevant in applications, as they are directly motivated by real-world phenomena (where physical boundaries are omnipresent) and involve a higher degree of complexity than initial value problems due to the additional presence of boundary data (on top of initial data). What is more, the study of initial-boundary value problems is essential in directions that extend well beyond the usual framework of analysis of partial differential equations. For example, in the field of systems theory and control, there is a close connection between the controllability of a system and the regularity of certain associated initial-boundary value problems. In particular, it is well-known that initial-boundary value problems with nonhomogeneous (i.e. nonzero) boundary conditions can be used to model physical evolutions in which the boundary input acts as control. Such boundary control models are particularly important for governing dynamics of physical processes in which access to the interior of a medium is blocked or not feasible, whereas manipulations through the boundary remain an efficient choice. 

The above brief remarks illustrate the importance of obtaining rigorous theoretical results for dispersive equations in the context of initial-boundary value problems, starting naturally from the fundamental question of Hadamard well-posedness. In 2002, Colliander and Kenig \cite{ck2002} as well as Bona, Sun and Zhang \cite{bsz2002} established the well-posedness of the KdV equation on the half-line with data in Sobolev spaces. In the first work, which is actually devoted to the generalized KdV equation, the authors performed their analysis by first expressing the (half-line) initial-boundary value problem as a  combination of cleverly constructed (whole-line) initial value problems and then employing ideas and techniques from the initial value problem toolbox --- notably, from the remarkable series of works on the NLS and KdV initial value problems by Kenig, Ponce and Vega~\cite{kpv1989,kpv1991-osc,kpv1991-kdv,kpv1993,kpv1996}. In the second work (as well as in later works by the same authors e.g. \cite{bsz2006,bsz2008,bsz2018}), the lack of a spatial Fourier transform in the initial-boundary value problem setting is bypassed by employing instead a temporal Laplace transform. In both cases, the authors were able to derive solution formulae for the forced linear KdV on the half-line and hence obtain suitable linear estimates that allowed them to establish well-posedness of the KdV on the half-line via a contraction mapping argument. The subsequent works by Holmer \cite{h2005,h2006} provided, respectively, the counterpart of \cite{ck2002} for the NLS equation on the half-line and the well-posedness of KdV on the half-line in Bourgain spaces. Other notable works during this early stage of well-posedness analysis of initial-boundary value problems are those by Faminskii~\cite{fam2004,fam2007} and Strauss and Bu \cite{sb2001}. 

A third approach (in addition to the two by \cite{ck2002,bsz2002} above) for proving the well-posedness of dispersive nonlinear equations in the context of nonhomogeneous initial-boundary value problems was introduced in~\cite{fhm2016,fhm2017} for the KdV and NLS equations on the half-line. The main idea behind this alternative approach was to use, for the purpose of linear estimates and the Picard iteration, the solution formulae obtained for the respective forced linear problems via a method known as the \textit{unified transform} or the Fokas method. This method, which was introduced by Fokas in 1997 \cite{f1997} and was subsequently developed by many researchers in various directions (see \cite{f2008,fs2012,fp2015,m2023} for references to some of these contributions), is a method of solution of \textit{linear} initial-boundary value problems that provides a direct analogue (for domains with a boundary) of the spatial Fourier transform used in the analysis of the initial value problem. The method relies on a clever use of the symmetries of the dispersion relation of the (linearized) equation  combined with appropriate deformations in the spectral  complex plane. Hence, it yields novel explicit solution representations that involve complex contours of integration. These contours induce exponential decay in the associated integrands, which allows one to efficiently estimate the relevant integrals once the appropriate analysis tools are developed (since the tools used in the study of the initial value problem are often limited to real spectral (Fourier) parameters). 

The method of \cite{fhm2016,fhm2017} was employed for studying the well-posedness of other nonhomogeneous initial-boundary value problems for the NLS, KdV and other nonlinear  equations on the half-line or the finite interval, e.g. see \cite{hm2015,hmy2019-nls,hmy2019-kdv,hm2021,amo2024}. In addition, parallel efforts include the works~\cite{oy2019,ko2022} for the NLS and biharmonic NLS equations on the half-line, as well as the works~\cite{hy2022-jde,hy2022-na} for the KdV and higher-dispersion KdV equation on the half-line in low regularity spaces. 

An important challenge, following the progress on nonlinear initial-boundary value problems outlined in the above paragraphs, was  
the investigation of the well-posedness of such problems for dispersive nonlinear equations in higher than one spatial dimensions. In this direction, the works \cite{rsz2018,a2019,hm2020} established the well-posedness (independently and via different methods) of the  NLS equation on the half-plane $\left\{x_1\in\mathbb R, x_2 > 0\right\}$ (in fact, \cite{a2019} provided linear estimates for the $n$-dimensional half-space $\left\{(x_1, \ldots, x_{n-1})\in\mathbb R^{n-1}, x_n>0\right\}$). 
The work \cite{hm2020}, in particular, advanced for the first time the unified transform approach of \cite{fhm2016,fhm2017} to two spatial dimensions in the case of Dirichlet data (the case of Neumann and Robin data for the  NLS equation on the half-plane was addressed in \cite{hm2022}).

A remarkable feature that emerged spontaneously in the analysis of \cite{hm2020} is that the correct space for the boundary data of the NLS half-plane problem is a certain combination of Bourgain-type spaces analogous to those used as solution spaces in the well-posedness of the  NLS equation on the line \cite{b1999} (see \eqref{xsb-def} above). In other words, the solution space of the  initial value problem in one spatial dimension assumes the role of the boundary data space for the initial-boundary value problem in two spatial dimensions. 
It should be emphasized that identifying the correct space for the boundary data is an \textit{additional} task in the analysis of an initial-boundary value problem in comparison to that of an initial value problem. This is because placing the initial data in a certain (Sobolev) space affects the regularity of the boundary data; nevertheless, it is not a priori clear what this regularity precisely is. In one spatial dimension, the boundary data only depend on the temporal variable and so their regularity is simpler to describe. For example, it is shown in \cite{fhm2017} that in the case of the NLS on the half-line with initial data in the Sobolev space $H^s$ the Dirichlet boundary data must belong to the Sobolev space $H^{\frac{2s+1}{4}}$. In higher than one spatial dimensions, however, the boundary data depend on both space and time and hence they have a more complex nature described by Bourgain spaces, as shown in \cite{hm2020}.

An interesting observation about the NLS half-plane problem is that while one of the two spatial variables is restricted to the half-line the other one remains unrestricted. Hence, the half-plane (or, more generally, half-space) problem can essentially be treated as a direct generalization of the half-line problem, at least as far as the linear solution formula and relevant estimates are concerned. This is not to say that the works \cite{rsz2018,a2019,hm2020} do not involve a higher level of complexity and novelty of results when compared to the half-line problem studied in~\cite{h2005,fhm2017,bsz2018}. To the contrary, as noted above, new phenomena emerge including the use of Bourgain-type spaces for the boundary data. However, from an analysis point of view, it seems fair to say that the half-plane problem is not a fully two-space-dimensional  initial-boundary value problem, since one of the two spatial variables belongs to the whole line and, therefore, the associated terms in the iteration map can be handled entirely via initial value problem techniques. 

Following the above reasoning, the simplest spatial domain that leads to an initial-boundary value problem with a fully two-space-dimensional nature is the quarter-plane $\mathbb R_+^2 := \left\{x_1>0, x_2>0\right\}$ of the NLS initial-boundary value problem \eqref{qnls-ibvp}. The fact that both variables are restricted to the half-line introduces the need for two pieces of boundary data, one at $x_1=0$ and another one at $x_2=0$. The correct function spaces for these data are not the Bourgain spaces of \cite{hm2020} since, unlike the half-plane $\left\{x_1\in\mathbb R, x_2>0\right\}$, none of the spatial variables is now unrestricted and so the relevant pieces of boundary data are no longer globally defined in space (observe from~\eqref{xsb-def} that the Bourgain norm involves globally defined functions). Instead, our analysis shows that the suitable boundary data spaces for the quarter-plane problem~\eqref{qnls-ibvp} are defined via the norm \eqref{xt-def}, which involves a half-line spatial Fourier transform as opposed to the regular whole-line Fourier transform. 

Although the Bourgain-type spaces \eqref{xt-def} may seem like a natural and straightforward modification of the Bourgain spaces~\eqref{xsb-def} used in \cite{hm2020} due to the fact that the quarter-plane boundary data are half-line functions with respect to the spatial variable they depend on, the  implementation of this modification is not as straightforward as one may have expected. In fact, starting with problem \eqref{qnls-ibvp} of the present work, the spaces of the type \eqref{xt-def} have motivated a \textit{new approach} for the treatment of initial-boundary value problems which, to the best of our knowledge, differs from every other approach used so far in the literature in the following crucial way: The linear estimates, which are of vital importance in the contraction mapping argument for the nonlinear problem, are now derived by \textit{estimating directly the entire solution formula} for the forced linear problem. This is in stark contrast to the existing works on the well-posedness of dispersive initial-boundary value problems, including in particular the works~\cite{ck2002,bsz2002,fhm2017} that have inspired the three main approaches used by most works in the literature. Indeed, in all of those works, the forced linear initial-boundary value problem is broken into component problems by means of the superposition principle. Each of these pieces depends solely on either the initial data, or the forcing, or the boundary data and is estimated \textit{individually}. Importantly, those components that only depend on the initial data or the forcing are handled essentially via initial value problem techniques. Although this decomposition is convenient, it comes with certain compatibility requirements between the regularity of the original boundary data and of the boundary traces of the component problems that involve either the initial data or the forcing. In one spatial dimension, checking that these requirements are actually true is possible, although it does add some technicalities to the analysis. In two spatial dimensions, however, the boundary data become spatiotemporal functions and so the compatibility conditions required in order to employ the ``decomposition approach'' are more difficult to verify, especially when considering a fully two-dimensional domain like the quarter-plane (as opposed to the half-plane).

In order to state our main result, we recall the definition of the Bessel potential space $H^{s, p}(\mathbb R^2)$ as the space of all functions $\varphi \in L^p(\mathbb R^2)$ such that
\begin{equation}
\no{\varphi}_{H^{s, p}(\mathbb R^2)}
:=
\no{\mathcal F^{-1}\left\{\left(1+k_1^2+k_2^2\right)^{\frac s2} \mathcal F\{\varphi\}(k_1, k_2)\right\}}_{L^p(\mathbb R^2)} < \infty,
\end{equation}
and the space $H^{s, p}(\mathbb R_+^2)$ as the restriction of $H^{s, p}(\mathbb R^2)$ to the quarter-plane $\mathbb R_+^2$ via the infimum norm
\begin{equation}
\no{\phi}_{H^{s, p}(\mathbb R_+^2)}
=
\inf\left\{\no{\varphi}_{H^{s, p}(\mathbb R^2)}: \varphi|_{\mathbb R_+^2} = \phi \right\}.
\end{equation}
Moreover, we refer to the pair $(q, p)$ as \textit{admissible} if it satisfies the conditions
\begin{equation}\label{adm-cond}
\frac 1q + \frac 1p = \frac 12, \quad 2 \leq p < \infty.
\end{equation}
With the above definitions at hand, the main result of this work can be stated as follows:
\begin{theorem}[Hadamard well-posedness for NLS on the spatial quarter-plane]\label{lwp-t}
Let $0\leq s < \frac 12$, $2 \leq \alpha \leq  \frac{3-s}{1-s}$ and $(q, p)$ be the admissible pair  
\begin{equation}\label{qp-t} 
q = \frac{2\alpha}{(1-s)(\alpha-1)}, \quad 
p=\frac{2\alpha}{1+(\alpha-1)s}.
\end{equation}
Then, the initial-boundary value problem \eqref{qnls-ibvp} for the  NLS equation on the spatial quarter-plane has a unique solution 
$$
u \in C_t([0, T]; H_x^s(\mathbb R_+^2)) \cap L_t^q((0, T); H_x^{s, p}(\mathbb R_+^2))
$$ 
that admits the estimate
\begin{equation}
\begin{aligned}
&\quad
\sup_{t\in [0, T]} \no{u(t)}_{H_x^s(\mathbb R_+^2)}
+
\no{u}_{L_t^q((0, T); H_x^{s, p}(\mathbb R_+^2))}
\\
&\leq
c_{s, \alpha} 
\Big(\no{u_0}_{H^s(\mathbb R_+^2)} + \no{g_0}_{\mathcal X_{1, T}^{0, \frac{2s+1}{4}}} + \no{g_0}_{\mathcal X_{1, T}^{s, \frac 14}} 
+ \no{h_0}_{\mathcal X_{2, T}^{0, \frac{2s+1}{4}}} + \no{h_0}_{\mathcal X_{2, T}^{s, \frac 14}}\Big),
\end{aligned}
\end{equation}
where $c_{s, \alpha}>0$ is a constant that only depends on $s$ and $\alpha$, and the lifespan $T = T(u_0, g_0, h_0, s, \alpha) > 0$ satisfies
\begin{equation}
\left(2 c_{s, \alpha}\right)^\alpha
\Big(\no{u_0}_{H^s(\mathbb R_+^2)} + \no{g_0}_{\mathcal X_{1, T}^{0, \frac{2s+1}{4}}} + \no{g_0}_{\mathcal X_{1, T}^{s, \frac 14}} 
+ \no{h_0}_{\mathcal X_{2, T}^{0, \frac{2s+1}{4}}} + \no{h_0}_{\mathcal X_{2, T}^{s, \frac 14}}\Big)^{\alpha-1} T^{\frac{3-s-\alpha(1-s)}{2}} < 1
\end{equation}
with the additional requirement in the critical case $\alpha = \frac{3-s}{1-s}$ that the initial data norm $\no{u_0}_{H^s(\mathbb R_+^2)}$ is sufficiently small. 
Furthermore, the data-to-solution map $\{u_0, g_0, h_0\} \mapsto u$ is locally Lipschitz continuous.
\end{theorem}

Theorem \ref{lwp-t} will be proved through a contraction mapping argument by combining certain nonlinear estimates (see Proposition \ref{qnls-non-p}) with the following crucial linear estimate, which holds for any $0\leq s < \frac 12$ and any admissible pair $(q, p)$ in the sense of \eqref{adm-cond}:
\begin{equation}\label{fls-le}
\begin{aligned}
&\quad
\sup_{t\in [0, T]} \no{S\big[u_0, g_0, h_0; f\big](t)}_{H_x^s(\mathbb R_+^2)}
+
\no{S\big[u_0, g_0, h_0; f\big]}_{L_t^q((0, T); H_x^{s, p}(\mathbb R_+^2))}
\\
&\leq
c_{s, p} \Big(\no{u_0}_{H^s(\mathbb R_+^2)} + \no{g_0}_{\mathcal X_{1, T}^{0, \frac{2s+1}{4}}} + \no{g_0}_{\mathcal X_{1, T}^{s, \frac 14}} 
+ \no{h_0}_{\mathcal X_{2, T}^{0, \frac{2s+1}{4}}} + \no{h_0}_{\mathcal X_{2, T}^{s, \frac 14}}
+ \no{f}_{L_t^1((0, T); H_x^s(\mathbb R_+^2))}
\Big),
\end{aligned}
\end{equation}
where $S\big[u_0, g_0, h_0; f\big]$ denotes the unified transform solution \eqref{qnls-utm-sol} to the forced linear Schr\"odinger equation on the spatial quarter-plane, i.e. to the problem
\begin{equation}\label{qnls-fls-ibvp}
\begin{aligned}
&iu_t + u_{x_1x_1} + u_{x_2x_2} = f(x_1, x_2, t), \quad (x_1, x_2) \in \mathbb R_+^2, \ 0<t<T, 
\\
&u(x_1, x_2, 0) = u_0(x_1, x_2),
\\
&u(x_1, 0, t) = g_0(x_1, t),
\quad
u(0, x_2, t) = h_0(x_2, t).
\end{aligned}
\end{equation}
The proof of the linear estimate \eqref{fls-le} follows by combining the Sobolev and Strichartz estimates of Theorems~\ref{sob-est-t} and \ref{strich-est-t}. In particular, thanks to the quarter-plane Strichartz estimates of Theorem \ref{strich-est-t}, the linear estimate~\eqref{fls-le} and, in turn, the nonlinear well-posedness Theorem~\ref{lwp-t} are  applicable to rough spaces with spatial Sobolev exponent as low as $s=0$. In fact, we anticipate that this result is optimal at least as far as contraction mapping methods are concerned. To this end, we mention that, in a series of recent papers, Killip, Vi\c{s}an and collaborators have been able to push the well-posedness Sobolev exponents of the initial value problem of certain \textit{completely integrable} equations (such as the cubic NLS equation in one spatial dimension~\cite{zs1971}) below what was previously regarded as the optimal threshold (e.g. see \cite{kv2019,hknv2022,hkv2020} and also \cite{bp2022} by Bahouri and Perelman). Their method relies crucially on the infinite number of conservation laws admitted by these integrable equations. The question of whether their techniques can be generalized to non-integrable systems like the NLS equation in two spatial dimensions considered in this work is an interesting one.

As already noted above Theorem \ref{lwp-t}, instead of following the decomposition approach used so far in the literature, the present work introduces a different approach in which the entire linear solution formula is estimated directly, i.e. \textit{without escaping from the initial-boundary value problem framework}. This new way of deriving the linear estimate \eqref{qnls-fls-ibvp} highlights even further the ability of the unified transform (which is used here in order to derive the solution formula \eqref{qnls-utm-sol} for the forced linear Schr\"odinger equation on the spatial quarter-plane) to produce effective and versatile solution representations which can be used for proving the well-posedness of nonlinear dispersive initial-boundary value problems. 

Finally, it is worth noting yet another important difference between the quarter-plane $\mathbb R_+^2$ considered in this work and the half-plane $\mathbb R \times \mathbb R_+$ (as well as the half-line and the finite interval in one dimension). Unlike these previously studied domains, the boundary of the quarter-plane is no longer smooth but only Lipschitz continuous, as it involves a corner at the origin. This lack of smoothness prevents one from directly employing certain classical results related to extensions in Sobolev spaces, since in their standard form found in the literature these results require either a smooth or a compact boundary. Hence, in order to study the initial-boundary value problem~\eqref{qnls-ibvp}, it is necessary to expand the validity of these Sobolev extension results to a domain with a non-smooth (Lipschitz) and non-compact boundary like the quarter-plane. The relevant proofs, which are given in Appendix \ref{app-s}, combine existing ideas from the theory of Sobolev spaces along with suitable adaptations in order to handle the specific nature of the quarter-plane boundary. In this connection, we note that the present work is, to the best of our knowledge, the first one on the well-posedness of a nonhomogeneous nonlinear initial-boundary value problem formulated on a domain with a non-smooth (Lipschitz) and non-compact boundary.

The unified transform and the direct estimation technique used in this work could also be applied to other dispersive or, more generally, evolution-type partial differential equations. One example is the Zakharov-Kuznetsov equation on a two-dimensional half-plane or quarter-plane. In general, our approach would be especially helpful in situations where it is difficult to establish time estimates for the nonhomogeneous Cauchy problems that arise when using standard decomposition methods. While the core idea behind our approach still applies, additional technical challenges may appear due to different dispersion relations or nonlinear terms.

Another promising direction is to extend our results to more general polygonal domains with Lipschitz and non-compact boundaries beyond the quarter-plane, provided an explicit Ehrenpreis-type representation formula is available for the associated linear problem. Although working with corners and non-smooth boundaries introduces new difficulties, especially when it comes to Sobolev extensions and boundary trace regularity, we believe that the ideas developed in Appendix \ref{app-s} of the present paper provide a solid starting point for handling such cases.
\\[2mm]
\textbf{Organization of the paper.}
In Section \ref{le-s}, we establish Sobolev as well as Strichartz-type estimates for the forced linear quarter-plane problem \eqref{qnls-fls-ibvp} --- Theorems \ref{sob-est-t} and \ref{strich-est-t}, respectively --- which yield the linear estimate~\eqref{fls-le}. Importantly, we estimate the linear problem as a whole directly through the relevant unified transform solution formula, i.e. without decomposing it by means of the initial value problem. Having established the necessary linear estimates, in Section \ref{lwp-s} we combine them with appropriate nonlinear estimates and a contraction mapping argument in order to prove Theorem \ref{lwp-t} for the Hadamard well-posedness of the NLS quarter-plane problem~\eqref{qnls-ibvp}. 
The unified transform solution formula used for proving the linear estimate \eqref{fls-le} is established in Appendix \ref{utm-s}. Finally, in Appendix \ref{app-s}, we provide the proof of certain extension results in Sobolev spaces which remain valid in the case of the quarter-plane despite the fact that the boundary of this domain is non-smooth and non-compact.

\section{Linear estimates on the spatial quarter-plane}
\label{le-s}

In this section, we establish the key linear estimate \eqref{fls-le} for the forced linear counterpart of the NLS quarter-plane problem \eqref{qnls-ibvp}, namely problem \eqref{qnls-fls-ibvp}. This is achieved by directly estimating the relevant solution formula
\begin{equation}\label{qnls-utm-sol}
\begin{aligned}
u(x_1, x_2, t)
&=
\frac{1}{(2\pi)^2}
\int_{\mathbb R}\int_{\mathbb R}
e^{ik_1x_1+ik_2x_2-i\omega t}
\Big[
\what u_0(k_1, k_2)
-i\int_0^t 
e^{i\omega t'}\what f(k_1, k_2, t')dt'
\Big]
dk_2dk_1
\\
&\quad
-\frac{1}{(2\pi)^2}
\int_{\mathbb R}\int_{\mathbb R}
e^{ik_1x_1+ik_2x_2-i\omega t}
\Big[
\what u_0(-k_1,k_2)
-
i \int_0^t e^{i\omega t'}\what f(-k_1, k_2, t')dt' 
\Big]  dk_2dk_1
\\
&\quad
-\frac{1}{(2\pi)^2}
\int_{\mathbb R}\int_{\mathbb R}
e^{ik_1x_1+ik_2x_2-i\omega t}
\Big[
\what u_0(k_1,-k_2) 
- i \int_0^t e^{i\omega t'}\what f(k_1, -k_2, t')dt'
\Big] dk_2 dk_1
\\
&\quad
+\frac{1}{(2\pi)^2}\int_{\mathbb R}\int_{\mathbb R}
e^{ik_1x_1+ik_2x_2-i\omega t} 
\Big[
\what u_0(-k_1,-k_2) -i\int_0^t 
e^{ i\omega  t'}
\what f(-k_1, -k_2, t')dt' 
\Big] dk_2dk_1
\\
&\quad
+\frac{1}{(2\pi)^2}
\int_{\mathbb R}\int_{\p D_1}
e^{ik_1x_1+ik_2x_2-i\omega t}
\cdot 
2k_1 \big[\, \widetilde  h_0(k_2, \omega, T) - \widetilde  h_0(-k_2, \omega, T)\big] 
dk_1dk_2
\\
&\quad
+\frac{1}{(2\pi)^2}
\int_{\mathbb R}\int_{\p D_2}
e^{ik_1x_1+ik_2x_2-i\omega t}
\cdot 2k_2 \big[\widetilde  g_0(k_1, \omega, T) - \widetilde  g_0(-k_1, \omega, T)\big]  dk_2 dk_1,
\end{aligned}
\end{equation}
which is derived via the unified transform of Fokas in Appendix \ref{utm-s}. In the above formula, $\what u_0$ and $\what f$ denote the \textit{quarter-plane} Fourier transforms of the initial data and the forcing, respectively, that is
\begin{equation}
\begin{aligned}
\what u_0(k_1, k_2) &= \int_{\mathbb R_+} \int_{\mathbb R_+} e^{-ik_1x_1-ik_2x_2} u_0(x_1, x_2) dx_2 dx_1,
\\
\what f(k_1, k_2, t) &= \int_{\mathbb R_+} \int_{\mathbb R_+} e^{-ik_1x_1-ik_2x_2} f(x_1, x_2, t) dx_2 dx_1,
\end{aligned}
\end{equation}
while $\widetilde g_0$ and $\widetilde h_0$ are spatiotemporal transforms of the Dirichlet boundary data given by
\begin{equation}\label{tilde-def-0}
\widetilde g_0(k_1, \omega, T)
=
\int_0^T e^{i\omega t} \what g_0^{x_1}(k_1, t) dt,
\quad
\widetilde h_0(k_2, \omega, T)
=
\int_0^T e^{i\omega t} \what h_0^{x_2}(k_2, t) dt
\end{equation}
where the dispersion relation is given by
\begin{equation}\label{omega-def}
\omega = \omega(k_1, k_2) = k_1^2+k_2^2
\end{equation}
and $\what g_0^{x_1}$ and $\what h_0^{x_2}$ denote the half-line Fourier transforms \eqref{hl-ft-def} of $g_0$ and $h_0$ with respect to $x_1$ and $x_2$ respectively. Moreover, the complex contours of integration $\p D_1$ and $\p D_2$ are the positively oriented boundaries of the first quadrant of the $k_1$ and the $k_2$ complex planes, respectively, as shown in Figure \ref{qnls-dplus}.

\begin{figure}[ht]
\centering
\vspace{2cm}
\begin{tikzpicture}[scale=1.1]
\pgflowlevelsynccm
\draw[line width=.5pt, black, dashed](0,0)--(-0.8,0);
\draw[line width=.5pt, black, dashed](0,0)--(0,-0.8);
\draw[line width=.5pt, black](1.45,1.8)--(1.45,1.45);
\draw[line width=.5pt, black](1.45,1.45)--(1.83,1.45);
\node[] at (1.59, 1.65) {\fontsize{8}{8} $k_j$};
\node[] at (-0.25, -0.25) {\fontsize{10}{10} $0$};
\draw[middlearrow={Stealth[scale=1.3, reversed]}] (0,0) -- (90:1.8);
\draw[middlearrow={Stealth[scale=1.3]}] (0,0) -- (0:1.8);
\node[] at (0.9, 0.95) {\fontsize{10}{10}\it $D_j$};
\end{tikzpicture}
\vspace{7mm}
\caption{The region $D_j$ and its positively oriented boundary $\p D_j$, $j=1,2$.}
\label{qnls-dplus}
\end{figure}
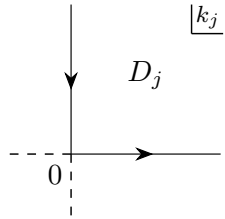

We begin with the estimation of the linear solution formula \eqref{qnls-utm-sol} for each $t\in [0, T]$ as a function of $x = (x_1, x_2)$ in the quarter-plane Sobolev space $H_x^s(\mathbb R_+^2)$.
\begin{theorem}[Sobolev estimate]\label{sob-est-t}
Suppose $0\leq s < \frac 12$. Then, the solution $u = S\big[u_0, g_0, h_0; f\big]$ to the quarter-plane initial-boundary value problem~\eqref{qnls-fls-ibvp} of the forced linear Schr\"odinger equation, as given by formula \eqref{qnls-utm-sol}, admits the Sobolev estimate
\begin{equation}\label{fls-hs}
\begin{aligned}
\sup_{t\in [0, T]} \no{S\big[u_0, g_0, h_0; f\big](t)}_{H_x^s(\mathbb R_+^2)}
&\leq
c_s \Big(\no{u_0}_{H^s(\mathbb R_+^2)} + \no{g_0}_{\mathcal X_{1, T}^{0, \frac{2s+1}{4}}} + \no{g_0}_{\mathcal X_{1, T}^{s, \frac 14}} 
\\
&\hskip 1cm
+ \no{h_0}_{\mathcal X_{2, T}^{0, \frac{2s+1}{4}}} + \no{h_0}_{\mathcal X_{2, T}^{s, \frac 14}}
+ \no{f}_{L_t^1((0, T); H_x^s(\mathbb R_+^2))}
\Big).
\end{aligned}
\end{equation}
\end{theorem}
\begin{proof}
Omitting dependencies for brevity, we write the solution formula \eqref{qnls-utm-sol} as the combination
\begin{equation}\label{qnls-utm-sol-dec}
\begin{aligned}
u &= 
\left(I_{+1, +2} - I_{-1, +2}  - I_{+1, -2}  + I_{-1, -2}\right)
\\
&\quad
- i \left(J_{+1, +2} - J_{-1, +2} - J_{+1, -2} + J_{-1, -2}\right)
\\
&\quad
+ \left(K_{+1} - K_{-1} + K_{+2} - K_{-2}\right),
\end{aligned}
\end{equation}
where the various integrals associated with the initial data, the forcing, and the boundary data are given by
\begin{align}
I_{\pm 1, \pm 2}(x_1, x_2, t) 
&=
\frac{1}{(2\pi)^2}
\int_{\mathbb R}\int_{\mathbb R}
e^{ik_1x_1+ik_2x_2-i\omega t}
\, \what u_0(\pm k_1, \pm k_2)
dk_2dk_1,
\label{i12-def}
\\
J_{\pm 1, \pm 2}(x_1, x_2, t) 
&=
\frac{1}{(2\pi)^2}
\int_{\mathbb R}\int_{\mathbb R}
e^{ik_1x_1+ik_2x_2-i\omega t}
\int_0^t e^{i\omega t'}\what f(\pm k_1, \pm k_2, t')dt'
dk_2dk_1,
\label{j12-def}
\\
K_{\pm 1}(x_1, x_2, t) 
&=
\frac{1}{(2\pi)^2}
\int_{\mathbb R}\int_{\p D_2}
e^{ik_1x_1+ik_2x_2-i\omega t}
\cdot 2k_2 \widetilde  g_0(\pm k_1, \omega, T)  dk_2 dk_1,
\label{k1-def}
\\
K_{\pm 2}(x_1, x_2, t) 
&=
\frac{1}{(2\pi)^2}
\int_{\mathbb R}\int_{\p D_1}
e^{ik_1x_1+ik_2x_2-i\omega t}
\cdot 
2k_1 \widetilde  h_0(\pm k_2, \omega, T) dk_1dk_2.
\label{k2-def}
\end{align}
We estimate each of the above four types of integrals separately. We note, in particular, that the terms \eqref{k1-def} and~\eqref{k2-def} (which involve the boundary data) require novel techniques developed specifically for the initial-boundary value problem framework.
\vskip 3mm
\noindent
\textit{Estimation of $I_{\pm 1, \pm 2}$.} 
The key observation for estimating these four terms is that they actually make sense for all $(x_1, x_2) \in \mathbb R^2$ and not just for $(x_1, x_2) \in \mathbb R_+^2$. Consider the extension by zero of $u_0 \in H^s(\mathbb R_+^2)$ defined by
\begin{equation}\label{0-ext}
U_0(x_1, x_2) = \left\{\begin{array}{ll} u_0(x_1, x_2), &(x_1, x_2) \in \mathbb R_+^2,
\\
0, &(x_1, x_2) \in \mathbb R^2 \setminus \mathbb R_+^2.
\end{array}
\right.
\end{equation}
By Proposition \ref{es-p}, if $0\leq s < \frac 12$ then $U_0 \in H^s(\mathbb R^2)$ with 
\begin{equation}\label{ext-0}
\no{U_0}_{H^s(\mathbb R^2)} \leq c_s \no{u_0}_{H^s(\mathbb R_+^2)}. 
\end{equation}
Thus, observing that $\what u_0(k_1, k_2) = \mathcal F\{U_0\}(k_1, k_2)$, we have
\begin{equation}\label{0-ext-2}
I_{\pm 1, \pm 2}(x_1, x_2, t) 
=
\frac{1}{(2\pi)^2}
\int_{\mathbb R}\int_{\mathbb R}
e^{ik_1x_1+ik_2x_2-i\omega t}
\mathcal F\{U_0\}(\pm k_1, \pm k_2)
dk_2dk_1
\end{equation}
which implies that $\mathcal F_x\{I_{\pm 1, \pm 2}\}(k_1, k_2, t) = e^{-i\omega t} \mathcal F\{U_0\}(\pm k_1, \pm k_2)$. In turn, by the definition of the Sobolev norm, 
\begin{align}
\no{I_{\pm 1, \pm 2}(t)}_{H_x^s(\mathbb R_+^2)}
&\leq
\no{I_{\pm 1, \pm 2}(t)}_{H_x^s(\mathbb R^2)}
=
\no{\left(1+k_1^2+k_2^2\right)^{\frac s2} e^{-i\omega t} \mathcal F\{U_0\}(\pm k_1, \pm k_2)}_{L_k^2(\mathbb R^2)}
\nn\\
&=
\no{\left(1+k_1^2+k_2^2\right)^{\frac s2} \mathcal F\{U_0\}(k_1, k_2)}_{L_k^2(\mathbb R^2)}
=
\no{U_0}_{H^s(\mathbb R^2)},
\quad 
s\in\mathbb R,
\nn 
\end{align}
with the penultimate equality due to the fact that $\omega$ is real for $k_1, k_2 \in \mathbb R$. Hence, in view of \eqref{ext-0}, 
\begin{equation}\label{ipm-est}
\sup_{t\in [0, T]} \no{I_{\pm 1, \pm 2}(t)}_{H_x^s(\mathbb R_+^2)}
\leq 
c_s \no{u_0}_{H^s(\mathbb R_+^2)}, \quad 0\leq s < \frac 12.
\end{equation}
\vskip 3mm
\noindent
\textit{Estimation of $J_{\pm 1, \pm 2}$.} 
We proceed similarly to the handling of $I_{\pm 1, \pm 2}$ and exploit the fact that $J_{\pm 1, \pm 2}$ make sense for all $x_1, x_2 \in \mathbb R$ and not just for $x_1, x_2 \in \mathbb R_+$. 
For each $t\in [0, T]$ and $f(t) \in  H_x^s(\mathbb R_+^2)$, we consider 
\begin{equation}\label{f-ext}
F(x_1, x_2, t) = \left\{\begin{array}{ll} f(x_1, x_2, t), &(x_1, x_2) \in \mathbb R_+^2,
\\
0, &(x_1, x_2) \in \mathbb R^2 \setminus \mathbb R_+^2.
\end{array}
\right.
\end{equation}
For each $t\in (0, T)$, by Proposition \ref{es-p}, if $0\leq s < \frac 12$ then $F(t) \in H_x^s(\mathbb R^2)$ with
\begin{equation}\label{ext-f}
\no{F(t)}_{H_x^s(\mathbb R^2)} \leq c_s \no{f(t)}_{H_x^s(\mathbb R_+^2)}, \quad t\in (0, T).
\end{equation}
Therefore, noticing that 
\begin{equation}\label{f-ext-2}
\mathcal F_x\{J_{\pm 1, \pm 2}\}(k_1, k_2, t) = e^{-i\omega t}
\int_0^t e^{i\omega t'}\mathcal F_x\{F\}(\pm k_1, \pm k_2, t') dt',
\end{equation}
we have
\begin{align}
\no{J_{\pm 1, \pm 2}(t)}_{H_x^s(\mathbb R_+^2)}
&\leq
\no{J_{\pm 1, \pm 2}(t)}_{H_x^s(\mathbb R^2)}
\nn\\
&=
\no{\left(1+k_1^2+k_2^2\right)^{\frac s2} e^{-i\omega t}
\int_0^t e^{i\omega t'}\mathcal F_x\{F\}(\pm k_1, \pm k_2, t') dt'}_{L_k^2(\mathbb R^2)}
\nn
\end{align}
which, by Minkowski's integral inequality and the fact that $e^{i\omega t}$ has unit modulus, yields
\begin{align}
\no{J_{\pm 1, \pm 2}(t)}_{H_x^s(\mathbb R_+^2)}
&\leq
\left(\int_{\mathbb R}\int_{\mathbb R} \left(1+k_1^2+k_2^2\right)^s
\left(\int_0^t \left|\mathcal F_x\{F\}(\pm k_1, \pm k_2, t')\right| dt'\right)^2 dk_2 dk_1\right)^{\frac 12}
\nn\\
&\leq
\int_0^t
\left(\int_{\mathbb R}\int_{\mathbb R} \left(1+k_1^2+k_2^2\right)^s
 \left|\mathcal F_x\{F\}(\pm k_1, \pm k_2, t')\right|^2 dk_2 dk_1 \right)^{\frac 12} dt'
\nn\\
&=
\int_0^t \no{F(t')}_{H_x^s(\mathbb R^2)} dt' 
\leq
\no{F}_{L_t^1((0, T); H_x^s(\mathbb R^2))},
\quad 
s\in\mathbb R.
\nn
\end{align}
Therefore, in view of \eqref{ext-f},  
\begin{equation}\label{jpm-est}
\sup_{t\in [0, T]} \no{J_{\pm 1, \pm 2}(t)}_{H_x^s(\mathbb R_+^2)}
\leq 
c_s \no{f}_{L_t^1((0, T); H_x^s(\mathbb R_+^2))}, \quad 0\leq s < \frac 12.
\end{equation}
\vskip 3mm
\noindent
\textit{Estimation of $K_{\pm 1}$ and $K_{\pm 2}$.} 
These terms are genuinely different than the previous ones because they involve the boundary conditions and hence encapsulate the essence of the nature of our initial-boundary value problem. Let us begin with $K_{\pm 1}$. Parametrizing along the complex contour of integration $\p D_2$ (see Figure \ref{qnls-dplus}), we write
\begin{equation}\label{k1-comp}
K_{\pm 1}(x_1, x_2, t)
=
K_{\pm 1, r}(x_1, x_2, t) + K_{\pm 1, i}(x_1, x_2, t)
\end{equation}
where the components $K_{\pm 1, r}$ and $K_{\pm 1, i}$ originate from the portions of $\p D_2$ that coincide with the positive real axis and the positive imaginary axis, respectively, so that
\begin{align}
K_{\pm 1, r}(x_1, x_2, t) 
&=
\frac{1}{(2\pi)^2}
\int_{\mathbb R}\int_{\mathbb R_+}
e^{ik_1x_1+ik_2x_2-i (k_1^2+k_2^2) t}
\cdot 2k_2 \widetilde  g_0(\pm k_1, k_1^2+k_2^2, T)  dk_2 dk_1,
\label{kr-def}
\\
K_{\pm 1, i}(x_1, x_2, t) 
&=
\frac{1}{(2\pi)^2}
\int_{\mathbb R}\int_{\mathbb R_+}
e^{ik_1x_1 - k_2x_2-i(k_1^2-k_2^2) t}
\cdot 2 k_2 \widetilde  g_0(\pm k_1, k_1^2-k_2^2, T)  dk_2 dk_1.
\label{ki-def}
\end{align}

The term $K_{\pm 1, r}$ makes sense for all $x_1, x_2 \in \mathbb R$. Thus, using the Fourier transform characterization of the $H^s(\mathbb R^2)$ norm  and the fact that $e^{-i (k_1^2+k_2^2) t}$ is unimodular, we find
\begin{align}\label{kr-temp1}
\no{K_{\pm 1, r}(t)}_{H_x^s(\mathbb R_+^2)}
\leq
\no{K_{\pm 1, r}(t)}_{H_x^s(\mathbb R^2)}
&=
\no{\left(1+k_1^2+k_2^2\right)^{\frac s2}
2k_2 \widetilde  g_0(\pm k_1, k_1^2+k_2^2, T)
}_{L_{k_1, k_2}^2(\mathbb R \times \mathbb R_+)}
\nn\\
&=
\no{\left(1+k_1^2+k_2^2\right)^{\frac s2}
2k_2 \widetilde  g_0(k_1, k_1^2+k_2^2, T)
}_{L_{k_1, k_2}^2(\mathbb R \times \mathbb R_+)}.
\end{align}
For each $x_1\in \mathbb R_+$, let 
\begin{equation}
G_0(x_1, t) = \left\{\begin{array}{ll} g_0(x_1, t), &t \in (0, T),
\\
0, &t\in \mathbb R \setminus [0, T],
\end{array}
\right.
\end{equation}
and take the half-line Fourier transform \eqref{hl-ft-def} with respect to $x_1$ to obtain
\begin{equation}\label{G0-hft}
\what G_0^{x_1}(k_1, t) = \left\{\begin{array}{ll} \what g_0^{x_1}(k_1, t), &t \in (0, T),
\\
0, &t\in \mathbb R \setminus [0, T],
\end{array}
\right.
\quad k_1 \in \mathbb R.
\end{equation}
Thus, in view of the definition \eqref{tilde-def-0}, 
\begin{equation}\label{G0-tilde}
\begin{aligned}
\widetilde g_0(k_1, k_1^2+k_2^2, T)
&:=
\int_0^T e^{i (k_1^2+k_2^2) t} \, \what g_0^{x_1}(k_1, t) dt
=
\int_0^T e^{i (k_1^2+k_2^2) t} \what G_0^{x_1}(k_1, t) dt
\\
&=
\int_{\mathbb R} e^{i (k_1^2+k_2^2) t} \what G_0^{x_1}(k_1, t) dt
=
\mathcal F_t\{\what G_0^{x_1}(k_1, t)\}(-k_1^2-k_2^2)
\end{aligned}
\end{equation}
and returning to \eqref{kr-temp1} we have
\begin{equation}\label{kr-temp2} 
\no{K_{\pm 1, r}(t)}_{H_x^s(\mathbb R_+^2)}
\leq
\left(\int_{\mathbb R}\int_{\mathbb R_+} \left(1+k_1^2+k_2^2\right)^s
4k_2^2 
\left|
\mathcal F_t\{\what G_0^{x_1}(k_1, t)\}(-k_1^2-k_2^2) \right|^2 dk_2 dk_1 \right)^{\frac 12}.
\end{equation}
For $s\geq 0$, $\left(1+k_1^2+k_2^2\right)^s \simeq \left(1+k_1^2\right)^s + \left(k_2^2\right)^s$. Hence, 
\begin{align} 
\no{K_{\pm 1, r}(t)}_{H_x^s(\mathbb R_+^2)}
&\lesssim
\left(\int_{\mathbb R}\int_{\mathbb R_+} \left(1+k_1^2\right)^s
4k_2^2 
\left|
\mathcal F_t\{\what G_0^{x_1}(k_1, t)\}(-k_1^2-k_2^2) \right|^2 dk_2 dk_1 \right)^{\frac 12}
\nn\\
&\quad
+
\left(\int_{\mathbb R}\int_{\mathbb R_+} \left(k_2^2\right)^s
4k_2^2 
\left|
\mathcal F_t\{\what G_0^{x_1}(k_1, t)\}(-k_1^2-k_2^2) \right|^2 dk_2 dk_1 \right)^{\frac 12}
\end{align}
and making the change of variable $k_2 = \sqrt{-\tau}$, which implies $\tau = -k_2^2$, we obtain
\begin{align} 
\no{K_{\pm 1, r}(t)}_{H_x^s(\mathbb R_+^2)}
&\lesssim
\left(\int_{\mathbb R}\int_{\mathbb R_-} \left(1+k_1^2\right)^s
\left|\tau\right|^{\frac 12} 
\left|
\mathcal F_t\{\what G_0^{x_1}(k_1, t)\}(\tau-k_1^2) \right|^2 d\tau dk_1 \right)^{\frac 12}
\nn\\
&\quad
+
\left(\int_{\mathbb R}\int_{\mathbb R_-} \left|\tau\right|^{s+\frac 12} 
\left|
\mathcal F_t\{\what G_0^{x_1}(k_1, t)\}(\tau-k_1^2) \right|^2 d\tau dk_1 \right)^{\frac 12}.
\nn
\end{align}
The fact that $\mathcal F_t\{\varphi\}(\tau-k_1^2) = \mathcal F_t\{e^{ik_1^2t} \varphi(t)\}(\tau)$ allows us to express the above estimate in the form
\begin{align} 
\no{K_{\pm 1, r}(t)}_{H_x^s(\mathbb R_+^2)}
&\lesssim
\left(\int_{\mathbb R} \left(1+k_1^2\right)^s \int_{\mathbb R_-} 
\left|\tau\right|^{\frac 12} 
\left|
\mathcal F_t\big\{e^{ik_1^2 t} \what G_0^{x_1}(k_1, t)\big\}(\tau) \right|^2 d\tau dk_1 \right)^{\frac 12}
\nn\\
&\quad
+
\left(\int_{\mathbb R}\int_{\mathbb R_-} \left|\tau\right|^{s+\frac 12} 
\left|
\mathcal F_t\big\{e^{ik_1^2 t} \what G_0^{x_1}(k_1, t)\big\}(\tau) \right|^2 d\tau dk_1 \right)^{\frac 12}
\nn\\
&\leq
\left(\int_{\mathbb R} \left(1+k_1^2\right)^s \no{e^{ik_1^2 t} \what G_0^{x_1}(k_1, t)}_{H_t^\frac 14(\mathbb R)}^2 d\tau dk_1 \right)^{\frac 12}
\nn\\
&\quad
+
\left(\int_{\mathbb R} \no{e^{ik_1^2 t} \what G_0^{x_1}(k_1, t)}_{H_t^\frac{2s+1}{4}(\mathbb R)}^2 d\tau dk_1 \right)^{\frac 12},
\quad s \geq - \frac 12, 
\label{k1r-temp5}
\end{align}
where the restriction on $s$ is placed in order to employ the inequality $|\tau| \leq \left(1+\tau^2\right)^{\frac 12}$. 
Finally, further restricting $-\frac 12 \leq s < \frac 12$ (so that $0 \leq \frac{2s+1}{4} < \frac 12$) allows us to combine \eqref{k1r-temp5} with \eqref{G0-hft}  and Theorem 11.4 of \cite{lm1972} in order to conclude 
\begin{equation}\label{k1r-hs}
\sup_{t\in [0, T]} \no{K_{\pm 1, r}(t)}_{H_x^s(\mathbb R_+^2)}
\leq
c_s
\Big(\no{g_0}_{\mathcal X_{1, T}^{s, \frac 14}} + \no{g_0}_{\mathcal X_{1, T}^{0, \frac{2s+1}{4}}}\Big), 
\quad - \frac 12 \leq s < \frac 12,
\end{equation}
where the Bourgain-type space $\mathcal X_{j, T}^{s, b}$ is defined by \eqref{xt-def}.

Next, we proceed to the estimation of $K_{\pm 1, i}$. Notice that the corresponding expression \eqref{ki-def} does \textit{not} make sense for $x_2<0$. Hence, the Fourier transform method used for $K_{\pm 1, r}$ cannot be employed for $K_{\pm 1, i}$ and a different approach is required. The idea is to use the physical space characterization of the $H^s(\mathbb R_+^2)$ norm, which for $0\leq s < \frac 12$ (the upper bound on $s$ has already been placed in order to estimate $K_{\pm 1, r}$) is made up of the $L^2$ norm and the fractional Sobolev-Slobodecki seminorm. 
First, recalling \eqref{G0-tilde} and combining the fact that \eqref{ki-def} makes sense for all $x_1 \in \mathbb R$ with Plancherel's theorem and the fact that $e^{-ik_1^2t}$ has unit modulus, we obtain
\begin{align}
\no{K_{\pm 1, i}(t)}_{L_x^2(\mathbb R_+^2)}
&\leq
\no{K_{\pm 1, i}(t)}_{L_{x_1, x_2}^2(\mathbb R \times \mathbb R_+)}
\nn\\
&=
\no{
\frac{1}{\sqrt{2\pi}}
\no{
\frac{1}{2\pi}
\int_{\mathbb R_+}
e^{- k_2x_2-i(k_1^2-k_2^2) t}  
\cdot 2 k_2 \mathcal F_t\{\what G_0^{x_1}(\pm k_1, t)\}(-k_1^2+k_2^2)  dk_2
}_{L_{k_1}^2(\mathbb R)}
}_{L_{x_2}^2(\mathbb R_+)}
\nn\\
&=
\frac{1}{2\pi \sqrt{2\pi}}
\no{
\no{
\int_{\mathbb R_+}
e^{- k_2x_2+ik_2^2 t}  
\cdot 2 k_2 \mathcal F_t\{\what G_0^{x_1}(\pm k_1, t)\}(-k_1^2+k_2^2)  dk_2
}_{L_{x_2}^2(\mathbb R_+)}
}_{L_{k_1}^2(\mathbb R)}.
\nn
\end{align}
Then, using the boundedness of the Laplace transform in $L^2(\mathbb R_+)$ established in \cite{h1929} (see also Lemma 3.2 in~\cite{fhm2017}) for the $x_2$ norm of the $k_2$ integral, we obtain
\begin{align}
\no{K_{\pm 1, i}(t)}_{L_x^2(\mathbb R_+^2)}
&\leq
\frac{1}{2\pi \sqrt{2\pi}}
\no{
\sqrt \pi  
\no{
e^{ik_2^2 t} 
\cdot 2 k_2 \mathcal F_t\{\what G_0^{x_1}(\pm k_1, t)\}(-k_1^2+k_2^2)
}_{L_{k_2}^2(\mathbb R_+)}
}_{L_{k_1}^2(\mathbb R)}
\nn\\
&=
\frac{1}{2\pi}
\left(
\int_{\mathbb R} \int_{\mathbb R_+}
2k_2^2 \left|\mathcal F_t\{\what G_0^{x_1}(k_1, t)\}(-k_1^2+k_2^2)\right|^2 dk_2 dk_1
\right)^{\frac 12}.
\label{k1i-temp1}
\end{align}
The integral on the right-hand side is analogous to the one in \eqref{kr-temp2} hence, proceeding in a similar fashion by making the change of variable $k_2 = \sqrt \tau$, we obtain
\begin{align*}
\no{K_{\pm 1, i}(t)}_{L_x^2(\mathbb R_+^2)}
&\leq
\frac{1}{2\pi}
\left(
\int_{\mathbb R} \int_{\mathbb R_+}
\tau^{\frac 12} \left|\mathcal F_t\{\what G_0^{x_1}(k_1, t)\}(\tau-k_1^2)\right|^2 d\tau dk_1
\right)^{\frac 12}
\nn\\
&\leq
\frac{1}{2\pi}
\left(
\int_{\mathbb R} \int_{\mathbb R}
\left(1+\tau^2\right)^{\frac 14} 
\left|\mathcal F_t\big\{e^{ik_1^2t} \what G_0^{x_1}(k_1, t)\big\}(\tau)\right|^2 d\tau dk_1
\right)^{\frac 12}
\nn\\
&=
\frac{1}{2\pi}
\left(
\int_{\mathbb R} \no{e^{ik_1^2t} \what G_0^{x_1}(k_1, t)}_{H_t^{\frac 14}(\mathbb R)}^2 dk_1
\right)^{\frac 12}.
\end{align*}
Thus, in view of \eqref{G0-hft}, Theorem 11.4 of \cite{lm1972} and the definition \eqref{xt-def},  
\begin{equation}\label{k1i-l2}
\sup_{t\in [0, T]} \no{K_{\pm 1, i}(t)}_{L_x^2(\mathbb R_+^2)}
\lesssim
\left(
\int_{\mathbb R} \no{e^{ik_1^2t} \what g_0^{x_1}(k_1, t)}_{H_t^{\frac 14}(0, T)}^2 dk_1
\right)^{\frac 12}
=:
\no{g_0}_{\mathcal X_{1, T}^{0, \frac 14}}.
\end{equation}
This completes the estimation of the $L_x^2(\mathbb R_+^2)$ norm of $K_{\pm 1, i}$. 

Next, for each $t\in [0, T]$, we consider the fractional Sobolev-Slobodecki seminorm
\begin{equation}
\no{K_{\pm 1, i}(t)}_s
:=
\left(
\int_{\mathbb R_+} \int_{\mathbb R_+}\int_{\mathbb R_+} \int_{\mathbb R_+}
\frac{\left|K_{\pm 1, i}(x_1, x_2, t) - K_{\pm 1, i}(y_1, y_2, t)\right|^2}{\left[(x_1-y_1)^2 + (x_2-y_2)^2\right]^{1+s}} \, dy_1dy_2 dx_1 dx_2
\right)^{\frac 12},
\quad
0< s < 1.
\nn
\end{equation}
In fact, by symmetry, 
\begin{align}
\no{K_{\pm 1, i}(t)}_s^2
&=
2 \int_{\mathbb R_+} \int_{\mathbb R_+}\int_{\mathbb R_+} \int_{\mathbb R_+}
\frac{\left|K_{\pm 1, i}(x_1+z_1, x_2+z_2, t) - K_{\pm 1, i}(x_1, x_2, t)\right|^2}{\left(z_1^2 + z_2^2\right)^{1+s}} \, dz_1dz_2 dx_1 dx_2
\label{frac1}
\\
&\quad
+
2 \int_{\mathbb R_+} \int_{\mathbb R_+}\int_{\mathbb R_+} \int_{\mathbb R_+}
\frac{\left|K_{\pm 1, i}(x_1+z_1, x_2, t) - K_{\pm 1, i}(x_1, x_2+z_2, t)\right|^2}{\left(z_1^2 + z_2^2\right)^{1+s}} \, dz_1dz_2 dx_1 dx_2.
\label{frac2}
\end{align}
Concerning the first of the above terms, we have
\begin{align}
\eqref{frac1}
&\leq
2 \int_{\mathbb R_+} \int_{\mathbb R_+} \left(z_1^2 + z_2^2\right)^{-(1+s)}  
\int_{\mathbb R_+} \int_{\mathbb R_+}
\bigg|
\frac{1}{(2\pi)^2}
\int_{\mathbb R}\int_{\mathbb R_+}
e^{ik_1x_1 - k_2x_2-i(k_1^2-k_2^2) t}
\left(e^{ik_1z_1 - k_2z_2} - 1\right)
\nn\\
&\hskip 8cm
\cdot 2 k_2 \mathcal F_t\{\what G_0^{x_1}(\pm k_1, t)\}(-k_1^2+k_2^2)   dk_2 dk_1
\bigg|^2 dx_1 dx_2 dz_1dz_2
\nn
\end{align}
so, by Plancherel's theorem between the $x_1$ and $k_1$ integrals,  
\begin{align}
\eqref{frac1}
&\leq
\frac{1}{\pi^3} \int_{\mathbb R_+} \int_{\mathbb R_+} \left(z_1^2 + z_2^2\right)^{-(1+s)}  
\int_{\mathbb R_+} \int_{\mathbb R}
\bigg|
\int_{\mathbb R_+}
e^{- k_2x_2-i(k_1^2-k_2^2) t}
\left(e^{ik_1z_1 - k_2z_2} - 1\right)
\nn\\
&\hskip 6.7cm
\cdot k_2 \mathcal F_t\{\what G_0^{x_1}(\pm k_1, t)\}(-k_1^2+k_2^2)   dk_2
\bigg|^2 dk_1 dx_2 dz_1dz_2.
\nn
\end{align}
Shifting our attention to the pair of $x_2$ and $k_2$ integrals, we employ the boundedness of the Laplace transform in $L^2(\mathbb R_+)$ to further obtain
\begin{equation}
\eqref{frac1}
\leq
\frac{1}{\pi^2} 
\int_{\mathbb R}
\int_{\mathbb R_+}
k_2^2 \left|\mathcal F_t\{\what G_0^{x_1}(k_1, t)\}(-k_1^2+k_2^2) \right|^2
\mathcal J(k_1, k_2, s) dk_2 dk_1
\label{k1i-temp2}
\end{equation}
where 
$$
\mathcal J(k_1, k_2, s) := \int_{\mathbb R_+} \int_{\mathbb R_+}   
\frac{\left|e^{ik_1z_1 - k_2z_2} - 1\right|^2}{\left(z_1^2 + z_2^2\right)^{1+s}}  dz_1dz_2.
$$

As $\mathcal J(k_1, k_2, s) = \mathcal J(-k_1, k_2, s)$, it suffices to estimate $\mathcal J$ for $k_1\geq 0$. Suppose first that $k_2=\lambda k_1$, $\lambda \in [0, 1]$. Then, letting $(\zeta_1, \zeta_2) = (k_1 z_1, k_1 z_2)$, we have
$$
\mathcal J(k_1, \lambda k_1, s)
=
\left(k_1^2\right)^s
\int_{\mathbb R_+}\int_{\mathbb R_+}
\frac{\left|e^{i\zeta_1 - \lambda \zeta_2} - 1 \right|^2}
{\left(\zeta_1^2+\zeta_2^2\right)^{1+s}}\,
d\zeta_2d\zeta_1
=
\left(k_1^2\right)^s
\int_{0}^{\frac \pi 2} \int_{\mathbb R_+}
\frac{\left|e^{\left(i\cos \theta - \lambda \sin\theta\right)r}-1\right|^2}
{r^{1+2s}}\,
drd\theta
$$
after switching to polar coordinates. The singularity at $r=0$ is integrable since for $r\ll 1$ we have
$e^{\left(i\cos \theta - \lambda \sin\theta\right)r} - 1 = \left(i\cos \theta - \lambda \sin\theta\right)r + O(r^2)$ and $s<1$. 
Furthermore, for $r\gg 1$ integrability follows by observing that $\left|e^{\left(i\cos \theta - \lambda \sin\theta\right)r}-1\right| \leq 2$ (the boundedness of the exponential is important here) and $r^{-(1+2s)}$ is integrable at infinity since  $s>0$ (in order for the fractional seminorm to be present). Thus, for $k_2=\lambda k_1$ with $\lambda \in [0, 1]$ we have that $\mathcal J(k_1, \lambda k_1, s) \lesssim \left(k_1^2\right)^s$.
In the remaining case, namely $k_1=\lambda k_2$ with $\lambda \in [0, 1]$, we have
$$
\mathcal J(\lambda k_2, k_2, s)
=
\left(k_2^2\right)^{s} \int_{\mathbb R_+}\int_{\mathbb R_+} 
\frac{\left|e^{i\lambda \zeta_1 - \zeta_2}-1\right|^2}
{\left(\zeta_1^2+\zeta_2^2\right)^{1+s}}\,
d\zeta_2d\zeta_1.
$$
so similarly to the first case we find $\mathcal J(\lambda k_2, k_2, s) \lesssim \left(k_2^2\right)^s$.
Therefore, overall we have the estimate
$$
\mathcal J(k_1, k_2, s) \lesssim \left(k_1^2+k_2^2\right)^s, \quad k_1\in\mathbb R, \ k_2\in\mathbb R_+, \ 0<s<1,
$$
which can be combined with \eqref{k1i-temp2} to deduce
\begin{equation}
\eqref{frac1}
\lesssim
\int_{\mathbb R} \int_{\mathbb R_+}
k_2^2 \left|\mathcal F_t\{\what G_0^{x_1}(k_1, t)\}(-k_1^2+k_2^2) \right|^2 
\left(k_1^2+k_2^2\right)^s
dk_2 dk_1.
\nn
\end{equation}
The right-hand side of this inequality is similar to the one of \eqref{k1i-temp1}. Hence, proceeding as before, we obtain the following estimate for the term \eqref{frac1}:
\begin{align}
\eqref{frac1}
&\lesssim
\int_{\mathbb R} \int_{\mathbb R_+}
\tau^{\frac 12} \left|\mathcal F_t\{\what G_0^{x_1}(k_1, t)\}(\tau-k_1^2) \right|^2 
\left(\tau + k_1^2\right)^s
d\tau dk_1
\nn\\
&\leq
\int_{\mathbb R} 
\left(1 + k_1^2\right)^s
\int_{\mathbb R_+}
\left(1 + \tau^2\right)^{\frac 14} \left|\mathcal F_t\big\{e^{ik_1^2 t} \what G_0^{x_1}(k_1, t)\big\}(\tau) \right|^2 
d\tau dk_1
\nn\\
&\quad
+
\int_{\mathbb R} 
\int_{\mathbb R_+}
\left(1 + \tau^2\right)^{\frac{2s+1}{4}} \left|\mathcal F_t\big\{e^{ik_1^2 t} \what G_0^{x_1}(k_1, t)\big\}(\tau) \right|^2 
d\tau dk_1
\nn\\
&\leq
\int_{\mathbb R} 
\left(1 + k_1^2\right)^s
\no{e^{ik_1^2 t} \what G_0^{x_1}(k_1, t)}_{H_t^{\frac 14}(\mathbb R)}^2 dk_1
+
\int_{\mathbb R} 
\no{e^{ik_1^2 t} \what G_0^{x_1}(k_1, t)}_{H_t^{\frac{2s+1}{4}}(\mathbb R)}^2 dk_1.
\label{k1i-temp3}
\end{align}

The term \eqref{frac2} can be estimated in a similar way to \eqref{frac1}. Indeed, by Plancherel's theorem and the Laplace transform boundedness in $L^2(\mathbb R_+)$,
\begin{equation}
\eqref{frac2}
\lesssim
\int_{\mathbb R}
\int_{\mathbb R_+}
k_2^2 \left|\mathcal F_t\{\what G_0^{x_1}(k_1, t)\}(-k_1^2+k_2^2) \right|^2
\left(
\int_{\mathbb R_+} \int_{\mathbb R_+}   
\frac{\left|e^{ik_1z_1} - e^{- k_2z_2}\right|^2}{\left(z_1^2 + z_2^2\right)^{1+s}}  dz_1dz_2
\right) dk_2 dk_1.
\nn
\end{equation}
The integral with respect to $z_1$ and $z_2$ is equal to the one in \eqref{k1i-temp2} since $\left|e^{ik_1z_1} - e^{- k_2z_2}\right| = \left|e^{ik_1z_1- k_2z_2}-1\right|$.
Hence, we once again arrive at estimate \eqref{k1i-temp3}, namely
\begin{equation}\label{k1i-temp4}
\eqref{frac2}
\lesssim
\int_{\mathbb R} 
\left(1 + k_1^2\right)^s
\no{e^{ik_1^2 t} \what G_0^{x_1}(k_1, t)}_{H_t^{\frac 14}(\mathbb R)}^2 dk_1
+
\int_{\mathbb R} 
\no{e^{ik_1^2 t} \what G_0^{x_1}(k_1, t)}_{H_t^{\frac{2s+1}{4}}(\mathbb R)}^2 dk_1.
\end{equation}

Combining \eqref{k1i-temp3} with \eqref{k1i-temp4}, the definition \eqref{G0-hft} of $G_0$, Theorem 11.4 of \cite{lm1972} and the definition \eqref{xt-def} of $\mathcal X_{1, T}^{\sigma, b}$, we deduce 
\begin{equation}\label{k1i-frac}
\sup_{t\in [0, T]}  \no{K_{\pm 1, i}(t)}_s
\lesssim
\no{g_0}_{\mathcal X_{1, T}^{s, \frac 14}}
+
\no{g_0}_{\mathcal X_{1, T}^{0, \frac{2s+1}{4}}},
\quad
0<s<\frac 12.
\end{equation}
The $L^2$ estimate \eqref{k1i-l2} and the fractional estimate \eqref{k1i-frac} yield
\begin{equation}\label{k1i-hs}
\sup_{t\in [0, T]} \no{K_{\pm 1, i}(t)}_{H_x^s(\mathbb R_+^2)}
\lesssim
\no{g_0}_{\mathcal X_{1, T}^{s, \frac 14}}
+
\no{g_0}_{\mathcal X_{1, T}^{0, \frac{2s+1}{4}}},
\quad
0\leq s < \frac 12,
\end{equation}
which together with estimate \eqref{k1r-hs} for $K_{\pm 1, r}$ implies the final Sobolev estimate for $K_{\pm 1}$, namely
\begin{equation}\label{k1-hs}
\sup_{t\in [0, T]} \no{K_{\pm 1}(t)}_{H_x^s(\mathbb R_+^2)}
\lesssim
\no{g_0}_{\mathcal X_{1, T}^{s, \frac 14}}
+
\no{g_0}_{\mathcal X_{1, T}^{0, \frac{2s+1}{4}}},
\quad
0\leq s < \frac 12.
\end{equation}

The estimation of $K_{\pm 2}$ is entirely analogous with the one of $K_{\pm 1}$, as each integral can be obtained from the other by swapping the indices ``1'' and ``2'' and also $g_0$ with $h_0$. Therefore, the proof of Theorem \ref{sob-est-t} is complete. \end{proof}

In the low-regularity setting $0\leq s < \frac 12$ considered in this work, the space $H_x^s(\mathbb R_+^2)$ is not an algebra. Hence, in order to handle the norm that arises in estimate \eqref{fls-hs} when the forcing $f$ is replaced by the NLS nonlinearity $\pm |u|^{\alpha-1} u$, we additionally need to establish the following Strichartz-type estimates. 
\begin{theorem}[Strichartz estimates]\label{strich-est-t}
Let $0\leq s < \frac 12$ and suppose that the pair $(q, p)$ is admissible in the sense of \eqref{adm-cond}. Then, the formula \eqref{qnls-utm-sol} for the solution $u = S\big[u_0, g_0, h_0; f\big]$ to the quarter-plane initial-boundary value problem~\eqref{qnls-fls-ibvp} of the forced linear Schr\"odinger equation admits the estimate
\begin{equation}\label{fls-strich}
\begin{aligned}
\no{S\big[u_0, g_0, h_0; f\big]}_{L_t^q((0, T); H_x^{s, p}(\mathbb R_+^2))}
&\leq
c_{s, p} \Big(\no{u_0}_{H^s(\mathbb R_+^2)} + \no{g_0}_{\mathcal X_{1, T}^{0, \frac{2s+1}{4}}} + \no{g_0}_{\mathcal X_{1, T}^{s, \frac 14}} 
\\
&\hskip 1.2cm
+ \no{h_0}_{\mathcal X_{2, T}^{0, \frac{2s+1}{4}}} + \no{h_0}_{\mathcal X_{2, T}^{s, \frac 14}}
+\no{f}_{L_t^1((0, T); H_x^s(\mathbb R_+^2))}
\Big).
\end{aligned}
\end{equation}
\end{theorem}

\begin{proof}
As for the Sobolev estimate of Theorem \ref{sob-est-t}, we write the solution $u=S\big[u_0, g_0, h_0; f\big]$ in the form \eqref{qnls-utm-sol-dec} and estimate each of the four terms \eqref{i12-def}-\eqref{k2-def} separately. 

For the terms $I_{\pm 1, \pm 2}$ defined by \eqref{i12-def}, we use once again the zero extension \eqref{0-ext} to express those terms in the form \eqref{0-ext-2}, which is valid for all $(x_1, x_2) \in \mathbb R^2$ and, due to the fact that $\mathcal F\{U_0(x_1, x_2)\}(\pm k_1, \pm k_2)  = \mathcal F\{U_0(\pm x_1, \pm x_2)\}(k_1, k_2)$, corresponds to the solution of the initial value problem for the homogeneous linear Schr\"odinger equation on the plane with initial data $U_0(\pm x_1, \pm x_2)$. 
Therefore, invoking the classical Strichartz estimates from the theory of the initial value problem (see \cite{s1977} and also Remark 2.3.8 in \cite{c2003}), we have
\begin{equation}\label{i12-strich}
\no{I_{\pm 1, \pm 2}}_{L_t^q((0, T); H_x^{s, p}(\mathbb R^2))}
\lesssim
\no{U_0}_{H^s(\mathbb R^2)}
\lesssim
\no{u_0}_{H^s(\mathbb R_+^2)}, \quad 0 \leq s < \frac 12,
\end{equation}
provided that $(q, p)$ satisfies the admissibility conditions \eqref{adm-cond}. Note that the restriction on $s$ allows us to employ Proposition \ref{es-p} in order to infer the last inequality in \eqref{i12-strich}.

The terms $J_{\pm 1, \pm 2}$ defined by \eqref{j12-def} can be handled in a similar way, namely, by using the zero extension \eqref{f-ext} to express them in the form \eqref{f-ext-2}, which allows us to readily invoke the Strichartz estimates of \cite{s1977,c2003} for the forced linear initial value problem with forcing $F(\pm x_1, \pm x_2, t)$ to deduce 
\begin{equation*}
\no{J_{\pm 1, \pm 2}}_{L_t^q((0, T); H_x^{s, p}(\mathbb R^2))}
\lesssim
\no{F}_{L_t^{\frac{r}{r-1}}((0, T); H_x^{s, \frac{\rho}{\rho-1}}(\mathbb R^2))}, \quad s\in\mathbb R,
\end{equation*}
where $(q, p)$ and $(r, \rho)$ are any two pairs satisfying \eqref{adm-cond}. For the particular choice $(r, \rho) = (\infty, 2)$, the above estimate reduces to
\begin{equation}\label{j12-strich}
\no{J_{\pm 1, \pm 2}}_{L_t^q((0, T); H_x^{s, p}(\mathbb R^2))}
\lesssim
\no{F}_{L_t^1((0, T); H_x^s(\mathbb R^2))}
\lesssim
\no{f}_{L_t^1((0, T); H_x^s(\mathbb R_+^2))}, \quad 0 \leq s < \frac 12,
\end{equation}
for any admissible pair $(q, p)$ where, as before, the restriction on $s$ allows us to infer the last inequality in \eqref{j12-strich} via Proposition \ref{es-p}.

We proceed to the boundary data terms $K_{\pm 1}$ and $K_{\pm 2}$. Since the relevant expressions \eqref{k1-def} and \eqref{k2-def} can be obtained from one another by swapping ``1'' with ``2'' and $g_0$ with  $h_0$, it suffices to estimate $K_{\pm 1}$. Parametrizing the contour $\p D_2$, we write $K_{\pm 1}$ as the sum \eqref{k1-comp}. We begin with the estimation of the component $K_{\pm 1, r}$ given by~\eqref{kr-def}. Since this expression makes sense for all $(x_1, x_2) \in \mathbb R^2$, we treat it like $I_{\pm 1, \pm 2}$ by writing it as 
\begin{equation*}
K_{\pm 1, r}(x_1, x_2, t) 
=
\frac{1}{(2\pi)^2}
\int_{\mathbb R}\int_{\mathbb R}
e^{ik_1x_1+ik_2x_2-i (k_1^2+k_2^2) t} \phi(k_1, k_2)  dk_2 dk_1
\end{equation*}
where
\begin{equation*}
\phi(k_1, k_2) 
=
\left\{
\begin{array}{ll}
2k_2 \widetilde  g_0(\pm k_1, k_1^2+k_2^2, T), & k_2>0,
\\
0, & k_2<0,
\end{array}
\right.
\quad k_1 \in \mathbb R.
\end{equation*}
Thus, similarly to \eqref{i12-strich}, 
\begin{equation*}
\no{K_{\pm 1, r}}_{L_t^q((0, T); H_x^{s, p}(\mathbb R^2))}
\lesssim
\no{\phi}_{H^s(\mathbb R^2)}, \quad s\in\mathbb R,
\end{equation*}
with $(q, p)$ satisfying \eqref{adm-cond}. Then, proceeding along the lines of \eqref{kr-temp1}-\eqref{k1r-hs}, we obtain
\begin{equation}\label{k1r-strich}
\no{K_{\pm 1, r}}_{L_t^q((0, T); H_x^{s, p}(\mathbb R^2))}
\lesssim
\no{g_0}_{\mathcal X_{1, T}^{s, \frac 14}} + \no{g_0}_{\mathcal X_{1, T}^{0, \frac{2s+1}{4}}}, 
\quad - \frac 12 \leq s < \frac 12.
\end{equation}

The second component of $K_1$, namely the term $K_{\pm 1, i}$ given by \eqref{ki-def}, does not make sense for $x_2<0$ and so it must be estimated via a different technique developed for the initial-boundary value problem setting (see \cite{hm2020}).
Suppose first that $s\in \mathbb N\cup \{0\}$. Then, thanks to a classical result by Calder\'on \cite{c1961}, the Bessel potential space $H^{s, p}(\mathbb R^2)$ coincides with the Sobolev space $W^{s, p}(\mathbb R^2)$ for all $1<p<\infty$. Hence, $\left\| \cdot \right\|_{H^{s, p}(\mathbb R^2)} \simeq \left\| \cdot \right\|_{W^{s, p}(\mathbb R^2)}$ and the same is true for restrictions of these spaces on $\mathbb R_+^2$, so that
\begin{equation}\label{hsp-eq}
\left\|K_{\pm 1, i}(t) \right\|_{H_x^{s, p}(\mathbb R_+^2)}
\simeq
\sum_{j=0}^s \sum_{j_1=0}^j
\left\|\p_{x_1}^{j_1} \p_{x_2}^{j-j_1} K_{\pm 1, i}(t) \right\|_{L_x^p(\mathbb R_+^2)}.
\end{equation}
In turn, noting also that $K_{\pm 1, i}$ makes sense for all $x_1, t \in \mathbb R$, we have
\begin{equation}\label{hsp-eq2}
\left\|K_{\pm 1, i} \right\|_{L_t^q((0, T); H_x^{s, p}(\mathbb R_+^2))}
\leq
\sum_{j=0}^s \sum_{j_1=0}^j
\left\|\p_{x_1}^{j_1} \p_{x_2}^{j-j_1} K_{\pm 1, i} \right\|_{L_t^q(\mathbb R; L_{x_1, x_2}^p(\mathbb R \times \mathbb R_+))}.
\end{equation}

Let $\psi_\pm \in L^2(\mathbb R^2)$ be defined through its Fourier transform on $\mathbb R^2$ by
\begin{equation}\label{psipm-def}
\mathcal F\{\psi_\pm\}(k_1, k_2) 
:= 
\left\{
\begin{array}{ll}
2k_2 \, \widetilde  g_0(\pm k_1, k_1^2-k_2^2, T), & k_2>0,
\\
0, & k_2<0,
\end{array}
\right.
\quad k_1 \in \mathbb R.
\end{equation}
Then, for each $j, j_1\in \left\{0, \ldots, s\right\}$ with $j_1 \leq j$, taking the derivative $\p_{x_1}^{j_1} \p_{x_2}^{j-j_1}$ of \eqref{ki-def} and noting that, by the inverse Fourier transform,
$$
\p_{x_1}^{j_1} \p_{x_2}^{j-j_1} 
\psi_{\pm}(x_1, x_2) = \frac{1}{(2\pi)^2} \int_{\mathbb R} \int_{\mathbb R_+} e^{ik_1x_1+ik_2x_2} (ik_1)^{j_1} (ik_2)^{j-j_1} \, 2k_2 \, \widetilde  g_0(\pm k_1, k_1^2-k_2^2, T) dk_2 dk_1,
$$
we can express $\p_{x_1}^{j_1} \p_{x_2}^{j-j_1} K_{\pm 1, i}$ in the form
\begin{equation}\label{k1i-k}
\begin{aligned}
\p_{x_1}^{j_1} \p_{x_2}^{j-j_1} K_{\pm 1, i}(x_1, x_2, t) 
&=
\frac{i^{j-j_1} }{(2\pi)^2}
\int_{\mathbb R}\int_{\mathbb R}
\mathcal K(x_1, x_2, y_1, y_2, t)\,  \p_{y_1}^{j_1} \p_{y_2}^{j-j_1}  \psi_\pm(y_1, y_2) dy_1 dy_2,
\\
\mathcal K(x_1, x_2, y_1, y_2, t) 
&:= 
\int_{\mathbb R}\int_{\mathbb R_+} e^{ik_1(x_1-y_1)-k_2x_2-ik_2y_2-i(k_1^2-k_2^2)t} dk_2 dk_1.
\end{aligned}
\end{equation}
Combining this representation with duality and the Cauchy-Schwarz inequality, we have
\begin{equation}\label{M-temp4}
\begin{aligned}
&\quad
\left\|\p_{x_1}^{j_1} \p_{x_2}^{j-j_1} K_{\pm 1, i} \right\|_{L_t^q(\mathbb R; L_{x_1, x_2}^p(\mathbb R \times \mathbb R_+))}
\\
&=
\sup \left\{ 
\left|\int_{\mathbb R} \int_{\mathbb R}\int_{\mathbb R_+} \p_{x_1}^{j_1} \p_{x_2}^{j-j_1} K_{\pm 1, i}(x_1, x_2, t) \eta(x_1, x_2, t) dx_2 dx_1 dt\right|:
\left\| \eta \right\|_{L_t^{q'}(\mathbb R; L_{x_1, x_2}^{p'}(\mathbb R\times \mathbb R_+))}=1
\right\}
\\
&\leq 
\sup \left\{ M \no{\p_{y_1}^{j_1} \p_{y_2}^{j-j_1}  \psi_\pm}_{L^2(\mathbb R^2)}: \left\| \eta \right\|_{L_t^{q'}(\mathbb R; L_{x_1, x_2}^{p'}(\mathbb R\times \mathbb R_+))}=1
\right\}
\end{aligned}
\end{equation}
where
\begin{equation}
M := 
\frac{1}{(2\pi)^2}
\left(
\int_{\mathbb R} \int_{\mathbb R} \left|\int_{\mathbb R} \int_{\mathbb R} \int_{\mathbb R_+}
\mathcal K(x_1, x_2, y_1, y_2, t) \eta(x_1, x_2, t) dx_2 dx_1 dt\right|^2  dy_2 dy_1 \right)^{\frac 12}.
\nn
\end{equation}

Our goal is to estimate $M$. Note that
\begin{equation}
M^2 
=
\frac{1}{(2\pi)^4}
\int_{\mathbb R} \int_{\mathbb R} \int_{\mathbb R}  \int_{\mathbb R_+} \int_{\mathbb R}  \int_{\mathbb R_+} 
\eta(x_1, x_2, t) \overline{\eta(z_1, z_2, t')} 
N(x_1,x_2, z_1, z_2, t, t')  dz_2dz_1 dx_2 dx_1 dt' dt
\label{M-temp1}
\end{equation}
where
$$
\begin{aligned}
N(x_1,x_2, z_1, z_2, t, t')
&:=
\int_{\mathbb R} \int_{\mathbb R}  \mathcal K(x_1, x_2, y_1, y_2, t) \overline{\mathcal K(z_1,z_2,y_1,y_2,t')} dy_2dy_1
\\
&=
\int_{\mathbb R}
\int_{\mathbb R_+} e^{-i\lambda_1 z_1-\lambda_2 z_2 + i(\lambda_1^2-\lambda_2^2)t'}
\int_{\mathbb R} \int_{\mathbb R} e^{i\lambda_1y_1+i\lambda_2 y_2}
\mathcal K(x_1, x_2, y_1, y_2, t) dy_2 dy_1 d\lambda_2 d\lambda_1
\end{aligned}
$$
Hence, observing that $\mathcal K$ can be regarded as the Fourier transform 
$$
\mathcal K(x_1, x_2, y_1, y_2, t) = \mathcal F_{k_1, k_2}\left\{\mathcal H(k_2) e^{ik_1x_1-k_2x_2-i(k_1^2-k_2^2)t} \right\}(y_1, y_2)
$$
with $\mathcal H(k_2)$ denoting the Heaviside function, we use the inverse Fourier transform to write $N$ as
$$
\begin{aligned}
N(x_1,x_2, z_1, z_2, t, t')
&=
(2\pi)^2 
\int_{\mathbb R}
\int_{\mathbb R_+} e^{-i\lambda_1 z_1-\lambda_2 z_2 + i(\lambda_1^2-\lambda_2^2)t'}
e^{i\lambda_1x_1-\lambda_2x_2-i(\lambda_1^2-\lambda_2^2)t} d\lambda_2 d\lambda_1
\\
&=
(2\pi)^2 N_1(x_1-z_1, t-t') N_2(x_2+z_2,  t-t')
\end{aligned}
$$
where
\begin{equation}\label{IJ-def-strich}
N_1(x_1, t) := \int_{\mathbb R} e^{i\lambda_1x_1-i\lambda_1^2 t} d\lambda_1 = e^{-i \frac \pi 4 \text{sgn}(t)} \sqrt \pi \, 
e^{i\frac{x_1^2}{4t}} |t|^{-\frac12},
\quad
N_2(x_2, t) := \int_{\mathbb R_+} e^{-\lambda_2 x_2 + i\lambda_2^2 t} d\lambda_2.
\end{equation}
In view of this rearrangement of $N$, using H\"older's inequality in $x_1, x_2$ (with exponents $p, p'$) and $t$ (with exponents $q, q'$) in the expression  \eqref{M-temp1} for $M^2$, we have
\begin{equation}\label{Q-0}
M^2
\leq
\frac{1}{(2\pi)^2} 
\bigg\|
\int_{\mathbb R} Q_p(t, t') dt'
\bigg\|_{L_t^q(\mathbb R)}
\left\| \eta \right\|_{L_t^{q'}(\mathbb R; L_{x_1, x_2}^{p'}(\mathbb R\times \mathbb R_+))}
\end{equation}
where
$$
Q_p(t, t') := 
\bigg\|
\int_{\mathbb R}  \int_{\mathbb R_+} 
N_1(x_1-z_1, t-t') N_2(x_2+z_2,  t-t') \overline{\eta(z_1, z_2, t')} 
dz_2 dz_1
\bigg\|_{L_{x_1, x_2}^p(\mathbb R \times \mathbb R_+)}.
$$

Thus, our goal becomes to estimate $Q_p$. As in the proof of the classical Strichartz estimates, we first handle the cases $p=2$ and $p=\infty$ and then interpolate between the two resulting estimates in order to cover the entire range $2\leq p \leq \infty$.
For $p=2$, combining the definition of $N_1$ with  Plancherel's theorem (twice), we find
\begin{align}
Q_2^2(t, t')
&=
\int_{\mathbb R_+} \int_{\mathbb R} 
\bigg|
\int_{\mathbb R} 
e^{i\lambda_1x_1-i\lambda_1^2(t-t')}
\int_{\mathbb R} e^{-i\lambda_1z_1} \int_{\mathbb R_+} 
 N_2(x_2+z_2,  t-t') \overline{\eta(z_1, z_2, t')} 
dz_2 dz_1
d\lambda_1
\bigg|^2 dx_1 dx_2
\nn\\
&=
\frac{1}{2\pi}
\int_{\mathbb R_+} \int_{\mathbb R} 
\bigg|
\int_{\mathbb R} e^{-i\lambda_1z_1} \int_{\mathbb R_+} 
 N_2(x_2+z_2,  t-t') \overline{\eta(z_1, z_2, t')} 
dz_2 dz_1
\bigg|^2 d\lambda_1 dx_2
\nn\\
&=
\int_{\mathbb R_+} \int_{\mathbb R} 
\bigg|
\int_{\mathbb R_+} 
N_2(x_2+z_2,  t-t') \overline{\eta(z_1, z_2, t')} 
dz_2 
\bigg|^2 dz_1 dx_2.
\nn
\end{align}
Hence, substituting for $N_2$ via \eqref{IJ-def-strich} and using (twice) the boundedness of the Laplace transform in $L^2(\mathbb R_+)$ (Lemma~3.2 in \cite{fhm2017}), we obtain
\begin{equation}
\begin{aligned}
Q_2^2(t, t')
&= 
\int_{\mathbb R}
\int_{\mathbb R_+}
\bigg|
\int_{\mathbb R_+} e^{-\lambda_2 x_2}
\int_{\mathbb R_+} 
e^{-\lambda_2 z_2+i\lambda_2^2 (t-t')} 
 \overline{\eta(z_1, z_2, t')} 
dz_2   d\lambda_2
\bigg|^2   dx_2 dz_1
\\
&\leq
\pi \int_{\mathbb R}
\int_{\mathbb R_+}
\bigg|
\int_{\mathbb R_+} 
e^{-\lambda_2 z_2}
e^{i\lambda_2^2 (t-t')}  \overline{\eta(z_1, z_2, t')}  dz_2   
\bigg|^2   d\lambda_2 dz_1
\leq
\pi^2 \left\|\eta(t')\right\|_{L_{x_1, x_2}^2(\mathbb R \times \mathbb R_+)}^2.
\label{Q2-est}
\end{aligned}
\end{equation}
For $p=\infty$, we treat the $z_1$ integral in $Q_\infty$ as a convolution and employ Young's convolution inequality to find
\begin{align}
Q_\infty(t, t')
&=
\sup_{x_2 \in \mathbb R_+}
\bigg\|
\bigg(
N_1(\cdot, t-t') *
\int_{\mathbb R_+} 
N_2(x_2+z_2,  t-t') \overline{\eta(\cdot, z_2, t')} 
dz_2
\bigg) (x_1)  
\bigg\|_{L_{x_1}^\infty(\mathbb R)}
\nn\\
&\leq
\left\|N_1(t-t')\right\|_{L_{x_1}^\infty(\mathbb R)}
\sup_{x_2\in\mathbb R_+} 
\left\|
\int_{\mathbb R_+} N_2(x_2+z_2,  t-t') \overline{\eta(x_1, z_2, t')} 
dz_2
\right\|_{L_{x_1}^1(\mathbb R)}
\nn
\end{align}
so that in view of the explicit formula for $N_1$ in \eqref{IJ-def-strich} we obtain
\begin{equation}
Q_\infty(t, t')
\leq
\sqrt \pi 
\left|t-t'\right|^{-\frac 12}
\sup_{x_2, z_2\in\mathbb R_+} \left|N_2(x_2+z_2,  t-t')\right|
\left\|\eta(t')\right\|_{L_{x_1, x_2}^1(\mathbb R \times \mathbb R_+)}.
\nn
\end{equation}
By van der Corput's lemma (see page 334 of \cite{s1993}), we have
$$
\bigg|\int_0^b e^{-\lambda_2 x_2 + i\lambda_2^2 t} d\lambda_2\bigg|
\leq
c \left|t\right|^{-\frac 12} \bigg(\left|e^{-b x_2}\right| + \int_0^b \left|\p_{\lambda_2} e^{-\lambda_2 x_2}\right| d\lambda_2 \bigg)
=
c \left|t\right|^{-\frac 12}, \quad x_2 \in \mathbb R_+,
$$
where $c$ is a constant \textit{independent} of $x_2, t, b$. Taking the limit of this inequality as $b\to\infty$, we deduce (after recalling~\eqref{IJ-def-strich}) that $\left|N_2(x_2+z_2, t-t')\right| \leq c\left|t-t'\right|^{-\frac 12}$ with $c$ \textit{independent} of $x_2, t$. Therefore, 
\begin{equation}\label{Qinf-est}
Q_\infty(t, t')
\leq
c \sqrt \pi \left|t-t'\right|^{-1}\left\|\eta(t')\right\|_{L_{x_1, x_2}^1(\mathbb R \times \mathbb R_+)}.
\end{equation}
Thanks to the Riesz-Thorin interpolation theorem, estimates \eqref{Q2-est} and \eqref{Qinf-est} imply
\begin{equation}\label{Q-interp-est}
Q_p(t, t')
\lesssim
\left|t-t'\right|^{\frac 2p-1} \left\|\eta(t')\right\|_{L_{x_1, x_2}^{p'}(\mathbb R \times \mathbb R_+)}, 
\quad
2\leq p \leq \infty.
\end{equation}

Having estimated $Q_p$, we return to the estimation of $M^2$. 
For $2<p< \infty$, we combine~\eqref{Q-0} with~\eqref{Q-interp-est} and then employ the Hardy-Littlewood-Sobolev inequality (Theorem 2.6 in \cite{lp2009}) to obtain
\begin{align} 
M^2
&\lesssim
\Big\|
\int_{\mathbb R}
\left|t-t'\right|^{\frac 2p-1} \left\|\eta(t')\right\|_{L_{x_1, x_2}^{p'}(\mathbb R \times \mathbb R_+)}
dt'
\Big\|_{L_t^q(\mathbb R)}
\left\| \eta \right\|_{L_t^{q'}(\mathbb R; L_{x_1, x_2}^{p'}(\mathbb R\times \mathbb R_+))}
\nn\\
&\lesssim
\left\|\eta\right\|_{L_t^r(\mathbb R; L_{x_1, x_2}^{p'}(\mathbb R \times \mathbb R_+))}
\left\| \eta \right\|_{L_t^{q'}(\mathbb R; L_{x_1, x_2}^{p'}(\mathbb R\times \mathbb R_+))}
\nn
\end{align}
where $\frac 1r = \frac 1q + \frac 2p$. Note that this step is only valid for $2<p<\infty$ and not for $p= 2, \infty$.  Then, thanks to the admissibility condition \eqref{adm-cond}, we have that $r=q'$ and so $M^2 \lesssim \left\| \eta \right\|_{L_t^{q'}(\mathbb R; L_{x_1, x_2}^{p'}(\mathbb R\times \mathbb R_+))}^2$. 

The case $p=2$ can be handled separately. Indeed, in that case, \eqref{adm-cond} implies $q=\infty$ so we combine \eqref{Q-0} with~\eqref{Q2-est}  to infer $M^2 \leq \frac{1}{4\pi} \left\| \eta \right\|_{L_t^\infty(\mathbb R; L_{x_1, x_2}^2(\mathbb R\times \mathbb R_+))}^2$. On the other hand, the case $p=\infty$ must be excluded, hence the relevant restriction in \eqref{adm-cond}.

Overall, for any admissible pair \eqref{adm-cond}, we have shown that $M\lesssim \left\| \eta \right\|_{L_t^{q'}(\mathbb R; L_{x_1, x_2}^{p'}(\mathbb R\times \mathbb R_+))}$. Thus, returning to~\eqref{M-temp4},  
$$
\left\|\p_{x_1}^{j_1} \p_{x_2}^{j-j_1} K_{\pm 1, i} \right\|_{L_t^q(\mathbb R; L_{x_1, x_2}^p(\mathbb R \times \mathbb R_+))}
\lesssim
\no{\p_{y_1}^{j_1} \p_{y_2}^{j-j_1}  \psi_\pm}_{L^2(\mathbb R^2)}.
$$
Combining this estimate with \eqref{hsp-eq2} and recalling the equivalence \eqref{hsp-eq}, we conclude 
\begin{equation}\label{k1i-strich-int}
\no{K_{\pm 1, i}}_{L_t^q((0, T); H_x^{s, p}(\mathbb R_+^2))}
\lesssim
\no{\psi_\pm}_{H^s(\mathbb R^2)}, \quad s\in\mathbb N\cup \{0\}.
\end{equation}
By interpolation (e.g. see Theorem 5.1 in \cite{lm1972}), the validity of estimate \eqref{k1i-strich-int} can be extended to all $s\geq 0$. In addition, by Plancherel's theorem and the definition \eqref{psipm-def} of $\psi_\pm$, we have
$$
\no{\psi_\pm}_{H_x^s(\mathbb R^2)}
\simeq
\no{\left(1+k_1^2+k_2^2\right)^{\frac s2}
2k_2 \widetilde  g_0(\pm k_1, k_1^2-k_2^2, T)}_{L_{k_1, k_2}^2(\mathbb R\times\mathbb R_+)}.
$$
Therefore, analogously to \eqref{kr-temp1}-\eqref{k1r-hs}, we obtain
\begin{equation}\label{k1i-strich}
\no{K_{\pm 1, i}}_{L_t^q((0, T); H_x^{s, p}(\mathbb R_+^2))}
\lesssim
\no{g_0}_{\mathcal X_{1, T}^{s, \frac 14}} + \no{g_0}_{\mathcal X_{1, T}^{0, \frac{2s+1}{4}}}, \quad 0\leq s < \frac 12.
\end{equation}

Combining \eqref{k1r-strich} and \eqref{k1i-strich} with their analogues for $K_{\pm 2, i}$ and the previously derived estimates \eqref{i12-strich} and~\eqref{j12-strich}  allows us to complete the proof of Theorem \ref{strich-est-t}.
\end{proof}

\section{Nonlinear analysis: proof of Theorem \ref{lwp-t}}
\label{lwp-s}

In this section, we prove the Hadamard well-posedness of the NLS equation on the spatial quarter-plane~\eqref{qnls-ibvp} as stated in Theorem \ref{lwp-t}.
Having established the key linear estimate \eqref{fls-le}, our goal is to replace the forcing $f$ by the NLS nonlinearity $\pm |u|^{\alpha-1} u$ in the linear solution operator $S\big[u_0, g_0, h_0; f\big]$ and then use estimate~\eqref{fls-le} together with appropriate nonlinear estimates (see below) in order to prove that, for all $0\leq s < \frac 12$ and all $\alpha>1$, the iteration map 
\begin{equation}\label{it-map}
u \mapsto \Phi(u) := S\big[u_0, g_0, h_0; \pm |u|^{\alpha-1} u\big]
\end{equation}
is a contraction in (a subset of) the solution space
$$
\mathcal S := C_t([0, T]; H_x^s(\mathbb R_+^2)) \cap L_t^q((0, T); H_x^{s, p}(\mathbb R_+^2)),
$$
where the lifespan $T>0$ and the admissible pair $(q, p)$ are yet to be determined. 
Since $H^s(\mathbb R_+^2)$ is not an algebra in the low-regularity setting $0\leq s < \frac 12$, the norm $\no{|u|^{\alpha-1} u}_{L_t^1((0, T); H_x^s(\mathbb R_+^2))}$ that arises from estimate \eqref{fls-le}, as well as its counterpart involving the difference of nonlinearities, will be handled via the following result.
\begin{proposition}\label{qnls-non-p}
Suppose $0\leq s < 1$. Then, for 
\begin{equation}\label{ds-est-cond-t}
\quad p=\frac{2\alpha}{1+(\alpha-1)s}, \quad q = \frac{2\alpha}{(1-s)(\alpha-1)}, 
\quad
2\leq \alpha \leq \frac{3-s}{1-s}
\end{equation}
we have the estimates
\begin{align}\label{non-est}
\left\| |\varphi|^{\alpha-1} \varphi \right\|_{L_t^1((0, T); H_x^s(\mathbb R^2))}
&\lesssim 
{T}^{\frac{q-\alpha}{q}} \left\| \varphi \right\|_{L_t^q((0, T); H_x^{s, p}(\mathbb R^2))}^\alpha,
\\
\label{non-diff-est}
\left\| |\varphi|^{\alpha-1} \varphi - |\psi|^{\alpha-1} \psi \right\|_{L_t^1((0, T); H_x^s(\mathbb R^2))}
&\lesssim
T^{\frac{q-\alpha}{q}}
\left(
\left\| \varphi \right\|_{L_t^q((0, T); H_x^{s, p}(\mathbb R^2))}^{\alpha-1}
+
\left\| \psi \right\|_{L_t^q((0, T); H_x^{s, p}(\mathbb R^2))}^{\alpha-1}
\right) 
\nn\\
&\quad
\cdot
\left\| \varphi - \psi \right\|_{L_t^q((0, T); H_x^{s, p}(\mathbb R^2))}.
\end{align}
\end{proposition}
The techniques needed for proving Proposition \ref{qnls-non-p} are standard in the NLS literature. Nevertheless, for the sake of completeness, we provide an outline of the proof at the end of the current section. Prior to this, we employ Proposition \ref{qnls-non-p} to establish that the map \eqref{it-map} is indeed a contraction in an appropriate subset of the space $\mathcal S$.
For each $t\in [0, T]$,  given $u(t) \in H_x^{s, p}(\mathbb R_+^2)$ there is an extension $\varphi(t) \in H_x^{s, p}(\mathbb R^2)$  such that 
\begin{equation}\label{ext-of-u}
\left\|\varphi(t) \right\|_{H_x^{s, p}(\mathbb R^2)}
\leq
2 \left\| u(t) \right\|_{H_x^{s, p}(\mathbb R_+^2)}.
\end{equation}
Such an extension is guaranteed by the infimum approximation property (see also Theorem 5.22 and Remark~5.23 in \cite{af2003}). Moreover, since $\left\| u(t) \right\|_{H^{s, p}(\mathbb R_+^2)} \leq \left\| U(t) \right\|_{H^{s, p}(\mathbb R^2)}$ for any extension $U$ of $u$, we actually have
$$
\left\| \varphi(t) \right\|_{H^{s, p}(\mathbb R^2)} \simeq \left\| u(t) \right\|_{H^{s, p}(\mathbb R_+^2)}.
$$
In addition, since $\varphi = u$ almost everywhere on $\mathbb R_+^2$,  
$$
\left\| |u|^{\alpha-1} u (t) \right\|_{H_x^s(\mathbb R_+^2)}
=
\left\| |\varphi|^{\alpha-1} \varphi(t) \right\|_{H_x^s(\mathbb R_+^2)}
\leq
\left\| |\varphi|^{\alpha-1} \varphi(t) \right\|_{H_x^s(\mathbb R^2)}.
$$
Hence, combining the nonlinear estimate \eqref{non-est}   with the extension inequality \eqref{ext-of-u}, we obtain
\begin{equation}\label{non-est-qp}
\begin{aligned}
\left\| |u|^{\alpha-1} u\right\|_{L_t^1((0, T); H_x^s(\mathbb R_+^2))} 
&\lesssim
T^{\frac{q-\alpha}{q}} \left\| \varphi \right\|_{L_t^q((0, T); H_x^{s, p}(\mathbb R^2))}^\alpha
\lesssim
T^{\frac{q-\alpha}{q}} \left\| u \right\|_{L_t^q((0, T); H_x^{s, p}(\mathbb R_+^2))}^\alpha.
\end{aligned}
\end{equation}
Similarly, given $v\in L_t^1((0, T); H_x^{s, p}(\mathbb R_+^2))$ and letting $\psi \in L_t^1((0, T); H_x^{s, p}(\mathbb R^2))$ be an extension of  $v$ satisfying \eqref{ext-of-u},  the nonlinear estimate~\eqref{non-diff-est} yields
\begin{equation}\label{non-diff-est-qp}
\begin{aligned}
\left\| |u|^{\alpha-1} u - |v|^{\alpha-1} v   \right\|_{L_t^1((0, T); H_x^s(\mathbb R_+^2))}
&\lesssim
T^{\frac{q-\alpha}{q}}
\left(
\left\| u \right\|_{L_t^q((0, T); H_x^{s, p}(\mathbb R_+^2))}^{\alpha-1}
+
\left\| v \right\|_{L_t^q((0, T); H_x^{s, p}(\mathbb R_+^2))}^{\alpha-1}
\right)
\\
&\quad
\cdot
\left\| u - v \right\|_{L_t^q((0, T); H_x^{s, p}(\mathbb R_+^2))}.
\end{aligned}
\end{equation}

Suppose that $0\leq s < \frac 12$, $2\leq \alpha \leq  \frac{3-s}{1-s}$, and $q, p$ are given by \eqref{qp-t}. Then, combining the linear estimate~\eqref{fls-le} with the nonlinear estimate~\eqref{non-est-qp} (which can be employed thanks to our choice of $q, p, \alpha$), we have
\begin{equation}\label{into-1}
\begin{aligned}
\no{S\big[u_0, g_0, h_0; \pm |u|^{\alpha-1} u\big]}_{\mathcal S}
\leq
c_{s, \alpha} \Big(&\no{u_0}_{H^s(\mathbb R_+^2)} + \no{g_0}_{\mathcal X_{1, T}^{0, \frac{2s+1}{4}}} + \no{g_0}_{\mathcal X_{1, T}^{s, \frac 14}} 
\\
&+ \no{h_0}_{\mathcal X_{2, T}^{0, \frac{2s+1}{4}}} + \no{h_0}_{\mathcal X_{2, T}^{s, \frac 14}}
+ T^{\frac{3-s-\alpha(1-s)}{2}} \no{u}_{L_t^q((0, T); H_x^{s, p}(\mathbb R_+^2))}^\alpha
\Big).
\end{aligned}
\end{equation}
Let $\mathcal B(0, \varrho) := \left\{u\in \mathcal S: \no{u}_{\mathcal S} \leq \varrho\right\}$ denote the closed ball in $\mathcal S$ with center at the origin and radius
\begin{equation}
\varrho 
=
\varrho(T)
:=  
2c_{s, \alpha} \Big(\no{u_0}_{H^s(\mathbb R_+^2)} + \no{g_0}_{\mathcal X_{1, T}^{0, \frac{2s+1}{4}}} + \no{g_0}_{\mathcal X_{1, T}^{s, \frac 14}} 
+ \no{h_0}_{\mathcal X_{2, T}^{0, \frac{2s+1}{4}}} + \no{h_0}_{\mathcal X_{2, T}^{s, \frac 14}}\Big).
\end{equation}
Then, by estimate \eqref{into-1}, any $u\in \mathcal B(0, \varrho)$ satisfies
$$
\no{S\big[u_0, g_0, h_0; \pm |u|^{\alpha-1} u\big]}_{\mathcal S}
\leq
\frac{\varrho}{2} + c_{s, \alpha}  T^{\frac{3-s-\alpha(1-s)}{2}} \varrho^\alpha.
$$
Thus, if
\begin{equation}\label{into-cond}
2 c_{s, \alpha} \varrho^{\alpha-1} T^{\frac{3-s-\alpha(1-s)}{2}} \leq 1
\end{equation}
then the map $u \mapsto S\big[u_0, g_0, h_0; \pm |u|^{\alpha-1} u\big]$ takes $\mathcal B(0, \varrho)$ into itself. 
For $2\leq \alpha < \frac{3-s}{1-s}$, the exponent $\frac{3-s-\alpha(1-s)}{2}$ of $T$ is positive and also the radius $\varrho$ remains bounded as $T\to 0^+$. Hence, for $2\leq \alpha < \frac{3-s}{1-s}$ there exists a $T>0$ such that the condition~\eqref{into-cond} can be fulfilled without any restrictions on (the size of) the initial data.
On the other hand, for $\alpha = \frac{3-s}{1-s}$ the exponent of $T$ vanishes and so the condition~\eqref{into-cond}  is satisfied only if $\varrho$ and, in particular, the size of the initial data (since the size of the boundary data is controlled by $T$) is sufficiently small.

Furthermore, for any $u, v \in \mathcal B(0, \varrho)$, the linear estimate \eqref{fls-le} and the nonlinear estimate \eqref{non-diff-est-qp} imply
\begin{equation}\label{contr-1}
\begin{aligned}
&\quad
\no{S\big[u_0, g_0, h_0; \pm |u|^{\alpha-1} u\big]-S\big[u_0, g_0, h_0; \pm |v|^{\alpha-1} v\big]}_{\mathcal S}
=
\no{S\big[0, 0, 0; \pm \left(|u|^{\alpha-1} u - |v|^{\alpha-1} v\right)\big]}_{\mathcal S}
\\
&\leq
c_{s, \alpha} T^{\frac{3-s-\alpha(1-s)}{2}} 
\left(
\left\| u \right\|_{L_t^q((0, T); H_x^{s, p}(\mathbb R_+^2))}^{\alpha-1}
+
\left\| v \right\|_{L_t^q((0, T); H_x^{s, p}(\mathbb R_+^2))}^{\alpha-1}
\right)
\left\| u - v \right\|_{L_t^q((0, T); H_x^{s, p}(\mathbb R_+^2))}
\\
&\leq
c_{s, \alpha} T^{\frac{3-s-\alpha(1-s)}{2}} \cdot 2\varrho^{\alpha-1}
\left\| u - v \right\|_{\mathcal S}.
\end{aligned}
\end{equation}
Hence, if
\begin{equation}\label{contr-cond}
2 c_{s, \alpha} \varrho^{\alpha-1} T^{\frac{3-s-\alpha(1-s)}{2}} < 1
\end{equation}
then the map $u \mapsto S\big[u_0, g_0, h_0; \pm |u|^{\alpha-1} u\big]$ is a contraction on $\mathcal B(0, \varrho)$. In turn, by the Banach fixed point theorem, this map has a unique fixed point in $\mathcal B(0, \varrho)$. Equivalently, if \eqref{contr-cond} is satisfied, then the integral equation $u = S\big[u_0, g_0, h_0; \pm |u|^{\alpha-1} u\big]$ has a unique solution in $\mathcal B(0, \varrho)$, which provides the unique solution to the NLS initial-boundary value problem \eqref{qnls-ibvp}. Like for the ``into'' condition \eqref{into-cond}, we emphasize that for $2\leq \alpha < \frac{3-s}{1-s}$ the ``contraction'' condition \eqref{contr-cond} can be fulfilled by a suitable $T>0$ without any additional restriction on the initial data due to the fact that the relevant exponent of $T$ is positive, $\varrho$ is bounded as $T\to 0^+$, and $\alpha-1>0$. On the other hand, for $\alpha = \frac{3-s}{1-s}$ the condition \eqref{contr-cond} requires a sufficiently small norm for the initial data in order to make $\varrho$ sufficiently small. 

The above contraction mapping argument establishes existence and uniqueness of solution inside the ball $\mathcal B(0, \varrho) \subseteq \mathcal S$. However, following the method of Proposition 4.2 in \cite{cw1990} (see also Section 3.3 in \cite{amo2024}), uniqueness can actually be extended to the whole of the solution space $\mathcal S$. Finally, the Lipschitz continuity of the data-to-solution map follows by arguments similar to those used for the contraction above, e.g. see \cite{fhm2017,hm2020}. The proof of Theorem \ref{lwp-t} is complete.

\vspace*{3mm}
\noindent
\textit{Proof of Proposition \ref{qnls-non-p}.} 
As noted earlier, the proof is standard in the NLS literature. We begin with the following crucial result.
\begin{proposition}\label{prod-est-low-p}
Let $D^s$ be defined by $\mathcal F\{D^s \varphi\}(k) = |k|^s \mathcal F\{\varphi\}(k)$. 
Then, for all $0<s<1$ and $\alpha\geq 1$, 
\begin{equation}\label{prod-est-low}
\left\| D^s (|\varphi|^{\alpha-1}\varphi) \right\|_{L^2(\mathbb R^2)}
\lesssim
\left\| \varphi \right\|_{H^{s, p}(\mathbb R^2)}^\alpha
\end{equation}
where $p = \frac{2 \alpha}{1+(\alpha-1)s}$. 
\end{proposition}

\begin{proof}
The proof of Proposition \ref{prod-est-low-p} relies on fractional versions of the chain rule  and the Sobolev-Gagliardo-Nirenberg inequality. A chain rule for fractional derivatives that can be found in Proposition 3.1 of \cite{cw1991} (see also Lemma 2.3 in  \cite{kvz2008} and Proposition 5.1 in \cite{t2000}) is the following.
\begin{lemma}[Chain rule for fractional derivatives]
\label{frac-chain-l}
Suppose that $F\in C^1(\mathbb C)$, $0<s < 1$ and $1< r, r_1, r_2<\infty$ such that $\frac 1r = \frac{1}{r_1}+\frac{1}{r_2}$. Then, 
\begin{equation}\label{2d-frac-chain}
\left\| D^s F(\varphi) \right\|_{L^r(\mathbb R^2)} \lesssim \left\| F'(\varphi) \right\|_{L^{r_1}(\mathbb R^2)} \left\| D^s \varphi \right\|_{L^{r_2}(\mathbb R^2)}
\end{equation}
where $F'(\varphi)$ denotes the Jacobian matrix of $F(\varphi)$ with respect to $\textnormal{Re}(\varphi)$ and $\textnormal{Im}(\varphi)$.
\end{lemma}
We note that, although the proof of Lemma \ref{frac-chain-l} given in \cite{cw1991} is in one dimension, it is noted that the result remains valid in higher dimensions (see also its statement in \cite{kvz2008}). 
A generalization of the classical Sobolev-Gagliardo-Nirenberg inequality (Theorem 1.3.7 in \cite{c2003})  to fractional derivatives is provided by Corollary 2.4 of~\cite{hmow2011} and reads as follows. 
\begin{lemma}[Fractional Gagliardo-Nirenberg inequality]
\label{frac-sgn-l}
Let $1<r, r_0, r_1<\infty$, $\sigma, \sigma_1\in\mathbb R$ and $0\leq \theta\leq 1$. Then, the fractional Gagliardo-Nirenberg inequality of the type
\begin{equation}
\left\| \varphi \right\|_{\dot H^{\sigma, r}(\mathbb R^2)} \lesssim \left\| \varphi \right\|_{L^{r_0}(\mathbb R^2)}^{1-\theta} \left\| \varphi \right\|_{\dot H^{\sigma_1, r_1}(\mathbb R^2)}^\theta
\end{equation}
holds if and only if  $\frac 2r - \sigma = \left(1-\theta\right) \frac{2}{r_0} + \theta  \big(\frac{2}{r_1} - \sigma_1\big)$, $\sigma \leq \theta \sigma_1$.
\end{lemma}
In the above lemma, $\dot H^{s, p}(\mathbb R^2)$ denotes the Riesz potential space defined by
\begin{equation}\label{hsp-dot-def}
\dot H^{s, p}(\mathbb R^2) 
:=
\left\{\varphi \in L^r(\mathbb R^2): 
\left\| \varphi \right\|_{\dot H^{s, p}(\mathbb R^2)} := \left\| D^s\varphi \right\|_{L^p(\mathbb R^2)} =  \left\| \mathcal F^{-1}\left\{ |k|^s \mathcal F\{\varphi\}(k)\right\} \right\|_{L^p(\mathbb R^2)} < \infty
\right\}.
\end{equation}
This is the nonhomogeneous analogue of the Bessel potential space $H^{s, p}(\mathbb R^2)$ and for $p=2$ becomes the homogeneous   Sobolev space $\dot H^s(\mathbb R^2)$.
We now proceed to the proof of the nonlinear estimate \eqref{prod-est-low}. Letting  $F(\varphi) = |\varphi|^{\alpha-1} \varphi$ and writing $\varphi =  \text{Re}(\varphi) + i  \text{Im}(\varphi)$, we find that the Jacobian matrix $F'(\varphi)$ is continuous if and only if $\alpha\geq 1$. Hence, using Lemma \ref{frac-chain-l} with $r=2$ and $r_2=p$, 
we have
$$
\left\| D^s F(\varphi) \right\|_{L^2(\mathbb R^2)} \lesssim \left\| F'(\varphi) \right\|_{L^{\frac{2p}{p-2}}(\mathbb R^2)} \left\| D^s \varphi \right\|_{L^p(\mathbb R^2)}
\lesssim \left\| \varphi  \right\|_{L^{\frac{2p(\alpha-1)}{p-2}}(\mathbb R^2)}^{\alpha-1} \left\| D^s \varphi \right\|_{L^p(\mathbb R^2)}.
$$
In turn, employing Lemma \ref{frac-sgn-l} with $\sigma=0$, $\sigma_1=s$, $r=\frac{2p(\alpha-1)}{p-2}$, $r_1=p$ and $\theta=1$, which impose the condition
\begin{equation}\label{p-cond}
p  = \frac{2 \alpha}{1+(\alpha-1)s},
\end{equation}
(the above choices are compatible with the requirements $1<r, r_1 < \infty$ and $\sigma \leq \theta \sigma_1$) we obtain
$$
\left\| D^s F(\varphi) \right\|_{L^2(\mathbb R^2)} 
\lesssim 
\left\| \varphi  \right\|_{\dot H^{s, p}(\mathbb R^2)}^\alpha
\leq
\left\| \varphi  \right\|_{H^{s, p}(\mathbb R^2)}^\alpha
$$
with $0<s<1$ due to Lemma \ref{frac-chain-l}. 
\end{proof}

Having proved Proposition \ref{prod-est-low-p}, we proceed to the proof of the first half of Proposition \ref{qnls-non-p}, namely of the nonlinear estimate \eqref{non-est}. Restoring dependence on $t$ in \eqref{prod-est-low}, 
taking $L_t^1(0, T)$ norms and using H\"older's inequality in $t$ with $r\geq 1$ to be determined, we have
\begin{equation}\label{frac-chain-F6}
\left\| D^s (|\varphi|^{\alpha-1}\varphi) \right\|_{L_t^1((0, T); L_x^2(\mathbb R^2))} 
\lesssim 
T^{\frac{1}{r'}}
\left\| \varphi \right\|_{L_t^{\alpha r}((0, T); H_x^{s, p}(\mathbb R^2))}^\alpha.
\end{equation}
As our goal is to have the norm of the Strichartz space $L_t^qL_x^p$ appear on the right-hand side, we set $r=\tfrac{q}{\alpha}$ which, in view of \eqref{p-cond} and the admissibility conditions~\eqref{adm-cond}, implies
\begin{equation}\label{q-cond}
q = \frac{2 \alpha}{(1-s)(\alpha-1)}.
\end{equation}
Note that our choice of $r$ requires $q \geq \alpha$ or, equivalently,
\begin{equation}\label{alpha-cond}
\alpha  \leq \frac{3-s}{1-s}.
\end{equation}
Moreover, note that the requirement $2 \leq p < \infty$ imposed by \eqref{adm-cond} is satisfied since $0<s<1$ and  $(\alpha-1)s\geq 0$.
In turn, \eqref{frac-chain-F6} yields
\begin{equation}\label{ds-est}
\left\| D^s (|\varphi|^{\alpha-1}\varphi) \right\|_{L_t^1((0, T); L_x^2(\mathbb R^2))} 
\lesssim 
T^{\frac{q-\alpha}{q}} \left\| \varphi \right\|_{L_t^q((0, T); H_x^{s, p}(\mathbb R^2))}^\alpha.
\end{equation}

In the case $s=0$, which is not covered by the above arguments, we use H\"older's inequality in $x$ to infer
$$
\left\| |\varphi|^{\alpha-1}\varphi(t)  \right\|_{L_x^2(\mathbb R^2)}^2 
\leq
\left\| \varphi(t) \right\|_{L_x^p(\mathbb R^2)}^{2\alpha}
$$
provided that  $p=2\alpha$, which is in agreement with condition \eqref{p-cond} when $s=0$. Hence, by H\"older's inequality in $t$, 
$$
\left\| |\varphi|^{\alpha-1}\varphi  \right\|_{L_t^1((0, T); L_x^2(\mathbb R^2))}
\leq
T^{\frac{1}{r'}} \left(\int_0^T \left\| \varphi(t) \right\|_{L_x^p(\mathbb R^2)}^{\alpha r} dt\right)^{\frac{1}{r}}, \quad r\geq 1,
$$
and setting $r=\frac{q}{\alpha}$, which in view of \eqref{adm-cond} implies $q=\tfrac{2\alpha}{\alpha-1}$ (consistently with~\eqref{q-cond}) and requires that $q\geq \alpha$ or, equivalently, $\alpha \leq  3$ (consistently with \eqref{alpha-cond}), we have
\begin{equation}\label{non-l2-est}
\left\| |\varphi|^{\alpha-1}\varphi  \right\|_{L_t^1((0, T); L_x^2(\mathbb R^2))} 
\lesssim 
T^{\frac{q-\alpha}{q}} \left\| \varphi \right\|_{L_t^q((0, T); L_x^p(\mathbb R^2))}^\alpha.
\end{equation}
Note that the admissibility condition $2\leq p <\infty$ is met since  $1 \leq \alpha < \infty$.
Overall, estimates \eqref{ds-est} and \eqref{non-l2-est}  imply the desired estimate \eqref{non-est} of Proposition \ref{qnls-non-p}. We remark that the restriction $\alpha\geq 2$ actually comes from the proof of the difference estimate \eqref{non-diff-est} below.

Next, we establish the second half of Proposition \ref{qnls-non-p}, namely estimate \eqref{non-diff-est} for the difference of nonlinearities. We begin with the following well-known result (for a proof, see for example \cite{hkmms2024}).
\begin{lemma}\label{mvt-l}
For any pair of complex numbers $z, z'$,
\begin{equation}
|z|^{\alpha-1} z - |z'|^{\alpha-1} z' 
=
\frac{\alpha+1}{2} \left(\int_0^1  |Z_\lambda|^{\alpha-1} d\lambda\right) \left(z-z'\right)
+
\frac{\alpha-1}{2} \left(\int_0^1 |Z_\lambda|^{\alpha-3}
Z_\lambda^2 \, d\lambda\right)  \overline{\left(z-z'\right)},
\end{equation}
where $Z_\lambda := \lambda z + \left(1-\lambda\right) z'$, $\lambda\in [0, 1]$.
\end{lemma}
Employing Lemma \ref{mvt-l} with $z=\varphi$, $z'=\psi$ while suppressing for now the $t$ dependence, we get
\begin{align}
&\quad
\left\| D^s (|\varphi|^{\alpha-1} \varphi - |\psi|^{\alpha-1} \psi)  \right\|_{L_x^2(\mathbb R^2)}
\nn\\
&\lesssim
\sup_{\lambda\in [0, 1]} \left\| D^s \left[  |Z_\lambda|^{\alpha-1} \left(\varphi - \psi\right)\right]  \right\|_{L_x^2(\mathbb R^2)}
+
\sup_{\lambda\in [0, 1]}  \left\|  D^s\left[ |Z_\lambda|^{\alpha-3}
Z_\lambda^2 \,  \overline{\left(\varphi - \psi\right)}\right]  \right\|_{L_x^2(\mathbb R^2)}.
\label{2d-diff-temp5}
\end{align}
Proposition 3.3 in \cite{cw1991} (which is stated for $n=1$ but holds for all $n\in\mathbb N$ --- see also \cite{kp1988,bmmn2014}) provides the following Leibniz rule for fractional derivatives.
\begin{lemma}[Fractional Leibniz rule]
\label{prod-l}
Let $0<s<1$ and $1< r, r_1, r_2, \rho_1, \rho_2<\infty$ such that $\frac 1r = \frac{1}{r_1} + \frac{1}{\rho_1} = \frac{1}{r_2} + \frac{1}{\rho_2}$. Then, 
\begin{equation}
\left\| D^s\left(fg\right)\right\|_{L^r(\mathbb R^2)}
\lesssim
\left\| f \right\|_{L^{r_1}(\mathbb R^2)}
\left\| D^s g \right\|_{L^{\rho_1}(\mathbb R^2)}
+
\left\| D^s f \right\|_{L^{r_2}(\mathbb R^2)}
\left\| g \right\|_{L^{\rho_2}(\mathbb R^2)}.
\end{equation}
\end{lemma}
For the first term on the right-hand side of \eqref{2d-diff-temp5}, Lemma \ref{prod-l} with $f = |Z_\lambda|^{\alpha-1}$, $g=\varphi-\psi$, $r=2$, $\rho_1=p$ (which imply $r_1=\frac{2p}{p-2}$)  and $0<s<1$ yields
\begin{align}
\left\| D^s \left[  |Z_\lambda|^{\alpha-1} \left(\varphi - \psi\right)\right]  \right\|_{L_x^2(\mathbb R^2)} 
&\lesssim
\left\| |Z_\lambda|^{\alpha-1}\right\|_{L_x^{\frac{2p}{p-2}}(\mathbb R^2)}
\left\| D^s \left(\varphi - \psi\right) \right\|_{L_x^p(\mathbb R^2)} 
\nn\\
&\quad
+
\left\| D^s |Z_\lambda|^{\alpha-1} \right\|_{L_x^{r_2}(\mathbb R^2)} \left\| \varphi - \psi \right\|_{L_x^{\rho_2}(\mathbb R^2)}.
\label{diff-temp1}
\end{align}
In order to bound the $L^{\rho_2}$ norm by the $H^{s, p}$ norm on the right-hand side of \eqref{diff-temp1}, we use the fractional Gagliardo-Nirenberg inequality of Lemma \ref{frac-sgn-l} with $\sigma=0$, $r=\rho_2$, $\theta=1$, $\sigma_1=s$ and $r_1 = p$. In view of \eqref{p-cond}, this choice forces
$\rho_2 = \frac{2\alpha}{1-s}$
which via the condition $\frac 12 = \frac{1}{r_2} + \frac{1}{\rho_2}$ of Lemma \ref{prod-l}  implies  $r_2 = \frac{2\alpha}{\alpha-1 + s}$.
Note that the hypotheses of Lemmas  \ref{frac-sgn-l} and \ref{prod-l} are met since  $\alpha>1$ and $0<s<1$. With the above restrictions, Lemma \ref{frac-sgn-l} yields
\begin{equation}\label{embed}
 \left\| \varphi - \psi \right\|_{L_x^{\rho_2}(\mathbb R^2)}
 =
  \left\| \varphi - \psi \right\|_{L_x^{\frac{2\alpha}{1-s}}(\mathbb R^2)}
 \lesssim
  \left\| \varphi - \psi \right\|_{\dot H_x^{s, p} (\mathbb R^2)}
\end{equation}
so in view of \eqref{hsp-dot-def} inequality \eqref{diff-temp1} becomes
\begin{equation}\label{diff-temp2}
\left\| D^s \left[  |Z_\lambda|^{\alpha-1} (\varphi - \psi)\right]  \right\|_{L_x^2(\mathbb R^2)} 
\lesssim
\Big(
\left\| |Z_\lambda|^{\alpha-1}\right\|_{L_x^{\frac{2\alpha}{(\alpha-1)(1-s)}}(\mathbb R^2)}
+
\left\| D^s |Z_\lambda|^{\alpha-1} \right\|_{L_x^{\frac{2\alpha}{\alpha-1 + s}}(\mathbb R^2)}
\Big)
\left\| \varphi - \psi \right\|_{H_x^{s, p}(\mathbb R^2)}.
\end{equation}
By the triangle inequality, the fact that $\lambda\in [0, 1]$ and the embedding \eqref{embed}, 
\begin{equation}\label{sav1}
\left\| |Z_\lambda|^{\alpha-1}\right\|_{L_x^{\frac{2\alpha}{(\alpha-1)(1-s)}}(\mathbb R^2)}
\lesssim
\left\| \varphi \right\|_{H_x^{s, p}(\mathbb R^2)}^{\alpha-1}
+
\left\| \psi \right\|_{H_x^{s, p}(\mathbb R^2)}^{\alpha-1}.
\end{equation}
Furthermore,  
$$
\left\| D^s |Z_\lambda|^{\alpha-1} \right\|_{L_x^{\frac{2\alpha}{\alpha-1 + s}}(\mathbb R^2)}
\leq
\left\| D^s |\varphi|^{\alpha-1} \right\|_{L_x^{\frac{2\alpha}{\alpha-1 + s}}(\mathbb R^2)}
+
\left\| D^s |\psi|^{\alpha-1} \right\|_{L_x^{\frac{2\alpha}{\alpha-1 + s}}(\mathbb R^2)}.
$$
Thus, setting  $F(\varphi) = |\varphi|^{\alpha-1}$  and restricting $
\alpha > 2$ in order to have $F\in C^1(\mathbb C)$ ($\alpha=2$ can be handled directly via (13) on page 355 of \cite{rs1996}), we use the fractional chain rule of Lemma \ref{frac-chain-l} with $0<s<1$, $r=\frac{2\alpha}{\alpha-1 + s}$ (which satisfies $1<r<\infty$ by previous restrictions) and $r_2=p$ (which again satisfies $1<r<\infty$ by previous restrictions) to infer, after observing that $\left|F'(\varphi)\right| \simeq |\varphi|^{\alpha-2}$ and using once again the embedding \eqref{embed}, 
$$
\left\| D^s |Z_\lambda|^{\alpha-1} \right\|_{L_x^{\frac{2\alpha}{\alpha-1 + s}}(\mathbb R^2)}
\lesssim
\left\| \varphi \right\|_{H_x^{s, p}(\mathbb R^2)}^{\alpha-1}
+
\left\| \psi \right\|_{H_x^{s, p}(\mathbb R^2)}^{\alpha-1}.
$$
This estimate is combined with \eqref{sav1} and \eqref{diff-temp2} to yield
\begin{equation}\label{diff-temp6}
\left\| D^s \left[  |Z_\lambda|^{\alpha-1} (\varphi - \psi)\right]  \right\|_{L_x^2(\mathbb R^2)} 
\lesssim
\left(
\left\| \varphi \right\|_{H_x^{s, p}(\mathbb R^2)}^{\alpha-1}
+
\left\| \psi \right\|_{H_x^{s, p}(\mathbb R^2)}^{\alpha-1}
\right)
\left\| \varphi - \psi \right\|_{H_x^{s, p}(\mathbb R^2)},
\end{equation}
completing the estimation of the first term on the right-hand side of \eqref{2d-diff-temp5}. 

For the second term on the right-hand side of \eqref{2d-diff-temp5}, we employ Lemma \ref{prod-l} with $f=|Z_\lambda|^{\alpha-3} Z_\lambda^2$, $g=\overline{\varphi-\psi}$, $r=2$, $\rho_1 = p$ (which give $r_1=\frac{2p}{p-2} = \frac{2\alpha}{(\alpha-1)(1-s)}$) and $0<s<1$ to obtain
\begin{align}
&\quad
\left\| D^s\big[|Z_\lambda|^{\alpha-3} Z_\lambda^2 \overline{(\varphi-\psi)}\big]\right\|_{L_x^2(\mathbb R^2)}
\nn\\
&\lesssim
\left\| |Z_\lambda|^{\alpha-1} \right\|_{L_x^{\frac{2\alpha}{(\alpha-1)(1-s)}}(\mathbb R^2)}
\left\| D^s (\varphi-\psi) \right\|_{L_x^p(\mathbb R^2)}
+
\left\| D^s (|Z_\lambda|^{\alpha-3} Z_\lambda^2) \right\|_{L_x^{r_2}(\mathbb R^2)}
\left\| \varphi-\psi \right\|_{L_x^{\rho_2}(\mathbb R^2)}.
\end{align}
Employing the fractional Gagliardo-Nirenberg inequality of Lemma \ref{frac-sgn-l} with $\sigma=0$, $r=\rho_2$, $\theta=1$, $\sigma_1=s$ and $r_1=p$, we obtain once again the embedding \eqref{embed} and also $r_2 = \frac{2\alpha}{\alpha-1 + s}$. Thus, 
\begin{align}
&\quad
\left\| D^s\big[|Z_\lambda|^{\alpha-3} Z_\lambda^2 \overline{(\varphi-\psi)}\big]\right\|_{L_x^2(\mathbb R^2)}
\nn\\
&\lesssim
\Big(
\left\| |Z_\lambda|^{\alpha-1}  \right\|_{L_x^{\frac{2\alpha}{(\alpha-1)(1-s)}}(\mathbb R^2)}
+
\left\| D^s (|Z_\lambda|^{\alpha-3} Z_\lambda^2) \right\|_{L_x^{\frac{2\alpha}{\alpha-1 + s}}(\mathbb R^2)}
\Big)
\left\| \varphi-\psi \right\|_{H_x^{s, p}(\mathbb R^2)}.
\label{diff-temp7}
\end{align}
The first term in the above parenthesis is the same with the corresponding one in \eqref{diff-temp2} and hence admits the bound \eqref{sav1}. 
The second term can be handled by slightly adapting the proof of Proposition \ref{prod-est-low-p}, in particular, by setting $F(\phi) = |\phi|^{\alpha-3} \phi^2$ with $\phi = Z_\lambda = \lambda \varphi + (1-\lambda) \psi$ and noting that $F\in C^1(\mathbb C)$ if and only if $\alpha>2.$
Then, in view of the embedding \eqref{embed} and the fact that $|F'(\phi)|\lesssim |\phi|^{\alpha-2}$, the chain rule of Lemma \ref{frac-chain-l} with $r=\frac{2\alpha}{\alpha-1+s}$ and $r_2=p$ (which imply $r_1 = \frac{2\alpha}{(1-s)(\alpha-2)}$) yields
\begin{equation}\label{diff-temp8}
\left\| D^s (|Z_\lambda|^{\alpha-3} Z_\lambda^2) \right\|_{L_x^{\frac{2\alpha}{\alpha-1 + s}}(\mathbb R^2)}
\lesssim
\left\|\varphi\right\|_{H_x^{s, p}(\mathbb R^2)}^{\alpha-1}
+
\left\|\psi\right\|_{H_x^{s, p}(\mathbb R^2)}^{\alpha-1}.
\end{equation}
The bound \eqref{diff-temp8} has been established only for $\alpha>2$ since Lemma \ref{frac-chain-l} cannot be employed when $\alpha=2$. However, for $\alpha=2$ we can establish \eqref{diff-temp8} directly. We have
$$
\left\| D^s \left(\frac{Z_\lambda^2}{|Z_\lambda|}\right) \right\|_{L_x^{\frac{4}{1 + s}}(\mathbb R^2)}
\leq
\left\|\frac{Z_\lambda^2}{|Z_\lambda|} \right\|_{H_x^{s, \frac{4}{1 + s}}(\mathbb R^2)}
$$
so noting that $\frac{Z_\lambda^2}{|Z_\lambda|}$ is Lipschitz we may employ (2) on page 41 of \cite{rs1996} (recalling that the Triebel-Lizorkin space $F_{p, 2}^s$ is equal to $H^{s, p}$) to deduce
\begin{equation}\label{diff-temp13}
\left\| D^s \left(\frac{Z_\lambda^2}{|Z_\lambda|}\right) \right\|_{L_x^{\frac{4}{1 + s}}(\mathbb R^2)}
\leq
\left\| Z_\lambda \right\|_{H_x^{s, \frac{4}{1 + s}}(\mathbb R^2)}
\equiv
\left\| Z_\lambda \right\|_{H_x^{s, p}(\mathbb R^2)}
\lesssim
\left\|\varphi\right\|_{H_x^{s, p}(\mathbb R^2)} + \left\|\psi\right\|_{H_x^{s, p}(\mathbb R^2)}.
\end{equation}

Overall, inequalities \eqref{sav1}, \eqref{diff-temp7}, \eqref{diff-temp8} and \eqref{diff-temp13}  imply
$$
\left\| D^s\left[|Z_\lambda|^{\alpha-3} Z_\lambda^2 \overline{(\varphi-\psi)}\right]\right\|_{L_x^2(\mathbb R^2)}
\lesssim
\left(
\left\| \varphi \right\|_{H_x^{s, p}(\mathbb R^2)}^{\alpha-1}
+
\left\| \psi \right\|_{H_x^{s, p}(\mathbb R^2)}^{\alpha-1}
\right)
\left\| \varphi - \psi \right\|_{H_x^{s, p}(\mathbb R^2)}
$$
provided that $\alpha \geq 2$. This bound is combined with \eqref{diff-temp6} and \eqref{2d-diff-temp5} to yield  
\begin{equation}\label{diff-temp9}
\left\| D^s (|\varphi|^{\alpha-1} \varphi - |\psi|^{\alpha-1} \psi)  \right\|_{L_x^2(\mathbb R^2)}
\lesssim
\left(
\left\| \varphi \right\|_{H_x^{s, p}(\mathbb R^2)}^{\alpha-1}
+
\left\| \psi \right\|_{H_x^{s, p}(\mathbb R^2)}^{\alpha-1}
\right)
\left\| \varphi - \psi \right\|_{H_x^{s, p}(\mathbb R^2)}
\end{equation}
with $0<s< 1$, $\alpha\geq 2$ and $p = \frac{2\alpha}{1+(\alpha-1)s}$.

\vskip 3mm

Finally, for $s=0$ we recall the following standard result, which assumes the role of Lemma \ref{mvt-l} in this simpler case:
$$
\left| |\varphi|^{\alpha-1}\varphi - |\psi|^{\alpha-1}\psi \right|
\leq 
2^{\alpha+1}\alpha \left(|\varphi|^{\alpha-1}+|\psi|^{\alpha-1}\right) \left|\varphi-\psi\right|,
\quad
\alpha\geq 1.
$$
Combining this inequality with the generalized H\"older inequality for $r=2$, $r_1=\frac{2\alpha}{\alpha-1}$, $r_2 = 2\alpha$, we find
$$
\left\| |\varphi|^{\alpha-1} \varphi - |\psi|^{\alpha-1} \psi
\right\|_{L_x^2(\mathbb R^2)}
\lesssim
\left(
\left\| \varphi \right\|_{L_x^{2\alpha}(\mathbb R^2)}^{\alpha-1}
 + \left\| \psi \right\|_{L_x^{2\alpha}(\mathbb R^2)}^{\alpha-1}
\right) 
\left\| \varphi-\psi \right\|_{L_x^{2\alpha}(\mathbb R^2)},
$$
which agrees with estimate \eqref{diff-temp9} when $s=0$. Hence, both estimates can be summarized as
\begin{equation}\label{diff-hs-est}
\left\| |\varphi|^{\alpha-1} \varphi - |\psi|^{\alpha-1} \psi   \right\|_{H_x^s(\mathbb R^2)}
\lesssim
\left(
\left\| \varphi \right\|_{H_x^{s, p}(\mathbb R^2)}^{\alpha-1}
+
\left\| \psi \right\|_{H_x^{s, p}(\mathbb R^2)}^{\alpha-1}
\right)
\left\| \varphi - \psi \right\|_{H_x^{s, p}(\mathbb R^2)}
\end{equation}
with  $0\leq s<1$, $\alpha \geq 2$, $p = \frac{2\alpha}{1+(\alpha-1)s}$ and the rest of the restrictions that were imposed earlier. 

Taking $L_t^1(0, T)$ norms in \eqref{diff-hs-est} and using H\"older's inequality in $t$ with $\rho = q = \frac{2\alpha}{(1-s)(\alpha-1)}$ (this is a valid choice since $q \geq \alpha$ by \eqref{alpha-cond}), we have
\begin{align}
&\quad
\left\| |\varphi|^{\alpha-1} \varphi - |\psi|^{\alpha-1} \psi   \right\|_{L_t^1((0, T); H_x^s(\mathbb R^2))}
\nn\\
&\lesssim
\left(
\left\| \varphi \right\|_{L_t^{\rho'(\alpha-1)}((0, T); H_x^{s, p}(\mathbb R^2))}^{\alpha-1}
+
\left\| \psi \right\|_{L_t^{\rho'(\alpha-1)}((0, T); H_x^{s, p}(\mathbb R^2))}^{\alpha-1}
\right) 
\left\| \varphi - \psi \right\|_{L_t^q((0, T); H_x^{s, p}(\mathbb R^2))}.
\nn
\end{align}
Applying H\"older's inequality in $t$ once again, this time with $r$ such that $r' \rho'(\alpha-1)=q$ (this is a valid choice because $q  \geq \alpha$), we find
\begin{equation*}
\left\| \varphi \right\|_{L_t^{\rho'(\alpha-1)}((0, T); H_x^{s, p}(\mathbb R^2))}^{\alpha-1}
\leq
T^{\frac{1}{r\rho'}} \left\| \varphi \right\|_{L_t^{r' \rho'(\alpha-1)}((0, T); H_x^{s, p}(\mathbb R^2))}^{\alpha-1}
=
T^{\frac{q - \alpha}{q}}  \left\| \varphi \right\|_{L_t^q((0, T); H_x^{s, p}(\mathbb R^2))}^{\alpha-1}.
\end{equation*}
The last two inequalities yield the desired inequality \eqref{non-diff-est} under the restrictions \eqref{ds-est-cond-t}, completing the proof of Proposition \ref{qnls-non-p}.

\appendix 

\section{Solution of the forced linear Schr\"odinger equation on the spatial quarter-plane}
\label{utm-s}

In this section, we derive the solution formula \eqref{qnls-utm-sol} for the forced linear Schr\"odinger equation on the spatial quarter-plane, namely problem~\eqref{qnls-fls-ibvp}. We recall that this formula motivated the iteration map~\eqref{it-map} for the NLS quarter-plane problem \eqref{qnls-ibvp} and played a key role in the derivation of the linear estimate \eqref{fls-le}, which was used for proving the nonlinear well-posedness result of Theorem \ref{lwp-t}. For our derivation, we employ the unified transform of Fokas \cite{f1997,f2008}, which is a natural analogue of the Fourier transform method in the initial-boundary value problem setting. We assume existence of a solution and data with sufficient smoothness and decay at infinity as necessary for the various calculations to make sense. 

We begin by introducing the quarter-plane Fourier transform pair
\begin{equation}\label{qft-pair}
\begin{aligned}
\what \phi(k_1, k_2) &= \int_{\mathbb R_+} \int_{\mathbb R_+} e^{-ik_1x_1-ik_2x_2} \phi(x_1, x_2) dx_2 dx_1,
\quad \text{Im}(k_1),  \text{Im}(k_2) \leq 0,
\\
\phi(x_1, x_2) &= \frac{1}{(2\pi)^2} \int_{\mathbb R} \int_{\mathbb R} e^{ik_1x_1+ik_2x_2} \what \phi(k_1, k_2) dk_2 dk_1,
\quad (x_1, x_2) \in \mathbb R_+^2,
\end{aligned}
\end{equation}
which can be deduced from the standard two-dimensional Fourier transform pair for the extension of $\phi \in L^2(\mathbb R_+^2)$ by zero outside $\mathbb R_+^2$. In fact, via the same reasoning we additionally infer the following useful fact:
\begin{equation}\label{ft-0}
\int_{\mathbb R} \int_{\mathbb R} e^{ik_1x_1+ik_2x_2} \what \phi(k_1, k_2) dk_2 dk_1 = 0, \quad (x_1, x_2) \in \mathbb R^2 \setminus \overline{\mathbb R_+^2}.
\end{equation}
Taking the quarter-plane transform of the forced linear Schr\"odinger equation while assuming that $u$ and its derivatives decay to zero at infinity, and subsequently integrating with respect to $t$, we obtain an identity known in the unified transform terminology as the global relation:
\begin{equation}\label{qnls-gr}
\begin{aligned}
e^{i\omega t} \what u(k_1,k_2, t)
&=
\what u_0(k_1,k_2)
-\big[i\widetilde h_1(k_2, \omega, t) - k_1\widetilde h_0(k_2, \omega, t)\big]
-\big[i\widetilde  g_1(k_1, \omega, t) - k_2\widetilde g_0(k_1, \omega, t)\big]
\\
&\quad
-i\int_0^t e^{i\omega t'}\what f(k_1, k_2, t')dt',
\quad \text{Im}(k_1) \leq 0, \ \text{Im}(k_2) \leq 0,
\end{aligned}
\end{equation}
where the dispersion relation $\omega$ is given by \eqref{omega-def} and we have introduced the notation
\begin{equation}\label{tilde-def}
\widetilde g_j(k_1, \omega, t)
:=
\int_0^t e^{i\omega t'} \what g_j^{x_1}(k_1, t') dt',
\quad
\widetilde h_j(k_2, \omega, t)
:=
\int_0^t e^{i\omega t'} \what h_j^{x_2}(k_2, t') dt',
\quad
j=0, 1,
\end{equation}
with the half-line Fourier transforms $\what g_j^{x_1}$ and $\what h_j^{x_2}$ defined by \eqref{hl-ft-def}.
Inverting \eqref{qnls-gr} via \eqref{qft-pair}, we obtain the integral representation
\begin{equation}\label{qnls-ir}
\begin{aligned}
u(x_1, x_2, t)
&=
\frac{1}{(2\pi)^2}
\int_{\mathbb R}\int_{\mathbb R}
e^{ik_1x_1+ik_2x_2-i\omega t}
\left[
\what u_0(k_1, k_2) -i\int_0^t e^{i\omega t'}\what f(k_1, k_2, t')dt'
\right]
dk_2dk_1
\\
&\quad 
-\frac{1}{(2\pi)^2}
\int_{\mathbb R}\int_{\mathbb R}
e^{ik_1x_1+ik_2x_2-i\omega t}
\left[i \widetilde h_1(k_2, \omega, t)-k_1\widetilde h_0(k_2, \omega, t)\right] dk_2dk_1
\\
&\quad 
-\frac{1}{(2\pi)^2}
\int_{\mathbb R}\int_{\mathbb R}
e^{ik_1x_1+ik_2x_2-i\omega t}
\big[
i\widetilde  g_1(k_1, \omega, t)-k_2\widetilde  g_0(k_1, \omega, t)
\big] dk_2dk_1,
\end{aligned}
\end{equation}
which is not an explicit solution formula as it contains the unknown boundary values $u_{x_1}(0, x_2, t)$ and $u_{x_2}(x_1, 0, t)$ via the transforms $\widetilde h_1$ and $\widetilde g_1$.  
Nevertheless, these unknown transforms can be eliminated from \eqref{qnls-ir} by employing the identities obtained from the global relation \eqref{qnls-gr} via the transformations $k_1\mapsto -k_1$ and $k_2 \mapsto -k_2$, namely
\begin{align}
e^{i\omega t} \what u(-k_1,k_2, t)
&=
\what u_0(-k_1,k_2)
-\big[i\widetilde h_1(k_2, \omega, t) + k_1\widetilde h_0(k_2, \omega, t)\big]
-\big[i\widetilde  g_1(-k_1, \omega, t) - k_2\widetilde g_0(-k_1, \omega, t)\big]
\nn\\
&\quad
-i\int_0^t e^{i\omega t'}\what f(-k_1, k_2, t')dt',
\quad \text{Im}(k_1) \geq 0, \ \text{Im}(k_2) \leq 0,
\label{qnls-grs1}
\\
e^{i\omega t} \what u(k_1,-k_2, t)
&=
\what u_0(k_1,-k_2)
-\big[i\widetilde h_1(-k_2, \omega, t) - k_1\widetilde h_0(-k_2, \omega, t)\big]
-\big[i\widetilde  g_1(k_1, \omega, t) + k_2\widetilde g_0(k_1, \omega, t)\big]
\nn\\
&\quad
-i\int_0^t e^{i\omega t'}\what f(k_1, -k_2, t')dt',
\quad \text{Im}(k_1) \leq 0, \ \text{Im}(k_2) \geq 0.
\label{qnls-grs2}
\end{align}
The key is that the transformations leading to \eqref{qnls-grs1} and \eqref{qnls-grs2} leave $\omega$ invariant. Hence, making the crucial observation that, in view of \eqref{ft-0}, for any $(x_1, x_2) \in \mathbb R_+^2$  
\begin{equation}\label{qnls-uvan}
\begin{aligned}
&\int_{\mathbb R} \int_{\mathbb R} e^{ik_1x_1+ik_2x_2}\what u(-k_1,k_2, t) dk_2 dk_1
=
\int_{\mathbb R} \int_{\mathbb R} e^{i k_1 (-x_1)+ik_2x_2}\what u(k_1,k_2, t) dk_2 dk_1
= 0,
\\
&\int_{\mathbb R} \int_{\mathbb R} e^{ik_1x_1+ik_2x_2}\what u(k_1,-k_2, t) dk_2 dk_1
=
\int_{\mathbb R} \int_{\mathbb R} e^{i k_1 x_1+ik_2(-x_2)}\what u(k_1,k_2, t) dk_2 dk_1
= 0,
\end{aligned}
\end{equation}
we solve \eqref{qnls-grs1} for $\widetilde  h_1(k_2, \omega, t)$  and \eqref{qnls-grs2} for $\widetilde  g_1(k_1, \omega, t)$ to get 
\begin{equation}\label{qnls-ir2k1}
\begin{aligned}
&\quad
\int_{\mathbb R}\int_{\mathbb R}
e^{ik_1x_1+ik_2x_2-i\omega t}
\left[
i\widetilde  h_1(k_2, \omega, t)-k_1\widetilde  h_0(k_2, \omega, t)
\right]  dk_2dk_1
\\
&=
\int_{\mathbb R}\int_{\mathbb R}
e^{ik_1x_1+ik_2x_2-i\omega t}
\left\{
-\big[i\widetilde  g_1(-k_1, \omega, t)-k_2 \widetilde  g_0(-k_1, \omega, t)\big]
-2k_1\widetilde  h_0(k_2, \omega, t)
\right\}  dk_2dk_1
\\
&\quad
+
\int_{\mathbb R}\int_{\mathbb R}
e^{ik_1x_1+ik_2x_2-i\omega t}
\Big[
\what u_0(-k_1,k_2) - i \int_0^t e^{i\omega t'}\what f(-k_1, k_2, t')dt' 
\Big] dk_2dk_1.
\end{aligned}
\end{equation}
and
\begin{equation}\label{qnls-ir2k2}
\begin{aligned}
&\quad
\int_{\mathbb R}\int_{\mathbb R}
e^{ik_1x_1+ik_2x_2-i\omega t}
\left[
i\widetilde  g_1(k_1, \omega, t)-k_2\widetilde  g_0(k_1, \omega, t)
\right]  dk_2dk_1
\\
&=
\int_{\mathbb R}\int_{\mathbb R}
e^{ik_1x_1+ik_2x_2-i\omega t}
\left\{
-\big[i\widetilde  h_1(-k_2, \omega, t) - k_1 \widetilde  h_0(-k_2, \omega, t)\big]
-
2k_2\widetilde  g_0(k_1, \omega, t)
\right\}  dk_2dk_1
\\
&\quad
+
\int_{\mathbb R}\int_{\mathbb R}
e^{ik_1x_1+ik_2x_2-i\omega t}
\Big[
\what u_0(k_1,-k_2)
-i \int_0^t e^{i\omega t'}\what f(k_1, -k_2, t')dt'
\Big] dk_2dk_1.
\end{aligned}
\end{equation}
Furthermore, employing the additional identity
\begin{align}
e^{i\omega t} \what u(-k_1,-k_2, t)
&=
\what u_0(-k_1,-k_2)
-\big[i\widetilde h_1(-k_2, \omega, t) + k_1\widetilde h_0(-k_2, \omega, t)\big]
-\big[i\widetilde  g_1(-k_1, \omega, t) + k_2\widetilde g_0(-k_1, \omega, t)\big]
\nn\\
&\quad
-i\int_0^t e^{i\omega t'}\what f(-k_1, -k_2, t')dt',
\quad \text{Im}(k_1) \geq 0, \ \text{Im}(k_2) \geq 0,
\label{qnls-grs12}
\end{align}
(which comes from the global relation \eqref{qnls-gr} after applying simultaneously the transformations $k_1 \mapsto -k_1$ and $k_2 \mapsto -k_2$) we infer,  in view of \eqref{ft-0}, 
\begin{equation}\label{qnls-comb}
\begin{aligned}
&\quad
\int_{\mathbb R}\int_{\mathbb R}
e^{ik_1x_1+ik_2x_2-i\omega t}
\big[i\widetilde  h_1(-k_2, \omega, t) + i\widetilde  g_1(-k_1, \omega, t)\big]
dk_2dk_1
\\
&=
\int_{\mathbb R}\int_{\mathbb R}
e^{ik_1x_1+ik_2x_2-i\omega t} 
\Big[
\what u_0(-k_1,-k_2) 
-
i \int_0^t e^{i\omega t'}
\what f(-k_1, -k_2, t')dt' 
\Big] dk_2dk_1
\\
&\quad
-
\int_{\mathbb R}\int_{\mathbb R} e^{ik_1x_1+ik_2x_2-i\omega  t}
\left[
k_1\widetilde  h_0(-k_2, \omega, t)+k_2\widetilde  g_0(-k_1, \omega, t)
\right]  dk_2 dk_1.
\end{aligned}
\end{equation}
Combining  \eqref{qnls-ir2k1}, \eqref{qnls-ir2k2} and \eqref{qnls-comb} into the integral representation \eqref{qnls-ir}, we obtain the unified transform solution formula in the form
\begin{equation}\label{qnls-utm-sol-0}
\begin{aligned}
u(x_1, x_2, t)
&=
\frac{1}{(2\pi)^2}
\int_{\mathbb R}\int_{\mathbb R}
e^{ik_1x_1+ik_2x_2-i\omega t}
\Big[
\what u_0(k_1, k_2)
-i\int_0^t 
e^{i\omega t'}\what f(k_1, k_2, t')dt'
\Big]
dk_2dk_1
\\
&\quad
-\frac{1}{(2\pi)^2}
\int_{\mathbb R}\int_{\mathbb R}
e^{ik_1x_1+ik_2x_2-i\omega t}
\Big[
\what u_0(-k_1,k_2)
-
i \int_0^t e^{i\omega t'}\what f(-k_1, k_2, t')dt' 
\Big]  dk_2dk_1
\\
&\quad
-\frac{1}{(2\pi)^2}
\int_{\mathbb R}\int_{\mathbb R}
e^{ik_1x_1+ik_2x_2-i\omega t}
\Big[
\what u_0(k_1,-k_2) 
- i \int_0^t e^{i\omega t'}\what f(k_1, -k_2, t')dt'
\Big] dk_2 dk_1
\\
&\quad
+\frac{1}{(2\pi)^2}\int_{\mathbb R}\int_{\mathbb R}
e^{ik_1x_1+ik_2x_2-i\omega t} 
\Big[
\what u_0(-k_1,-k_2) -i\int_0^t 
e^{ i\omega  t'}
\what f(-k_1, -k_2, t')dt' 
\Big] dk_2dk_1
\\
&\quad
+\frac{1}{(2\pi)^2}
\int_{\mathbb R}\int_{\mathbb R}
e^{ik_1x_1+ik_2x_2-i\omega t}
\cdot 
2k_1 \big[\, \widetilde  h_0(k_2, \omega, t) - \widetilde  h_0(-k_2, \omega, t)\big] 
dk_2dk_1
\\
&\quad
+\frac{1}{(2\pi)^2}
\int_{\mathbb R}\int_{\mathbb R}
e^{ik_1x_1+ik_2x_2-i\omega t}
\cdot 2k_2 \big[\widetilde  g_0(k_1, \omega, t) - \widetilde  g_0(-k_1, \omega, t)\big]  dk_2 dk_1.
\end{aligned}
\end{equation}

A different version of the above formula, which is the one appropriate for proving the linear estimate \eqref{fls-le}, can be obtained via contour deformations. By Cauchy's theorem, we write
\begin{equation}\label{pre-jl}
\int_{\mathbb R} e^{ik_2x_2-i\omega t} k_2 \, \widetilde g_0(\pm k_1, \omega, t) dk_2 
=
\left(-\lim_{R\to\infty} \int_{C_R} + \int_{\p D_2}\right) e^{ik_2x_2-i\omega t} k_2\, \widetilde g_0(\pm k_1, \omega, t) dk_2
\end{equation}
where $C_R$ is the quarter-circle of radius $R$ and center at the origin in the second quadrant of the complex plane, and $\p D_j$ denotes the positively oriented boundary of the first quadrant $D_j$ of the $k_j$ plane, $j=1, 2$, in which  $\text{Re}(ik_j)<0$ and $\text{Re}(ik_j^2)<0$ (see Figure \ref{qnls-dplus}). Moreover, we have
\begin{lemma}\label{jl1-l}
For each $x_2>0$, $0\leq t \leq T$ and $k_1 \in \mathbb R$, 
$$
\lim_{R\to\infty} \int_{C_R}  e^{ik_2x_2-i\omega t} k_2 \, \widetilde g_0(\pm k_1, \omega, t) dk_2 = 0.
$$
\end{lemma}

\begin{proof}
Recalling that $\widetilde g_0(\pm k_1, \omega, t) = 
\int_0^t e^{i\omega t'} \what g_0^{x_1}(\pm k_1, t') dt'$, 
we integrate by parts in $t'$ to obtain
$$
\int_0^t e^{i\omega t'} \what g_0^{x_1}(\pm k_1, t') dt'
=
\frac{1}{i\omega}
\left[ e^{i\omega t} \what g_0^{x_1}(\pm k_1, t) - \what g_0^{x_1}(\pm k_1, 0) - \int_0^t e^{i\omega t'} \p_{t'} \what g_0^{x_1}(\pm k_1, t') dt' \right].
$$
Thus, since $\omega = k_1^2 + k_2^2$ and $\text{Re}(ik_1^2) = 0$ for $k_1 \in \mathbb R$, parametrizing $C_R$ by $k_2 = Re^{i\theta}$, $\frac \pi 2 \leq \theta \leq \pi$, we find
\begin{align}
&\quad
\left|\int_{C_R}  e^{ik_2x_2-i\omega t} k_2 \, \widetilde g_0(\pm k_1, \omega, t) dk_2\right|
\nn\\
&\leq
\int_{\frac \pi 2}^\pi  e^{-R x_2  \sin\theta} 
\frac{R}{\left|k_1^2 + R^2 e^{2i\theta}\right|} \left[ \left|\what g_0^{x_1}(\pm k_1, t)\right| + e^{R^2 \sin(2\theta) t} \left|\what g_0^{x_1}(\pm k_1, 0)\right|
 + \int_0^t e^{R^2 \sin(2\theta)(t-t')} \left|\p_{t'} \what g_0^{x_1}(\pm k_1, t')\right| dt' \right] R d\theta
\nn\\
&\leq
\int_{\frac \pi 2}^\pi  e^{-R x_2  \sin\theta} \frac{R^2}{R^2-|k_1|^2}
\left[ \left|\what g_0^{x_1}(\pm k_1, t)\right| + \left|\what g_0^{x_1}(\pm k_1, 0)\right|
 + \int_0^t  \left|\p_{t'} \what g_0^{x_1}(\pm k_1, t')\right| dt' \right] d\theta
\nn\\
 &\leq
\left(2\no{\what g_0^{x_1}(\pm k_1)}_{L_t^\infty(0, T)} +  \no{\p_t \what g_0^{x_1}(\pm k_1)}_{L_t^1(0, T)}\right) \frac{R^2}{R^2-|k_1|^2} \int_0^{\frac \pi 2}  e^{-R x_2  \sin \theta}  d\theta,
\nn
\end{align}
where for each given $k_1 \in \mathbb R$ we work with $R > |k_1|$. 
Then, using the inequality $\sin\theta \geq \frac 2\pi \theta$, $0\leq \theta \leq \frac \pi 2$, we conclude  
\begin{align*}
\left|\int_{C_R}  e^{ik_2x_2-i\omega t} k_2 \, \widetilde g_0(\pm k_1, \omega, t) dk_2\right|
&\leq
\left(2\no{\what g_0^{x_1}(\pm k_1)}_{L_t^\infty(0, T)} +  \no{\p_t \what g_0^{x_1}(\pm k_1)}_{L_t^1(0, T)}\right) \frac{R^2}{R^2-|k_1|^2} \cdot \frac{\pi}{2Rx_2} \left(1-e^{-Rx_2}\right),
\end{align*}
which yields the desired result for each $x_2>0$ and $k_1 \in \mathbb R$.
\end{proof}
In view of Lemma \ref{jl1-l}, \eqref{pre-jl} becomes
\begin{equation*}
\int_{\mathbb R} e^{ik_2x_2-i\omega t} k_2 \, \widetilde g_0(\pm k_1, \omega, t) dk_2 
=
\int_{\p D_2} e^{ik_2x_2-i\omega t} k_2\, \widetilde g_0(\pm k_1, \omega, t) dk_2
\end{equation*}
for each $k_1 \in \mathbb R$,  $x_2>0$ and $0\leq t \leq T$. In turn,
\begin{equation}\label{qnls-dd1}
\int_{\mathbb R}\int_{\mathbb R}
e^{ik_1x_1+ik_2x_2-i\omega t} k_2 \widetilde  g_0(-k_1, \omega, t) dk_2dk_1
=
\int_{\mathbb R}\int_{\p D_2}
e^{ik_1x_1+ik_2x_2-i\omega t} k_2 \widetilde  g_0(-k_1, \omega, t) dk_2dk_1.
\end{equation}
Analogously, we also have
\begin{equation}\label{qnls-dd2}
\int_{\mathbb R}\int_{\mathbb R} e^{ik_1x_1+ik_2x_2-i\omega t} k_1 \widetilde  h_0(-k_2, \omega, t) dk_2 dk_1
=
\int_{\mathbb R}\int_{\p D_1} e^{ik_1x_1+ik_2x_2-i\omega t} k_1 \widetilde  h_0(-k_2, \omega, t) dk_1 dk_2.
\end{equation}
In turn, the solution formula \eqref{qnls-utm-sol-0} can be written as
\begin{equation}\label{qnls-utm-sol2}
\begin{aligned}
u(x_1, x_2, t)
&=
\frac{1}{(2\pi)^2}
\int_{\mathbb R}\int_{\mathbb R}
e^{ik_1x_1+ik_2x_2-i\omega t}
\Big[
\what u_0(k_1, k_2)
-i\int_0^t 
e^{i\omega t'}\what f(k_1, k_2, t')dt'
\Big]
dk_2dk_1
\\
&\quad
-\frac{1}{(2\pi)^2}
\int_{\mathbb R}\int_{\mathbb R}
e^{ik_1x_1+ik_2x_2-i\omega t}
\Big[
\what u_0(-k_1,k_2)
-
i \int_0^t e^{i\omega t'}\what f(-k_1, k_2, t')dt' 
\Big]  dk_2dk_1
\\
&\quad
-\frac{1}{(2\pi)^2}
\int_{\mathbb R}\int_{\mathbb R}
e^{ik_1x_1+ik_2x_2-i\omega t}
\Big[
\what u_0(k_1,-k_2) 
- i \int_0^t e^{i\omega t'}\what f(k_1, -k_2, t')dt'
\Big] dk_2 dk_1
\\
&\quad
+\frac{1}{(2\pi)^2}\int_{\mathbb R}\int_{\mathbb R}
e^{ik_1x_1+ik_2x_2-i\omega t} 
\Big[
\what u_0(-k_1,-k_2) -i\int_0^t 
e^{ i\omega  t'}
\what f(-k_1, -k_2, t')dt' 
\Big] dk_2dk_1
\\
&\quad
+\frac{1}{(2\pi)^2}
\int_{\mathbb R}\int_{\p D_1}
e^{ik_1x_1+ik_2x_2-i\omega t}
\cdot 
2k_1 \big[\, \widetilde  h_0(k_2, \omega, t) - \widetilde  h_0(-k_2, \omega, t)\big] 
dk_1dk_2
\\
&\quad
+\frac{1}{(2\pi)^2}
\int_{\mathbb R}\int_{\p D_2}
e^{ik_1x_1+ik_2x_2-i\omega t}
\cdot 2k_2 \big[\widetilde  g_0(k_1, \omega, t) - \widetilde  g_0(-k_1, \omega, t)\big]  dk_2 dk_1.
\end{aligned}
\end{equation}

The next important observation is that $t$ in the argument of the tilde transforms can be replaced by any fixed $T\geq t\geq 0$. Indeed, we have
\begin{lemma}\label{jl3-l}
For each $k_2 \in \mathbb R$, $x_1>0$ and $0\leq t \leq T$, 
\begin{equation}
\int_{\p D_1} 
e^{ik_1x_1-i\omega t}
k_1 \int_t^T e^{i\omega t'} \what h_0^{x_2}(\pm k_2, t') dt' dk_1
= 0.
\end{equation}
\end{lemma}

\begin{proof}
By Cauchy's theorem, we can deform $\p D_1$ to the contour $\lim_{R\to\infty}\widetilde C_R$ denoting the quarter-circle that bounds the first quadrant of the complex plane. Then, 
$$
\int_{\p D_1} 
e^{ik_1x_1-i\omega t}
k_1 \int_t^T e^{i\omega t'} \what h_0^{x_2}(\pm k_2, t') dt' dk_1
=
-\lim_{R\to\infty} \int_{\widetilde C_R} 
e^{ik_1x_1-i\omega t}
k_1 \int_t^T e^{i\omega t'} \what h_0^{x_2}(\pm k_2, t') dt' dk_1.
$$
Integrating by parts in $t'$,  
$$
\int_t^T e^{i\omega t'} \what h_0^{x_2}(\pm k_2, t') dt' 
=
\frac{1}{i\omega}
\left[e^{i\omega T} \what h_0^{x_2}(\pm k_2, T) - e^{i\omega t} \what h_0^{x_2}(\pm k_2, t) - \int_t^T e^{i\omega t'} \p_{t'} \what h_0^{x_2}(\pm k_2, t') dt' \right].
$$
Thus, parametrizing $\widetilde C_R$ by $k_1 = Re^{i\theta}$, $0\leq\theta\leq\frac \pi 2$, for each $k_2\in\mathbb R$ (for which $\text{Re}(ik_2^2)=0$) we have
\begin{align}
&\quad
\left|
\int_{\widetilde C_R} 
e^{ik_1x_1-i\omega t}
k_1 \int_t^T e^{i\omega t'} \what h_0^{x_2}(\pm k_2, t') dt' dk_1
\right|
\nn\\
&\leq
\int_0^{\frac \pi 2} 
e^{-R x_1 \sin\theta}
\frac{R}{\left|R^2 e^{2i\theta} + k_2^2\right|}
\bigg[
e^{-R^2 \sin(2\theta)(T-t)} \left|\what h_0^{x_2}(\pm k_2, T)\right| 
+ \left|\what h_0^{x_2}(\pm k_2, t)\right| 
\nn\\
&\hspace*{4.9cm}
+ \int_t^T e^{-R^2 \sin(2\theta)(t'-t)} \left|\p_{t'} \what h_0^{x_2}(\pm k_2, t'))\right| dt'
\bigg]
R d\theta
\nn\\
&\leq
\frac{R^2}{R^2 - |k_2|^2} 
\left(
2\no{\what h_0^{x_2}(\pm k_2)}_{L_t^\infty(0, T)} + \no{\p_t \what h_0^{x_2}(\pm k_2)}_{L_t^1(0, T)}
\right)
\int_0^{\frac \pi 2} e^{-R x_1 \sin\theta} d\theta.
\nn
\end{align}
Hence, using once again  the inequality $\sin\theta \geq \frac 2\pi \theta$, $0\leq \theta \leq \frac \pi 2$, we obtain
\begin{align*}
&\quad
\left|
\int_{\widetilde C_R} 
e^{ik_1x_1-i\omega t}
k_1 \int_t^T e^{i\omega t'} \what h_0^{x_2}(\pm k_2, t') dt' dk_1
\right|
\nn\\
&\leq 
\left(
2\no{\what h_0^{x_2}(\pm k_2)}_{L_t^\infty(0, T)} + \no{\p_t \what h_0^{x_2}(\pm k_2)}_{L_t^1(0, T)}
\right)
\frac{R^2}{R^2 - |k_2|^2}
\cdot \frac{\pi}{2Rx_1} \left(1-e^{-Rx_1}\right),
\end{align*}
implying the desired result for each $k_2\in\mathbb R$ and $x_1>0$m $0\leq t \leq T$.
\end{proof}

Lemma \ref{jl3-l} implies
\begin{equation}
\int_{\mathbb R} \int_{\p D_1}
e^{ik_1x_1+ik_2x_2-i\omega t}
k_1  \widetilde  h_0(\pm k_2, \omega, t) dk_1 dk_2
=
\int_{\mathbb R} \int_{\p D_1}
e^{ik_1x_1+ik_2x_2-i\omega t}
k_1  \widetilde  h_0(\pm k_2, \omega, T) dk_1 dk_2.
\end{equation}
Similarly, we have
\begin{equation}
\int_{\mathbb R} \int_{\p D_2}
e^{ik_1x_1+ik_2x_2-i\omega t}
k_2  \widetilde  g_0(\pm k_1, \omega, t) dk_2 dk_1
=
\int_{\mathbb R} \int_{\p D_2}
e^{ik_1x_1+ik_2x_2-i\omega t}
k_2  \widetilde  g_0(\pm k_1, \omega, T) dk_2 dk_1.
\end{equation}
Therefore, formula \eqref{qnls-utm-sol2} can be written in the form \eqref{qnls-utm-sol}, which is the one convenient for proving the linear estimates of Theorems \ref{sob-est-t} and \ref{strich-est-t}.

\section{Zero extension operator for Sobolev spaces on the quarter-plane}
\label{app-s}

For each $s\geq 0$, define
\begin{equation}\label{xs-app-def}
E_s :=  \left\{u\in L^2(\mathbb{R}_+^2): \tilde{u}\in H^s(\mathbb{R}^2)\right\}
\end{equation}
where $\tilde{u}$ denotes the extension of $u$ by zero outside $\mathbb R_+^2$, i.e. 
\begin{equation}\label{util-def}
\tilde u(x) := \left\{\begin{array}{ll} u(x), & x\in\mathbb R_+^2, \\ 0, & x\in\mathbb R^2\setminus \mathbb R_+^2.\end{array}\right.
\end{equation}
The purpose of this appendix is to establish the following result:
\begin{proposition}\label{es-p}
If $0\leq s < \frac 12$, then $E_s = H^s(\mathbb R_+^2)$ and the zero extension $\tilde u$ of each $u \in E_s$ satisfies the bound
\begin{equation}
\no{\tilde{u}}_{{H}^s(\mathbb{R}^2)}\leq c_s \no{u}_{H^s(\mathbb{R}_+^2)}, \ c_s > 0.
\end{equation}
\end{proposition}

Proposition \ref{es-p} will be proved below by combining a series of results. In order to state these results, we first give some definitions.

Let $\mathcal{D}(\mathbb{R}_+^2)$ denote the space of test functions defined on $\mathbb{R}^2$ but with support in $\mathbb{R}_+^2$ (i.e. with support in a compact subset of $\mathbb R^2$ contained inside the open quadrant $\mathbb{R}_+^2$).  Furthermore, let $\mathscr{D}(\mathbb{R}_+^2)$ denote the space of test functions  defined on $\mathbb{R}_+^2$ (i.e. with support in a compact subset of the open quadrant $\mathbb{R}_+^2$). Note that $\mathcal{D}(\mathbb{R}_+^2) = \mathscr{D}(\mathbb{R}_+^2)$ in the sense that: (i) if $u\in \mathcal{D}(\mathbb{R}_+^2)$, then its restriction $u|_{\mathbb R_+^2}$ belongs to $ \mathscr{D}(\mathbb{R}_+^2)$ as the support of this restriction is a compact set inside $\mathbb{R}_+^2$, and (ii) if $u\in \mathscr{D}(\mathbb{R}_+^2)$, then its extension by zero outside $\mathbb R_+^2$ belongs to $\mathcal{D}(\mathbb{R}_+^2)$. Thus, hereafter we use the notation $\mathcal{D}(\mathbb{R}_+^2)$ to denote both of the above spaces of test functions.

For $s\in \mathbb{R}$,  define the space $H_0^s(\mathbb{R}_+^2)$ as the closure of $\mathcal{D}(\mathbb{R}_+^2)$ in $H^s(\mathbb{R}_+^2)$, i.e.
\begin{equation}\label{h0s-def}
H_0^s(\mathbb{R}_+^2) := \overline{\mathcal{D}(\mathbb{R}_+^2)}^{H^s(\mathbb{R}_+^2)}.
\end{equation}
Proposition \ref{es-p} can be deduced by combining the following two important results:
\begin{theorem}[Theorem 6.78 in \cite{l2023}]\label{h0s-t} \ 
If $0 \leq s \leq \frac 12$, then $H_0^s(\mathbb R_+^2) = H^s(\mathbb R_+^2)$.
\end{theorem}
Theorem \ref{h0s-t} is proved in \cite{l2023} for general Lipschitz domains and hence it provides a more general version of Theorem 11.5 in \cite{lm1972}, since in the latter case the domain is assumed to be bounded and with a smooth boundary (which is not the case for $\mathbb R_+^2$). Another related result is Theorem 3.40 in \cite{mclean2000}, however the proof given there assumes that the boundary of the domain (not necessarily the domain itself) is compact, which is not true for $\mathbb R_+^2$. 

\begin{theorem}\label{hst-h0s-t} 
Let $0 \leq s < \frac 12$. Then, $E_s=H_0^s(\mathbb{R}_+^2)$, and the zero extension $\tilde u$ of each $u\in E_s=H_0^s(\mathbb{R}_+^2)$ satisfies
\begin{equation}
\no{\tilde{u}}_{{H}^s(\mathbb{R}^2)}\leq c_s \no{u}_{H^s(\mathbb{R}_+^2)}, \ c_s > 0.
\end{equation}
\end{theorem}
As noted above, Proposition \ref{es-p} follows by combining Theorems \ref{h0s-t} and \ref{hst-h0s-t}. As Theorem \ref{h0s-t} is proved in~\cite{l2023}, it remains to establish Theorem \ref{hst-h0s-t}. 
In fact, this result is the quarter-plane analogue of  Theorem 11.4 in \cite{lm1972}, which cannot be employed for the quarter-plane because it is proved for a bounded domain with a smooth boundary. Another related result is Theorem 3.33 in \cite{mclean2000}, which is proved without the requirement of a bounded domain with a smooth boundary but does assume a compact boundary, which is not the case for $\mathbb R_+^2$. Hence, below we give a separate proof of Theorem \ref{hst-h0s-t} by suitably adapting the argument of \cite{mclean2000} after removing the compact boundary assumption.

Before proving Theorem \ref{hst-h0s-t}, we establish certain auxiliary results.
For $s\in\mathbb R$, we define the space $\tilde{H}^s(\mathbb{R}_+^2)$ as the closure of $\mathcal{D}(\mathbb{R}_+^2)$ in $H^s(\mathbb{R}^2)$, i.e.
\begin{equation}\label{hstil-def}
\tilde{H}^s(\mathbb{R}_+^2) := \overline{\mathcal{D}(\mathbb{R}_+^2)}^{H^s(\mathbb{R}^2)}.
\end{equation}

\begin{lemma}\label{secinc}	
For any $s\geq 0$,	
$\tilde{H}^s(\mathbb{R}_+^2)= H^s({\overline{\mathbb{R}_+^2}}) := \big\{u\in H^s(\mathbb{R}^2): \textnormal{supp}\, u\subseteq \overline{\mathbb{R}_+^2}\big\}$.
\end{lemma}

\begin{proof}
Since $\tilde H^s(\mathbb R_+^2)$ is the closure of $\mathcal D(\mathbb R_+^2)$ in $H^s(\mathbb R^2)$, if $u \in \tilde H^s(\mathbb R_+^2)$ then $u \in H^s(\mathbb R^2)$ with $\text{supp}(u) \subseteq \overline{\mathbb R_+^2}$. Hence, $u \in H^s({\overline{\mathbb R_+^2}})$ and so $\tilde H^s(\mathbb R_+^2) \subseteq H^s({\overline{\mathbb R_+^2}})$. It remains to prove the inclusion $H^s({\overline{\mathbb R_+^2}}) \subseteq \tilde H^s(\mathbb R_+^2)$.

Let $u \in H^s({\overline{\mathbb{R}_+^2}})$ and $\epsilon>0$ be arbitrary. Under the change of variables
	\begin{equation}\label{changeofvar}
			\xi_1=-\frac{1}{\sqrt{2}} \left(x_1-x_2\right),\ \xi_2=-\frac{1}{\sqrt{2}} \left(x_1+x_2\right),
		\end{equation}
which takes $x = (x_1, x_2)$ to $\xi=(\xi_1, \xi_2)$ via an anticlockwise rotation of $\frac{5\pi}{4}$, the quarter-plane  $\mathbb{R}_+^2$ is mapped to
$$
\Omega = \left\{\xi_2 < -|\xi_1|\right\} \subset\mathbb{R}^2.
$$
Importantly, $\Omega$ is described by a \textit{single inequality}  as opposed to the two inequalities describing $\mathbb R_+^2$. 
Letting $v\in H^s({\overline \Omega}) := \left\{f\in H^s(\mathbb R^2): \text{supp}(f) \subseteq \overline{\Omega}\right\}$ be the function associated with $u$ and defined by   
$$
v(\xi) = v(\xi_1, \xi_2) :=u(-\frac{1}{\sqrt{2}}(\xi_1+\xi_2), -\frac{1}{\sqrt{2}}(\xi_2-\xi_1)),
$$ 
we consider the sequence
$$
v_n(\xi) = v_n(\xi_1,\xi_2):=v(\xi_1,\xi_2+\frac{1}{n}), \quad n\in\mathbb N.
$$  
Since $v\in H^s(\mathbb{R}^2)$, we have $v_n\in H^s(\mathbb{R}^2)$ as the shift of $\xi_2$ by $\frac 1n$ does not affect $\mathbb R^2$. Moreover, since $\text{supp}(v) \subseteq \overline \Omega =  \left\{\xi_2 \leq -|\xi_1|\right\}$, we have
$$
\text{supp}(v_n)\subseteq \Big\{\xi_2 \leq -|\xi_1|-\frac{1}{n}\Big\}
$$
so that, in particular, $\text{supp}(v_n)$ is  contained in  the \textit{open set} $\{\xi_2 < -|\xi_1|\} = \Omega$. 
Next, note that
$$
\mathcal F\{v_n\}(k)
=
\int_{\mathbb R^2} e^{-ik\cdot \xi} v_n(\xi) d\xi 
=
\int_{\mathbb R^2} e^{-ik\cdot \xi} v(\xi_1, \xi_2 + \tfrac 1n) dx 
=
e^{i\frac{k_2}{n}} \mathcal F\{v\}(k)
$$
and so
$$
\no{v-v_{n}}_{H^s(\mathbb{R}^2)}^2
=
\int_{\mathbb R^2} (1+|k|^2)^s \left|\mathcal F\{v\}(k)\right|^2 \left|1-e^{i\frac{k_2}{n}}\right|^2 dk
\leq
4\no{v}_{H^s(\mathbb R^2)}^2
< \infty.
$$
Thus, by the dominated convergence theorem,  
$$
\lim_{n\to\infty} \|v-v_{n}\|_{H^s(\mathbb{R}^2)}
=
\int_{\mathbb R^2} (1+|k|^2)^s \left|\lim_{n\to\infty} \mathcal F\{v-v_n\}(k)\right|^2 dk
= 0
$$
and, therefore, given $\epsilon > 0$ there exists $n=n(\epsilon) \in \mathbb N$ such that 
\begin{equation}\label{vn-v-ineq}
\no{v-v_{n}}_{H^s(\mathbb{R}^2)}<\epsilon.
\end{equation}

Let $v_n^\rho=\eta_\rho \ast v_n$, $\xi\in \mathbb{R}^2$, where $\eta_\rho(\xi)=\rho^{-2}\eta(\xi/\rho)$, $\rho>0$, and $\eta$ is the standard mollifier.  Then, $v_n^\rho\in C^\infty(\mathbb{R}^2)$ and, for sufficiently small $\rho$, 
$$
\text{supp}(v_n^\rho) \subseteq \text{supp}(\eta_\rho) +\text{supp}(v_n)=B_\rho(0)+\text{supp}(v_n)\subset \Omega
$$ 
so that $v_n^\rho \in \mathcal D(\Omega)$  for small enough $\rho$ (here, $\mathcal D(\Omega)$ is the space of test functions on $\mathbb R^2$ with support in $\Omega$).
Since $\mathcal F\{\eta_\rho\}(k) = \mathcal F\{\eta\}(\rho k)$, we have
$$
	\|v_n^\rho-v_n\|_{H^s(\mathbb{R}^2)}^2 
	=\int_{\mathbb{R}^2}(1+|k|^2)^s\left|\mathcal F\{v_n\}(k)\right|^2 \left|\mathcal F\{\eta\}(\rho k)-1\right|^2dk.
$$
Hence, in view of the limit $\lim_{\rho \to 0^+}\mathcal F\{\eta\}(\rho k)= 1$ and the dominated convergence theorem, $v_n^\rho\rightarrow v_n$ in $H^s(\mathbb{R}^2)$ as $\rho\rightarrow 0^+$. Therefore,  for small enough $\rho$, we have $v_n^\rho \in \mathcal D(\Omega)$  such that
\begin{equation}\label{vn-phin-ineq}
	\|v_{n}-v_n^\rho\|_{H^s(\mathbb{R}^2)}<\epsilon.
\end{equation}	

Combining inequalities \eqref{vn-v-ineq} and \eqref{vn-phin-ineq}, we deduce
$$
\|v-v_n^\rho\|_{H^s({\overline{\Omega}})}  = \|v-v_n^\rho\|_{H^s(\mathbb R^2)} \le \|v-v_{n}\|_{H^s(\mathbb{R}^2)}+\|v_{n}-v_n^\rho\|_{H^s(\mathbb{R}^2)}<2\epsilon,
$$
which shows that any element of $H^s({\overline \Omega})$ can be approximated by a sequence in $\mathcal{D}(\Omega)$, i.e. $\mathcal{D}(\Omega)$ is dense in $H^s({\overline \Omega})$. But the limits of sequences in $\mathcal{D}(\Omega)$ belong to $\tilde{H}^s(\Omega)$ by definition of this space. Thus,  $H^s({\overline \Omega}) \subseteq \tilde{H}^s(\Omega)$, which implies $H^s({\overline{\mathbb R_+^2}}) \subseteq \tilde{H}^s(\mathbb R_+^2)$ after taking the rotation into account.
\end{proof}
	
Thanks to Lemma \ref{secinc}, we can prove 
\begin{lemma}\label{hst-xs-l} 
For any $s\geq 0$, we have $\tilde{H}^s(\mathbb{R}_+^2)=E_s$ where 
\\[1mm]
\textnormal{(i)} the inclusion $\tilde{H}^s(\mathbb{R}_+^2)\subseteq E_s$ is understood  in the sense that if $u\in \tilde{H}^s(\mathbb{R}_+^2)$ then $v:=u|_{\mathbb{R}_+^2}\in E_s$. In particular, $v \in L^2(\mathbb{R}_+^2)$ and its zero extension $\tilde{v}$ is equal to $u$ and hence belongs to $H^s(\mathbb{R}^2)$;
\\[1mm]
\textnormal{(ii)} the inclusion $E_s\subseteq \tilde{H}^s(\mathbb{R}_+^2)$ is understood in the sense that if $u\in E_s$ then its zero extension $\tilde{u}\in \tilde{H}^s(\mathbb{R}_+^2)$.
\end{lemma}
\begin{proof} To prove the inclusion $\tilde{H}^s(\mathbb{R}_+^2)\subseteq E_s$, let $u\in \tilde{H}^s(\mathbb{R}_+^2)$ and set $v=u|_{\mathbb{R}_+^2}$.  
By the definition \eqref{hstil-def} of $\tilde{H}^s(\mathbb{R}_+^2)$, if $s\geq 0$ then $u\in L^2(\mathbb{R}^2)$ and so $v\in L^2(\mathbb{R}_+^2)$.  Let $\tilde{v}$ be the zero extension of $v$ and $\phi\in \mathcal{D}(\mathbb{R}^2)$ be an arbitrary test function. Then, 
		$$\langle\tilde{v},\phi\rangle := \int_{\mathbb{R}^2}\tilde{v}\phi dx=\int_{\mathbb{R}_+^2}v\phi dx=\int_{\mathbb{R}_+^2}u\phi dx  = \int_{\mathbb{R}^2}u\phi dx=:\langle u,\phi\rangle$$ 
with the penultimate equality thanks to Lemma \ref{secinc}. 
Hence, $\tilde{v}$ and $u$ define the same distribution and so $\tilde{v}=u$.  
		
To prove the inclusion $E_s \subseteq \tilde{H}^s(\mathbb{R}_+^2)$, let $u\in E_s$. Then, by the definition \eqref{xs-app-def} of $E_s$, we have $u\in L^2(\mathbb{R}_+^2)$ and its zero extension $\tilde{u}\in H^s(\mathbb{R}^2)$.  Moreover, by the definition of the zero extension, $\text{supp}(\tilde{u}) \subseteq \overline{\mathbb{R}_+^2}$. Hence, $\tilde{u}\in H^s({\overline{\mathbb{R}_+^2}})$ and, in view of Lemma \ref{secinc}, it follows that $\tilde{u}\in \tilde{H}^s(\mathbb{R}_+^2)$.
\end{proof}

\begin{remark}[$H^s(\overline{\mathbb R_+^2})$ versus $H_0^s(\mathbb R_+^2)$]
For $s\geq 0$, Lemmas \ref{secinc} and \ref{hst-xs-l}  imply that $H^s(\overline{\mathbb R_+^2}) = E_s$. Hence, upon proving Theorem \ref{hst-h0s-t}, we will be able to deduce that $H^s(\overline{\mathbb R_+^2}) = H_0^s(\mathbb R_+^2)$ for all $s\geq 0$. 
\end{remark}

Our final observation before proceeding to the proof of  Theorem \ref{hst-h0s-t} is the following:
\begin{lemma}\label{hst-h0s-l}
For any $s\in\mathbb R$, we have the inclusion $\tilde{H}^s(\mathbb{R}_+^2)\subseteq H_0^s(\mathbb{R}_+^2)$,  in the sense that if $u\in \tilde{H}^s(\mathbb{R}_+^2)$ then the restriction $u|_{\mathbb{R}_+^2}\in H_0^s(\mathbb{R}_+^2)$.  
	 \end{lemma}
 \begin{proof}Let $u\in \tilde{H}^s(\mathbb{R}_+^2)$. Then, by the definition \eqref{hstil-def}, there is a sequence $u_n\in C_0^\infty(\mathbb{R}^2)$ with $\text{supp}(u_n)\subset \mathbb{R}_+^2$ such that $u_n\rightarrow u$ in  $H^s(\mathbb{R}^2)$.  Set $v_n=u_n|_{\mathbb{R}_+^2}$ and $v=u|_{\mathbb{R}_+^2}$. Then, $v_n\in \mathcal{D}(\mathbb{R}_+^2)$ and $v\in H^s(\mathbb{R}_+^2)$. Moreover,
	$$\|v_n-v\|_{H^s(\mathbb{R}_+^2)}\le \|u_n-u\|_{H^s(\mathbb{R}^2)}\longrightarrow 0, \ n\to \infty,$$ thus $v_n \to v$ in $H^s(\mathbb{R}_+^2)$ and so $v=u|_{\mathbb{R}_+^2}\in H_0^s(\mathbb{R}_+^2)$ by the definition \eqref{h0s-def}.
\end{proof} 

\noindent
\textit{Proof of Theorem \ref{hst-h0s-t}.} 
Lemmas  \ref{hst-xs-l} and \ref{hst-h0s-l} readily imply the inclusion $E_s=\tilde{H}^s(\mathbb{R}_+^2)\subseteq H_0^s(\mathbb{R}_+^2)$. Thus, it remains to prove the inclusion $H_0^s(\mathbb{R}_+^2) \subseteq E_s$. For this task, since $\mathcal D(\mathbb R_+^2)$ is dense in $H_0^s(\mathbb R_+^2)$ by definition of this space, without loss of generality we  assume $u\in \mathcal{D}(\mathbb{R}_+^2)$. 

By the definition of the zero extension,
$\no{\tilde{u}}_{L^2(\mathbb{R}^2)} = \no{u}_{L^2(\mathbb{R}_+^2)}$.
Moreover, if $0<s<\frac 12$, then
\begin{align}\label{longest}
\no{\tilde u}_{{H}^s(\mathbb{R}^2)}^2 
				&:=   \no{\tilde{u}}_{L^2(\mathbb{R}^2)}^2 + \iint\limits_{\mathbb{R}^2\times \mathbb{R}^2}\frac{|\tilde{u}(x)-\tilde{u}(y)|^2}{|x-y|^{2+2s}}dxdy
				\nn\\
				&=   \no{u}_{L^2(\mathbb{R}_+^2)}^2 + \iint\limits_{\mathbb{R}_+^2\times \mathbb{R}_+^2}\frac{|{u}(x)-{u}(y)|^2}{|x-y|^{2+2s}}dxdy
+ \iint\limits_{\mathbb{R}_+^2\times (\mathbb{R}^2\setminus\mathbb{R}_+^2)}\frac{|{u}(x)|^2}{|x-y|^{2+2s}}dxdy + \quad\iint\limits_{(\mathbb{R}^2\setminus\mathbb{R}_+^2)\times\mathbb{R}_+^2}\frac{|{u}(y)|^2}{|x-y|^{2+2s}}dxdy
				\nn\\
				& =: \no{u}_{H^s(\mathbb{R}_+^2)}^2+2 \int_{\mathbb{R}_+^2}|{u}(x)|^2f_s(x)dx
\end{align}
where 
$$
f_s(x) := \int_{\mathbb{R}^2\setminus\mathbb{R}_+^2}|x-y|^{-2(1+s)}dy.
$$
We claim that, for each $x=(x_1,x_2)\in \mathbb{R}_+^2$ and $\delta(x) := \min\{x_1,x_2\}$,  
\begin{equation}\label{ftheta}
f_s(x)\leq \frac \pi s \left(\delta(x)\right)^{-2s}.
\end{equation}
Indeed, for each $x\in\mathbb R_+^2$, the complement $B^c_{\delta(x)}(x)$ of the ball $B_{\delta(x)}(x)$ centered at $x$ and of radius $\delta(x)$ contains the set $\mathbb R^2 \setminus \mathbb R_+^2$.  Thus,  
$$ f_s(x)\le \int_{B^c_{\delta(x)}(x)}|x-y|^{-2(1+s)}dy=\int_0^{2\pi}\int_{\delta(x)}^\infty r^{-2s-1}drd\varphi=\frac{\pi}{s}\left(\delta(x)\right)^{-2s},
\quad
x\in \mathbb{R}_+^2.$$
Combining \eqref{longest} with \eqref{ftheta}, we obtain
\begin{equation}\label{prel}
\no{\tilde u}_{\tilde{H}^s(\mathbb{R}_+^2)}^2
\leq
\no{u}_{H^s(\mathbb{R}_+^2)}^2+\frac{2\pi}{s}  \int_{\mathbb{R}_+^2}  \left(\delta(x)\right)^{-2s} |{u}(x)|^2 dx.
\end{equation}
The following lemma allows us to estimate the integral remaining on the right-hand side.
\begin{lemma}\label{hardy-ineq-qp-l}
If $0<s<\frac 12$ and $u\in\mathcal D(\mathbb{R}_+^2)$, then		\begin{equation}\label{fthetaest}
			\int_{\mathbb{R}_+^2}\left(\delta(x)\right)^{-2s}|u(x)|^2dx\lesssim \|u\|_{H^s(\mathbb{R}_+^2)}^2.
\end{equation} 
\end{lemma}

We prove Lemma \ref{hardy-ineq-qp-l} after the end of the current proof. Back to inequality \eqref{prel}, employing estimate \eqref{fthetaest} we deduce
\begin{equation}
\no{\tilde u}_{\tilde{H}^s(\mathbb{R}_+^2)}^2
\lesssim
\no{u}_{H^s(\mathbb{R}_+^2)}^2+\frac{2\pi}{s}  \no{u}_{H^s(\mathbb{R}_+^2)}^2
\lesssim
 \no{u}_{H^s(\mathbb{R}_+^2)}^2
\end{equation}
which shows that $u \in E_s$ and hence establishes the desired inclusion $H_0^s(\mathbb R_+^2) \subseteq E_s$. The proof of Theorem \ref{hst-h0s-t} is complete. \hfill $\square$

The only  task remaining is to prove Lemma \ref{hardy-ineq-qp-l}, which actually provides the quarter-plane version of a well-known result known as Hardy's inequality that reads as follows:
\begin{lemma}[Lemma 3.31 in \cite{mclean2000} (see also Theorem 1.76 in \cite{l2023} and Theorem 3.4.2 in \cite{a2015})]\label{hardy-ineq-l}\
\\
If $0<s<\frac 12$ and $\phi\in \mathcal{D}(\mathbb R)$, then
\begin{equation}
\int_0^\infty \tau^{-2s}|\phi(\tau)|^2d\tau
\leq 
c_s \int_{0}^{\infty}\int_{0}^{\infty}\frac{|\phi(\tau_1)-\phi(\tau_2)|^2}{|\tau_1-\tau_2|^{1+2s}}\, d\tau_1 d\tau_2.
\end{equation}
\end{lemma}
	
In what follows, we will combine Lemma \ref{hardy-ineq-l} with the change of variables \eqref{changeofvar} in order to prove 	Lemma \ref{hardy-ineq-qp-l}.
\begin{proof}[Proof of Lemma \ref{hardy-ineq-qp-l}]
Making the change of variables \eqref{changeofvar}, which rotates the quarter-plane $\mathbb R_+^2$ by an angle of $\frac{5\pi}{4}$  in the anticlockwise direction, we have
\begin{equation}\label{fthetaest2}
\int_{\mathbb{R}_+^2}\left(\delta(x)\right)^{-2s}|v(x)|^2dx
=
\int_{\Omega}\left(d(\xi)\right)^{-2s}|w(\xi)|^2d\xi
\end{equation}  
where $\Omega = \left\{\xi_2<h(\xi_1):=-|\xi_1|\right\} \subset \mathbb{R}^2$ is the image of  $\mathbb R_+^2$ under \eqref{changeofvar} and 
$$
d(\xi):=\frac{1}{\sqrt{2}}\min\left\{-\left(\xi_1+\xi_2\right), \xi_1-\xi_2\right\}>0,
\quad
w(\xi):=v(-\frac{1}{\sqrt{2}}(\xi_1+\xi_2), -\frac{1}{\sqrt{2}}(\xi_2-\xi_1)) \in \mathcal D(\Omega).
$$
Note that $h$ is a Lipschitz function with Lipschitz constant equal to $1$, since by the triangle inequality		
$$
|h(\xi_1)-h(\xi_2)|=\left|-|\xi_1|+|\xi_2|\right|\le |\xi_1-\xi_2|.
$$ 
Further,  for any $\eta = (\eta_1, \eta_2) \in \partial\Omega$ (so that $\eta_2 = h(\eta_1) = -|\eta_1|$), 
\begin{equation*}
\begin{aligned}
|h(\xi_1)-\xi_2|&=|\eta_2-\xi_2+h(\xi_1)-h(\eta_1)|\le |\eta_2-\xi_2|+|h(\xi_1)-h(\eta_1)| = |\eta_2-\xi_2|+\big|-|\xi_1| + |\eta_1|\big|
\\
&\le |\eta_2-\xi_2|+|\eta_1-\xi_1| \leq \sqrt 2 \left(|\eta_2-\xi_2|^2+|\eta_1-\xi_1|^2\right)^{\frac 12} = \sqrt{2} \left|\eta-\xi\right|,\quad  \xi\in \Omega.
\end{aligned}
\end{equation*} 
Choosing $\eta$ to be the projection of $\xi$ on $\xi_2 = -|\xi_1|$, i.e. on $\xi_1=\xi_2$ (for $\xi_1<0$) and $\xi_1=-\xi_2$ (for $\xi_1>0$), it follows that $|h(\xi_1)-\xi_2|\leq \sqrt{2} \, d(\xi)$ for $\xi\in \Omega$. Therefore, 	
\begin{equation}\label{a18-0}
\int_{\Omega}\left(d(\xi)\right)^{-2s}|w(\xi)|^2d\xi\le 2^s\int_{\xi_2<h(\xi_1)}|h(\xi_1)-\xi_2|^{-2s}|w(\xi)|^2d\xi=:I.
\end{equation}
Changing variable via $\tau=h(\xi_1)-\xi_2$, we get
\begin{equation}\label{a18}
I=2^s\int_{\xi_1\in \mathbb{R}}\int_{\tau=0}^\infty \tau^{-2s}|w(\xi_1,h(\xi_1)-\tau)|^2d\tau d\xi_1.
\end{equation}
Hence, invoking the one-dimensional Hardy's inequality of Lemma \ref{hardy-ineq-l} for estimating the $\tau$ integral, we have
\begin{equation}\label{a18-1}
I\lesssim \int_{\xi_1\in\mathbb{R}}\int_{\xi_2<h(\xi_1)}\int_{y<h(\xi_1)} \frac{|w(\xi_1,y)-w(\xi_1,\xi_2)|^2}{|y-\xi_2|^{1+2s}}dyd\xi_2d\xi_1=:II.
\end{equation}

Recalling that $w \in \mathcal D(\Omega)$, let $W\in \mathcal{D}(\mathbb{R}^2)$ be such that $W|_{\Omega}=w$. Then, we have 
		\begin{equation*}
			\begin{split}
				&II=\int_{\xi_1\in\mathbb{R}}\int_{\xi_2\in \mathbb{R}}\int_{y\in \mathbb{R}}\frac{|W(\xi_1,y)-W(\xi_1,\xi_2)|^2}{|y-\xi_2|^{1+2s}}dyd\xi_2d\xi_1\\
				&\overset{y= \xi_2+h}{=}\int_{h\in \mathbb{R}}\int_{\xi_1\in\mathbb{R}}\int_{\xi_2\in \mathbb{R}}\frac{|W(\xi_1,\xi_2+h)-W(\xi_1,\xi_2)|^2}{|h|^{1+2s}}d\xi_2d\xi_1dh
\\
&=\int_{h\in \mathbb{R}}\frac{\|W(\cdot,\cdot+h)-W(\cdot,\cdot)\|_{L^2(\mathbb{R}^2)}^2}{|h|^{1+2s}}dh.
			\end{split}
		\end{equation*}
		Proceeding like in the proof of  the time estimate for the Schr\"odinger initial value problem (see Theorem 3.2 in~\cite{hm2020}), from Plancherel's theorem we have 
\begin{equation*}
\begin{split}
\no{W(\cdot,\cdot+h)-W(\cdot,\cdot)}_{L^2(\mathbb{R}^2)}^2
&=
\frac{1}{(2\pi)^2}
\no{\mathcal{F}\left\{W(\cdot,\cdot+h)-W(\cdot,\cdot)\right\}}_{L^2(\mathbb{R}^2)}^2
\\
&=
\frac{1}{(2\pi)^2}
\int_{\mathbb{R}^2}|\mathcal F\{W\}(\tau)|^2 \left|e^{i\tau_2h}-1\right|^2 d\tau.
\end{split}
\end{equation*}
It then follows that
\begin{equation}\label{a21-0}
\begin{aligned}
II&=
\frac{1}{(2\pi)^2}
\int_{\mathbb{R}^2}|\mathcal F\{W\}(\tau)|^2\int_{h\in \mathbb{R}} \frac{|e^{i\tau_2h}-1|^2}{|h|^{1+2s}}dhd\tau
\\
&=
c_s \int_{\mathbb{R}^2}|\tau_2|^{2s}|\mathcal F\{W\}(\tau)|^2d\tau
\leq 
c_s \|W\|_{H^s(\mathbb{R}^2)}^2
\end{aligned}
\end{equation} 
where $$c_s:= \frac{1}{(2\pi)^2} \int_{h\in \mathbb{R}} \frac{|e^{ih'}-1|^2}{|h'|^{1+2s}}dh'<\infty.$$

Estimate \eqref{a21-0} has been proved for any $W\in \mathcal D(\mathbb R^2)$. Since this space is dense in $H^s(\mathbb R^2)$, for any $W \in H^s(\mathbb R^2)$ such that $W|_\Omega = w$ we can find a sequence $W_n \in \mathcal D(\mathbb R^2)$ such that $W_n|_\Omega = w$ and $W_n \to W$ in $H^s(\mathbb R^2)$. Thus, employing~\eqref{a21-0} for $W_n$ in place of $W$ and then taking $n\to\infty$, we obtain
\begin{equation}\label{Ww-0}
II \leq c_s \no{W}_{H^s(\mathbb R^2)}^2
\end{equation}
for any $W \in H^s(\mathbb R^2)$ such that $W|_{\Omega} = w$. Thus, since 
$
\no{w}_{H^s(\Omega)} := \inf\big\{\no{W}_{H^s(\mathbb R^2)}: W|_\Omega = w\big\},
$
taking infimum on both sides of inequality \eqref{Ww-0} we deduce
\begin{equation}\label{Ww}
II \leq c_s \no{w}_{H^s(\Omega)}^2 = c_s \no{v}_{H^s(\mathbb{R}_+^2)}^2.
\end{equation}
Combining \eqref{fthetaest2}, \eqref{a18-0}, \eqref{a18}, \eqref{a18-1} and \eqref{Ww} leads to the desired result of Lemma~\ref{hardy-ineq-qp-l}.
\end{proof}

 \bibliographystyle{myamsalpha}
 \bibliography{references.bib}

\end{document}